\documentclass[12pt]{amsart}
\usepackage{physics, amsmath,amssymb,hyperref}
\usepackage{xcolor}

\allowdisplaybreaks

\newtheorem{theo}{Theorem}[section]
\newtheorem{cor}[theo]{Corollary}
\newtheorem{prop}[theo]{Proposition}
\newtheorem{lemma}[theo]{Lemma}

\theoremstyle{remark}
\newtheorem{remark}[theo]{Remark}

\newtheorem{problem}[theo]{Problem}
\theoremstyle{definition}

\newcommand{\Sym}{\operatorname{Sym}}
\newcommand{\Hom}{\operatorname{Hom}}
\newcommand{\Pic}{\operatorname{Pic}}
\newcommand{\lqko}{L_{k,q}}

\begin{document}

\title[Waring's problem over function fields]{The asymptotic in Waring's problem over function fields via a singular locus in the circle method}
\author{Will Sawin}

\maketitle

\begin{abstract} We give results on the asymptotic in Waring's problem over function fields that are stronger than the results obtained over the integers using the main conjecture in Vinogradov's mean value theorem. Similar estimates apply to Manin's conjecture for Fermat hypersurfaces over function fields. Following an idea of Pugin, rather than applying analytic methods to estimate the minor arcs, we treat them as complete exponential sums over finite fields and apply results of Katz, which bound the sum in terms of the dimension of a certain singular locus, which we estimate by tangent space calculations.\end{abstract}

\section{Introduction}

The asymptotic in Waring's problem is the formula 
\[ \# \{ \mathbf a\in  \mathbb N ^s \mid \sum_{i=1}^s a_i^k = n \} = (1+o(1))  n^{s/k-1}   \frac{ \Gamma(1 +1/k)^s}{\Gamma(s/k)} \prod_p  \ell_p(n) \] where \[ \ell_p(n) = \lim_{r\to\infty}  \frac{ \# \{\mathbf b \in ( \mathbb Z/p^{r} )^s \mid \sum_{i=1}^s b_i^k \equiv n \bmod p^r\}}{ p^{r(s-1)}}, \]
 the boldface variables $\mathbf a$ represent vectors $a_1,\dots, a_s$, and the term  $\frac{ \Gamma(1 +1/k)^s}{\Gamma(s/k)} $ is the local factor at $\infty$. Here $o(1)$ should go to $0$ as $n$ goes to $\infty$ with fixed $s,k$.

This was obtained by Hardy and Littlewood for $s$ sufficiently large with respect to $k$. Further research on the asymptotic has focused on proving it for $s$ as small as possible in terms of $k$, with work of Vaughan and Wooley, culminating in the result of Bourgain~\cite[Theorem 11]{Bourgain2017} proving the asymptotic for $s \geq k^2 + k-\max_{t \leq k} \lceil t \frac{ k-t-1}{k-t+1} \rceil $, using the breakthrough work of Bourgain, Demeter, and Guth~\cite{BDG} proving the main conjecture of Vinogradov's mean value theorem. (More precisely, \cite[Theorem 11]{Bourgain2017} only gives the best lower bound for $k >12$. For $4 \leq k \leq 12$, the best bound is given by a different formula and obtained from a conditional result of Wooley~\cite[Theorem 4.1]{Wooley2012} rendered unconditional by~\cite{BDG}. For $k=3$ the result is known for $s \geq 8$ due to earlier work of Vaughan~\cite{Vaughan1986}.)

We now describe an analogous problem for polynomials over finite fields. Let $\mathbb F_q$ be a finite field of size $q$ and characteristic $p$ and let $\mathbb F_q[T]$ be the ring of polynomials in one variable over $\mathbb F_q$. A \emph{prime} in $\mathbb F_q[T]$ is a monic irreducible polynomial.  In this case, there is no good analogue of the positivity of the $a_i$ that ensures the equation $\sum_{i=1}^s a_i^k = n$ has finitely many solutions. Instead, it is most natural to restrict attention to  polynomials $a_i$ of degree at most $e$. This gives the following problem.

\begin{problem} Fix a finite field $\mathbb F_q$ and a positive integer $k$. Obtain, for $s$ as small as possible, for all natural numbers $e$ and polynomials $f \in \mathbb F_q[T]$ of degree $\leq ke$, the estimate
\begin{equation}\label{poly-asymptotic}  \#\{\mathbf a \in \mathbb F_q[T]^s \mid \deg a_i \leq e, \sum_{i=1}^s a_i^k = f\} = (1+o(1)) q^{ e(s-k) +s-1  }  \ell_\infty (f)  \prod_{\substack{\pi \in \mathbb F_q[T] \\  \textrm{prime} }} \ell_\pi(f) \end{equation} where
 \[ \ell_\infty (f) =  \lim_{r\to \infty}  \frac{ \#\{ \mathbf b \in( \mathbb F_q[u]/u^r)^s \mid \sum_{i=1}^s b_i^k \equiv  u^{ke}  f(u^{-1}) \bmod u^r \}}{ q^{r (s-1) } }\]
 \[ \ell_\pi(f) =  \lim_{r\to\infty}  \frac{ \# \{\mathbf b  \in( \mathbb F_q[T]/\pi^r )^s \mid \sum_{i=1}^s b_i^k \equiv f \bmod \pi^r\}}{ q^{r (s-1) \deg \pi} } \]
 and $o(1)$ goes to $0$ as $e$ goes to $\infty$ for fixed $q,s,k$. \end{problem}
Here the term $ \ell_\infty$ is the local factor at $\infty$ and $\ell_\pi$ is the local factor at the prime $\pi$. In comparing these bounds, note that the expression $q^{\deg \pi}$ is comparable to $p$  and $q^e$ is comparable to $n^{1/k}$ so $q^{e(s-k)}$ is comparable to $n^{s/k-1}$. The $q^{s-1}$ factor could be folded into the local factor at $\infty$ if desired.

Prior work~\cite{Liu2010,Yamagishi2016} has often considered the case where $e$ is ``as small as possible'', i.e. $e = \lceil \frac{\deg f}{k} \rceil$ or $e=\frac{\deg f}{k}+1$ in certain exceptional cases. However, here we do not make that restriction. Removing the restriction matters little, except for the fact that it allows a direct application to Manin's conjecture, because shrinking $e$ tends to make the problem more difficult.

Yamagishi~\cite[Theorem 1.4]{Yamagishi2016} obtained \eqref{poly-asymptotic} in the case $k<p$ when $s\geq 2k^2-2 \lfloor (\log k)/\log(2) \rfloor$, as well as $s\geq 86$ when $k=7$ and $s\geq 2k^2-11$ when $k\geq 8$, generalizing earlier work of Kubota~\cite{R1974}. Yamagishi~\cite[Theorem 1.3]{Yamagishi2016} also obtained more complex bounds in the $k>p$ case. These bounds were obtained before \cite{BDG} and could likely be improved using the function field analogue~\cite[Corollary 17.3]{Wooley2018} but this does not seem to have been done in the literature.

We obtain \eqref{poly-asymptotic} with an improved (linear) dependence of $s$ on $k$, at the cost of assuming that both $q$ and $p$ are somewhat large.  (Note that \cite[Theorem 1.3]{Yamagishi2016} gives linear dependence of $s$ on $k$ in the special case where $k$ is a power of $p$ plus $1$, by taking advantage of the exceptional behavior of $p$th powers in characteristic $p$ which makes the setting fundamentally different from the $p>k$ one.)

To describe the conditions under which we can obtain \eqref{poly-asymptotic}, we first need to define a convex set in the plane: Let $\Delta_{k,p}$ be the convex hull of the set of points $(i,j) \in \mathbb N^2$ such that $\gcd(i,j,p)=1$ but $p \mid  (k-1) i - j$.  (Throughout this paper, the natural numbers $\mathbb N$ include $0$.) For $\gamma \in \mathbb R$, let $\gamma \Delta_{k,p}$ be the dilation $\{ \gamma v \mid v\in \Delta_{k,p}\}$ of $\Delta_{k,p}$ by $\gamma$. Let $\gamma_{k,p}$ be the supremum of the set of $\gamma \in (0,\infty)$ such that $ ( 1, \frac{k-2}{2} ) \in  \gamma \Delta_{k,p}$, or $0$ if no such $\gamma$ exists. We will see in Lemmas \ref{gamma-lower-bound} and \ref{gamma-upper-bound} that  for $p>k-1$ we have $0 \leq \gamma_{k,p}<1$ and more precisely 
\begin{equation}\label{gamma-bounds}  \frac{k-2}{2k-2} \leq \gamma_{k,p}\leq   \frac{k-2}{2k-2}\left( 1+ \frac{k}{p}\right).\end{equation}

\begin{theo}\label{intro-poly} Fix a finite field $\mathbb F_q$ of characteristic $p$ and positive integers $k$ and $ s$ such that $2\leq k<p$. Let $\lqko= \frac{\log (k+1)}{\log q}$.  For all natural numbers $e$ and polynomials $f \in \mathbb F_q[T]$ of degree $\leq ke$, we have
 \eqref{poly-asymptotic} as soon as \[ q > ( k+1)^{2 (k-1)} \] and
 \[ s > \max\left(  \frac{2 (k- \gamma_{k,p} -2 \lqko)}{ 1-\gamma_{k,p}  -2\lqko } ,  \frac{2k}{ 1- 2 (k-1)\lqko } \right) .\]  \end{theo}
 
 For $q$ sufficiently large, it suffices to have (taking $\lim_{q\to\infty}$ in the lower bound for $s$)
  \[ s > \max\left(  \frac{2 (k- \gamma_{k,p} )}{ 1-\gamma_{k,p}   } ,  2k  \right) = \frac{ 2 (k-\gamma_{k,p} ) }{ 1- \gamma_{k,p}} \] since $\gamma_{k,p}\in [0,1)$ so $\frac{k-\gamma_{k,p}} {1-\gamma_{k,p}} \geq k$.

 By \eqref{gamma-bounds} we obtain \eqref{poly-asymptotic} when $q$ is sufficiently large (depending on $s,k$) and 
\[ s >  \frac{ 2 ( k (2k-2) p - (k-2) (k+p) ) }{ (2k-2) p - (k-2) (k+p ) } = \frac{  2(2k^2-3k+2) p -2k(k-2) }{ k (p+2-k) } \]
which for $p$ sufficiently large (depending on $k$) simplifies to 
$ s > 4 k -6 + \frac{4}{k}$ which since $s$ is an integer is equivalent to $s\geq 8$ for $k=3$,  $s\geq 12$ for $k=4$, and $s \geq 4k-5 $ for all higher $k$.
For $k>3$ this improves on the best known bounds in the integer setting, e.g. for $k=4$ \cite[Theorem 4.1]{Wooley2012} requires $s \geq 15$ but we require $s \geq 12$, while for $k=5$ \cite[Theorem 4.1]{Wooley2012} requires $s \geq 23$  but we require $s \geq 15$.  

In fact, our result gives a power saving, but because the exact power saving has a complicated dependence on $s,k,q,p$ we defer its statement to Theorem \ref{main-simplified}.

The local factors $\ell_\pi(f)$ and $\ell_\infty(f)$ are always nonzero in the range of parameters we consider (which is checked in Lemma \ref{lv-lower-bound}) so that Theorem \ref{intro-poly} implies that, for $e$ sufficiently large, for all $f$ of degree $\leq ke$, there exists $a_1,\dots, a_s$ with $\deg a_i\leq e$ such that $\sum_{i=1}^s a_i^k = f$. Finding the minimum $s$, in terms of $k$ and $q$ for which this holds, is known as the \emph{restricted} Waring problem, and the minimum is often denoted $G_q(k)$. The strongest known bounds on the analogous quantity $G(k)$ over the integers are due to Wooley~\cite{Wooley1992} and are of the form $k\log k + k \log \log k +O(k)$, building on prior work of many mathematicians, in particular Vinogradov~\cite{Vinogradov1934} who was the first to obtain bounds of shape $O(k\log k)$. Similar bounds were proven over $\mathbb F_q[T]$ by Liu and Wooley~\cite{LiuWooley}. Since Theorem \ref{intro-poly} gives bounds that are linear in $k$, it improves on these bounds also when $q$ and $p$ are sufficiently large.

We also have an analogue for a general function field. Let $C$ be a smooth projective curve of genus $g$ over $\mathbb F_q$, where we always assume that curves are geometrically irreducible. Denote the set of closed points of $C$, or, equivalently, places of the function field $\mathbb F_q(C)$, by $\abs{C}$. For $v$ a closed point of $C$, let $\mathcal O_{C_v}$ be the corresponding local ring with uniformizer $\pi_v$.   

\begin{problem} Fix a finite field $\mathbb F_q$ and a positive integer $k$. Obtain, for $s$ as small as possible, for all natural numbers $e$, line bundles $L$ on $C$ of degree $e$, and sections $f \in H^0(C, L^{k})$, the estimate
\begin{equation}\label{curve-asymptotic} \#\{ \mathbf a \in H^0( C, L )^s \mid  \sum_{i=1}^s a_i^k = f\} \\ = (1+o(1)) q^{ e(s-k) +(s-1)(1-g)  } \prod_{\substack{v \in \abs{C} }} \ell_v(f) \end{equation} where \begin{equation}\label{curve-asymptotic-local-factor} \ell_v(f) =  \lim_{r\to\infty}  \frac{ \# \{ \mathbf b \in ( \mathcal O_{C_v} /\pi_v^r )^s \mid \sum_{i=1}^s b_i^k \equiv f \bmod \pi^r\}}{ q^{r (s-1) \deg v} }\end{equation} and $f$ is interpreted as an element of $\mathcal O_{C_v}$ by dividing by the $k$th power of an arbitrary local generator of $L$, and $o(1)$ goes to $0$ as $e$ goes to $\infty$ for fixed $q,s,k,C$. \end{problem}
The inclusion of the line bundle $L$ is necessary because on a curve of positive genus there does not exist a single analogue to the notion of a polynomial of degree $e$, rather, each line bundle of degree $e$ gives its own analogue. However, working locally, we may ignore this subtlety, which is why the formulas for the local factors look similar to the previous case (except that we use a local ring $\mathcal O_{C_v}$ instead of the global ring $\mathbb F_q[T]$ as there is not a convenient global ring to take). An intrinsic way to express the local factor would be a sum over tuples  $ \mathbf b \in (L \otimes_{\mathcal O_{C}} \mathcal O_{C_v}/\pi_v^r)^s$. However, $L$ is a locally free $\mathcal O_C$-module, hence $L \otimes_{\mathcal O_C} \mathcal O_{C_v}$ is a free $ \mathcal O_{C_v} $-module with an isomorphism to $\mathcal O_{C_v}$ given by dividing by fixed local generator of $L$ at $v$. This induces an isomorphism $L \otimes_{\mathcal O_{C}} \mathcal O_{C_v}/\pi_v^r \to \mathcal O_{C_v}/\pi_v^r$, and applying this isomorphism gives \eqref{curve-asymptotic-local-factor}. This requires a choice of local generator of $L$ at $v$, but this local generator is unique up to multiplication by a unit in $\mathcal O_{C_v}$, so dividing $f$ by the $k$th power of the local generator gives a unique answer up to multiplication by the $k$th power of a unit, which does not change the count  $ \# \{\mathbf b \in  (\mathcal O_{C_v} /\pi_v^r)^s \mid \sum_{i=1}^s b_i^k \equiv f \bmod \pi^r\}$. Hence \eqref{curve-asymptotic-local-factor} is well-defined.

\begin{theo}\label{intro-curve} Fix a finite field $\mathbb F_q$ of characteristic $p$, a smooth projective curve $C$ over $\mathbb F_q$, and positive integers $k$ and $ s$ such that $2\leq k<p$. Let $\lqko= \frac{\log (k+1)}{\log q}$. For all  natural numbers $e$, line bundles $L$ of degree $e$ over $C$, and  $f \in H^0(C, L^k)$, we have
 \eqref{curve-asymptotic} as soon as \[ q > ( k+1)^{2 (k-1)} \] and
 \[ s > \max\left(  \frac{2 (k- \gamma_{k,p} -2 \lqko)}{ 1-\gamma_{k,p}  -2\lqko } ,  \frac{2k}{ 1- 2 (k-1)\lqko } \right) .\]  \end{theo}
 
Since the conditions for an asymptotic are identical to the case of Theorem \ref{intro-poly}, the same considerations apply as to when the asymptotic holds for $q$ or $p$ sufficiently large. Note that the required conditions on $s$ and $q$ are the same for all curves $C$. 

Theorem \ref{intro-poly} is obtained from Theorem \ref{intro-curve} by setting $C = \mathbb P^1$. We will check that the formulas of Theorem \ref{intro-curve} specialize to the formulas of Theorem \ref{intro-poly} in this case at the end of Section \ref{s-minor} when we prove these theorems.

The application to Manin's conjecture for Fermat hypersurfaces is described by the following theorem.

\begin{theo}\label{intro-manin} Fix a finite field $\mathbb F_q$ of characteristic $p$ and positive integers $n$ and $d$ such that $2\leq d<p$. Let $X$ be the hypersurface in $\mathbb P^n_{\mathbb F_q}$ defined by the equation $\sum_{i=0}^n x_i^d=0$. Let $C$ be a smooth projective curve of genus $g$ over $\mathbb F_q$. Assume that
\[ q > ( d+1)^{2 (d-1)} \] and
 \[ n+1 > \max\left(  \frac{2 (d- \gamma_{d,p} -2 \frac{\log (d+1)}{\log q})}{ 1-\gamma_{d,p}  -2\frac{\log (d+1)}{\log q} } ,  \frac{2d}{ 1- 2 (d-1)\frac{\log (d+1)}{\log q} } \right) .\] 
 
 Let 
 \[ \tau_C(X) =  \prod_{v \in \abs{C} }  \Bigl ( (1 - q^{ -\deg v})\frac{ \# X(\mathbb F_{q^{\deg v} })}{ q^{(n-1) \deg v}} \Bigr).\]
 
 Then
 \[ \# \{f \colon C\to X \mid \textrm{degree } e \} = (1+o(1))   \tau_C(X) \frac{q^{e (n+1-d) + n(1-g) } \#{\operatorname{Pic}^0(C)}  }{q-1}.\]
 \end{theo}

\subsection{Outline of proof} Our method of proof begins with the standard approach via the circle method, obtaining a sum of exponential sums, and dividing these into major arcs and minor arcs. For the major arcs, our treatment (in \S\ref{s-major}) is similar to prior work in the circle method. For the minor arcs, we take a different approach. Instead of Weyl differencing, efficient congruencing, or any other analytic method, we treat these sums geometrically. The sums are over $s$-tuples of polynomials of degree $\leq e$ in $\mathbb F_q[T]$. One can view the set of polynomials of degree $\leq e$ as an $\mathbb F_q[T]$-analogue of an interval around $0$, and therefore the sums as $\mathbb F_q[T]$-analogues of short exponential sums. However, this set of polynomials is also a vector space of dimension $e+1$ over $\mathbb F_q$, so the sums are over a $s(e+1)$-dimensional vector space over $\mathbb F_q$, and can also be viewed as complete exponential sums associated to a polynomial function on this vector space. We then apply (in Lemma \ref{minor-katz}) a classical estimate of Katz for complete exponential sums in high dimension in terms of the dimension of the singular locus of the vanishing set of this polynomial (more precisely, the vanishing set of its leading term, but this polynomial is homogeneous and thus equal to its leading term). The lower the dimension of the singular locus, the stronger this estimate is.

An identical approach up to this point was taken by Pugin~\cite{Pugin} in the closely related problem of counting $\mathbb F_q(T)$-points on the Fermat hypersurface. However, the step where the dimension is calculated \cite[Lemma 2.5.1]{Pugin} contains a gap: It is claimed that the dimension of the inverse image under the diagonal morphism $\mathbb A^{r+1} \to (\mathbb A^{r+1})^{d-1}$ of a closed set is at most $\frac{1}{d-1}$ times the dimension of the closed set, which is not true in general, for example if the closed set is itself the diagonal. (This is true if the closed set is a product of $d-1$ closed subsets of $\mathbb A^{r+1}$ but that condition does not hold here.)

Instead, in Lemma \ref{exists-c}, we describe the singular locus explicitly as a set of polynomials $a$ whose $(k-1)$st power is congruent modulo a certain modulus (associated to a closed set $Z$)  to a constant times a low-degree polynomial $c$. To bound the dimension of this set, we first stratify by the joint root multiplicities of $a$ and $c$.  In Lemma \ref{strata-dim}, we estimate the dimension of the intersection of the singular locus with any stratum by bounding the dimension of its tangent space at an arbitrary point. The reason this leads to a good estimate is that, when describing the tangent space to the solution set of any equation, we consider solutions to the derivative of the equation. We may as well take the logarithmic derivative, and doing this simplifies the equation greatly: The $(k-1)$st power becomes a multiplication by $k-1$. Bounding the dimension of the space of solutions to this simplified equation is considerably easier. The greatest difficulty occurs in certain cases where the polynomial has repeated roots, which the stratification assists us in handling.

The method does not currently apply to Manin's conjecture for hypersurfaces other than the Fermat hypersurface or to representability by general homogeneous polynomials. In these settings, we can describe the singular locus explicitly, but it is not clear how to bound its dimension. In the final section, we give an approach to bound the dimension of the singular locus in terms of dimensions of moduli spaces of curves on a certain blowup of projective space.

Let $X$ be a smooth hypersurface in $\mathbb P^n$ defined by a polynomial $F$ of degree $d$ coprime to the characteristic of the base field. Let $\nabla F$ be the tuple of polynomials $\frac{\partial F}{\partial x_0},\dots, \frac{\partial F}{\partial x_n}$. Since $X$ is smooth and the degree $d$ is coprime to the characteristic, the tuple of polynomials $\nabla F$ has no common zeroes except $0$. Since $\nabla F$ is an $(n+1)$-tuple of polynomials in $n+1$ variables with no common zeroes except $0$, $\nabla F$ defines a map $\mathbb P^n \to \mathbb P^n$. Let $Y$ be the blowup of $\mathbb P^n \times \mathbb P^n$ along the graph of this map. Let $E$ be the exceptional divisor of this blowup. Since the graph of $\nabla F$ is a smooth codimension $n$ subset isomorphic to $\mathbb P^n$, $E$ is isomorphic to a $\mathbb P^{n-1}$-bundle over $\mathbb P^n$.

We give in Proposition \ref{general-dim-bound} below a bound for the dimension of the relevant singular locus (which we describe in more detail in \S\ref{s-arbitrary}) in terms of the dimensions of moduli spaces of maps $C \to Y$ whose image is not entirely contained in the exceptional divisor $E$.  This gives a case where  upper bounds on the dimension of the moduli space of maps to one variety (equivalently, upper bounds on the number of $\mathbb F_q(C)$-rational points of the variety that are uniform in $q$) can give precise asymptotics for the number of $\mathbb F_q(C)$ points on a different variety. Such results were known before (for instance, the approach to Waring's problem via the Vinogradov mean value theorem exactly involves turning upper bounds into asymptotics) but these could only apply to hypersurfaces of a special form like the Fermat hypersurface, while this approach potentially applies to an arbitrary hypersurface. However, since it is not clear how to understand the dimension of the moduli space of curves on this blowup, the approach is only speculative for now.

\subsection{Acknowledgments} The author was supported by NSF grants DMS-2101491 and DMS-2502029 and served as a Sloan Research Fellow. He would like to thank Tim Browning, Matthew Hase-Liu, Johan de Jong, Eric Riedl, Jason Starr, and Sho Tanimoto for helpful conversations and comments on earlier versions of this manuscript, as well as the anonymous referee for many helpful comments.

\section{Preliminaries}

\subsection{Algebraic geometry}

A \emph{variety} is a separated, geometrically integral scheme of finite type over a field. A \emph{curve} is a variety of dimension one. In particular, curves are always geometrically irreducible.

We will be interested in the finite closed subschemes of a curve $C$. The finite closed subschemes of $C$ are naturally in bijection with the effective divisors on $C$ with a divisor $m[v]$ for $v$ a point of $C$ corresponding to the unique closed subscheme of length $m$ supported at $v$ and sum of disjoint divisors corresponding to disjoint union of subschemes. Because of this, we will use the same notation for subschemes and their divisors, i.e. we will use $m[v]$ to refer to the length $m$ subscheme supported at $v$ and $\mathcal O_C(Z)$ to refer to the twist of $\mathcal O_C$ by the divisor of the subscheme $Z$.

We state some classical results about curves in a form convenient to us:

\begin{lemma}\label{clifford} Let $C$ be a smooth projective curve of genus $g$ and let $L$ be a line bundle on $C$ of degree $d$. Then
\begin{enumerate}
\item \[ \dim H^0(C, L ) \leq \max ( d+1-g, \frac{d}{2}+1, 0) .\]

\item If $ d\in [-2, 2g]$ then \[ \dim H^0(C,L) \leq \frac{d}{2}+1.\]

\item If $d \in [-2,2g]$ then \[ \dim H^0(C,L) + \dim H^0(C,K_C \otimes L^{-1} ) \leq g+1.\]

\end{enumerate}\end{lemma}

\begin{proof} All of these are variants of Clifford's theorem. 

In the form~\cite[\S8.6, Corollary 4]{Fulton2008},  Clifford's theorem states that $\dim H^0(C,L)$ is at most $\frac{d}{2}+1$ as long as $\dim H^0(C,L) >0$ and $\dim H^0(C, K_C \otimes L^{-1})>0$. In the remaining cases, we have either $\dim H^0(C,L)=0$ or $\dim  H^0(C, K_C \otimes L^{-1})=0$. If $\dim  H^0(C, K_C \otimes L^{-1})=0$ then we have $\dim H^0(C,L)= d+1-g$ by the Riemann-Roch theorem~\cite[\S8.6]{Fulton2008}. Hence in any case we have the upper bound (2).

If $d \geq -2$ then $\frac{d}{2}+1\geq 0$ and if $d \leq 2g$ then $\frac{d}{2}+1 \geq d+1-g$ so if $d \in [-2,2g]$ then $ \max ( d+1-g, \frac{d}{2}+1, 0) = \frac{d}{2}+1$ and we have (2).

For (3), applying then Riemann-Roch theorem~\cite[\S8.6]{Fulton2008} and then using part (2), we have \[  \dim H^0(C,L) +\dim H^0(C,K_C \otimes L^{-1} )  = 2\dim H^0(C,L)  + g-1-d \leq  d+2 +g-1-d=g+1 . \qedhere\] \end{proof}

\begin{lemma}\label{rh-bound} Let $C$ be a smooth projective curve of genus $g$ over $\mathbb F_q$. Let \[ \zeta_C(u) = \prod_{v \in \abs{C}} \frac{1}{1- u^{\deg v}}.\] Then for a real number $u\in (0, \frac{1}{q})$ we have
\[ \zeta_C(u) \leq \frac{ ( 1 + u q^{1/2})^{2g}}{ (1- u) ( 1- u q) } .\]
\end{lemma}
\begin{proof} We apply Weil's theorem~\cite[Theorem 5.10]{Rosen2002} that \[ \zeta_C(u) =  \frac{\prod_{i=1}^{2g} (1-\alpha_i u)} { (1- u) ( 1- u q) } \] for some complex numbers $\alpha_1,\dots,\alpha_{2g}$ with absolute value $q^{1/2}$ and then trivially bound each term in the numerator. \end{proof}

\subsection{Preliminaries on $\Delta_{k,p}$}
We give some brief preparatory lemmas on $\Delta_{k,p}$  and $\gamma_{k,p}$.

\begin{lemma}\label{Delta-stability} For $\begin{pmatrix} x \\ y\end{pmatrix} \in \Delta_{k,p}$, we have $\begin{pmatrix} x+a \\ y+b \end{pmatrix} \in \Delta_{k,p}$ whenever $a,b\geq 0$.  \end{lemma}

\begin{proof} This follows immediately from the fact that if $\begin{pmatrix} i \\ j \end{pmatrix} $ satisfy $\gcd(i,j,p)=1$ but $p\mid (k-1)i-j$ then so do $\begin{pmatrix} i+p \\ j \end{pmatrix} $ and $\begin{pmatrix} i \\ j+p \end{pmatrix} $, together with the definition of convex hull.  \end{proof}

The next two results give upper and lower bounds on $\gamma_{k,p}$.

\begin{lemma}\label{gamma-lower-bound} We have $\gamma_{k,p} \geq \frac{k-2}{2k-2}$. \end{lemma}

\begin{proof} If $k=2$ then the statement is that $\gamma_{k,p} \geq 0$ which is true by definition, so we may assume $k>2$. Observe that $\Delta_{k,p}$ certainly contains the point $\begin{pmatrix}1 \\ k-1 \end{pmatrix}$. By Lemma \ref{Delta-stability}, $\Delta_{k,p}$ contains all points  $\begin{pmatrix} x \\ y\end{pmatrix}$ with $x \geq 1 $ and $y \geq k-1$. In particular this includes $\begin{pmatrix} \frac{ 2k-2}{k-2} \\ k-1 \end{pmatrix} = \frac{2k-2}{k-2}  \begin{pmatrix}  1\\ \frac{k-2}{2} \end{pmatrix} $, showing that $\begin{pmatrix}  1\\ \frac{k-2}{2} \end{pmatrix}  \in  \frac{k-2}{2k-2} \Delta_{k,p}$ and hence $\gamma_{k,p} \geq \frac{k-2}{2k-2}$.  \end{proof}

\begin{lemma} \label{gamma-upper-bound} If $p> k-1$ we have \[\gamma_{k,p} \leq \frac{k-2}{2k-2} \left( 1+ \frac{k}{p} \right).\]
\end{lemma}

\begin{proof} If $p \mid (k-1) i -j$ but $\gcd(i,j,p)=1$ then we cannot have $j=0$ as this implies $p\mid (k-1) i$ and thus $p\mid i$ so $\gcd(i,j,p)=p$, and we cannot have $i=0$ as this implies $p \mid j$ so $\gcd(i,j,p)=p$. Since $p \mid (k-1) i -j$ , we have either $ (k-1) i- j \geq p$ or  $(k-1)i -j \leq 0$. In the first case we have $j\geq 1$ so $i \geq \frac{p+1}{k-1}$ and in the second case we have $i\geq 1$ so $j \geq k-1$. Hence in either case we have
\begin{equation}\label{gub-line} (k-1) (k-2)  i + (p+2-k) j \geq   p (k-1) \end{equation}
and thus \eqref{gub-line} is satisfied for each $\begin{pmatrix} i \\ j \end{pmatrix} \in \Delta_{k,p}$.

In particular if $ \begin{pmatrix}  1\\ \frac{k-2}{2} \end{pmatrix}  \in  \gamma \Delta_{k,p}$ then this implies
\[ (k-1)(k-2)  + (p+2-k) \frac{k-2}{2} \geq \gamma  p (k-1)  \]
so \[ \gamma \leq    \frac{ 2 (k-1)(k-2)  + (p+2-k) (k-2) } { 2p (k-1) }  = \frac{k-2}{2k-2} \left( 1+ \frac{k}{p} \right) .\] Since this holds for all such $\gamma$, we have $ \gamma_{k,p} \leq    \frac{k-2}{2k-2} \left( 1+ \frac{k}{p} \right) .$ \end{proof}

\section{The circle method: setup and major arcs}\label{s-major}

Throughout the proof we fix a finite field $\mathbb F_q$, a smooth projective curve $C$ of genus $g$ over $\mathbb F_q$, positive integers $s,k,e$, a line bundle $L$ of degree $e$ on $C$, and $f\in H^0(C, L^k)$. We also fix a nontrivial additive character $\psi\colon \mathbb F_q\to \mathbb C^\times$. Implicit constants in big $O$ notation will be uniform in these parameters unless specified otherwise.

We assume throughout that $e> 2g-2$, and $e> 0$ if $g=0$, which implies by Riemann-Roch~\cite[\S8.6, Corollary 2]{Fulton2008} that $\dim H^0(C, L) =e+1-g$ and $\dim H^0(C, L^k) =ke+1-g$. 

For $\alpha \in H^0(C, L^k)^\vee$, let \begin{equation}\label{S1-def} S(\alpha) =  \sum_{a \in H^0(C, L)} \psi ( \alpha (a^k)).\end{equation}

The first step of the circle method is provided by the following lemma.

\begin{lemma} We have 
\begin{equation}\label{circle-setup} \#\{ \mathbf a \in H^0( C, L )^s \mid  \sum_{i=1}^s a_i^k = f\}  = \frac{1}{ q^{ke+1-g}} \sum_{ \alpha \in H^0(C,L^k)^\vee}  S(\alpha)^s \overline{\psi(\alpha(f))} . \end{equation}
\end{lemma}

\begin{proof} We have 
\[  \frac{1}{ q^{ke+1-g}} \sum_{ \alpha \in H^0(C,L^k)^\vee}  S(\alpha)^s \overline{\psi(\alpha(f))} =  \frac{1}{ q^{ke+1-g}} \sum_{ \alpha \in H^0(C,L^k)^\vee}\overline{\psi(\alpha(f))} \Bigl( \sum_{a \in H^0(C, L)} \psi ( \alpha (a^k))\Bigr) ^s \]
\[ =  \frac{1}{ q^{ke+1-g}} \sum_{ \alpha \in H^0(C,L^k)^\vee} \overline{\psi(\alpha(f))}  \sum_{ \mathbf a \in H^0( C, L )^s} \prod_{i=1}^s \psi ( \alpha (a_i^k))  \]
\[ =  \frac{1}{ q^{ke+1-g}} \sum_{ \alpha \in H^0(C,L^k)^\vee}\sum_{ \mathbf a \in H^0( C, L )^s}  \psi ( \alpha (\sum_{i=1}^s a_i^k -f)) \]
\[ =  \sum_{ \mathbf a \in H^0( C, L )^s} \frac{1}{ q^{ke+1-g}} \sum_{ \alpha \in H^0(C,L^k)^\vee} \psi ( \alpha (\sum_{i=1}^s a_i^k -f)) \]
and \[  \frac{1}{ q^{ke+1-g}} \sum_{ \alpha \in H^0(C,L^k)^\vee}  \psi ( \alpha (\sum_{i=1}^s a_i^k -f)) =\begin{cases} 1 & \textrm{if }\sum_{i=1}^s a_i^k =f\\ 0 & \textrm{otherwise}\end{cases} \] which gives \eqref{circle-setup}.
\end{proof}

In this section, we begin by analyzing the different $\alpha$ appearing in the sum of \eqref{circle-setup}. We break the $\alpha$ into major arcs and minor arcs. We then analyze the major arcs, explicitly evaluating the sums $S(\alpha)$ for $\alpha$ in the major arcs, and then extracting the main term from the sum over $\alpha$ in the major arcs, up to an error term which we bound.

To break $\alpha$ into major arcs and minor arcs, we will need to consider the finite closed subschemes of $C$. We drop the word ``finite'', that is, a closed subscheme $Z$ of $C$ is always assumed to be finite.

For a closed subscheme $Z$ of $C$, we will mainly be interested in the space $H^0(Z, L^k)$ of global sections of the restriction of the line bundle $L^k$ to the subscheme $Z$. For $i \colon Z \to C$ the closed immersion, this can also be written as $ H^0(Z, i^* L^k) = H^0(C, i_* i^* L^k)$. We have the short exact sequence $0 \to \mathcal O_C(-Z) \to \mathcal O_C \to i_* i^* \mathcal O_C\to 0$~\cite[II, Proposition 6.18 and Definition on p. 115]{Hartshorne1977} whose tensor product with $L^k$ is $0 \to  L^k(-Z) \to  L^k \to i_* i^* L^k \to 0$, so that $H^0(Z, L^k) =  H^0(C, L^k/ L^k(-Z))$, giving an alternate definition for $H^0(Z,L^k)$. This short exact sequence induces the long exact sequence
\begin{equation}\label{subscheme-les} 0 \to H^0 (C, L^k(-Z)) \to H^0(C, L^k) \to H^0(Z, L^k) \to H^1( C, L^k(-Z)) \to H^1(C, L^k) \to 0 \end{equation} where we stop at $H^1$ because all higher cohomology groups vanish \cite[III, Theorem 2.7]{Hartshorne1977}.

We say $\alpha\in H^0(C, L^k)^\vee $ \emph{factors through} a closed subscheme $ Z \subset C$ if $\alpha$ factors through the restriction map $H^0(C, L^k) \to H^0(Z, L^k)$. If $ke -\deg Z > 2g-2$ then the restriction map is surjective by \eqref{subscheme-les}, the Serre duality isomorphism~\cite[III, Corollary 7.7]{Hartshorne1977}  $H^1( C, L^k(-Z)) = H^0(C, K_C \otimes L^{-k}(Z))^\vee$ and the fact that $\deg K_C \otimes L^{-k}(Z) = 2g-2 -ke + \deg Z<0$ so that  $H^0(C, K_C \otimes L^{-k}(Z))$ vanishes~\cite[IV, Lemma 1.2]{Hartshorne1977}. Thus the induced linear form $H^0(Z, L^k)\to \mathbb F_q$ is unique. In this case we denote this linear form $H^0(Z, L^k)\to \mathbb F_q$ by $\overline{\alpha}$, or, if the dependence on $Z$ is relevant, by $\overline{\alpha}_Z$.

We say $\overline{\alpha }\in H^0(Z, L^k)^\vee$ is \emph{primitive} if $\overline{\alpha}$ does not factor through $H^0(Z', L^k)$ for any proper closed subscheme $Z'$ of $Z$ (i.e. a closed subscheme of $Z$ other than $Z$ itself).
 
The analogue of Dirichlet's approximation theorem in this context is as follows.

\begin{lemma}{\cite[Lemma 12]{MHL}}\label{mhl-one} For any $\alpha$ in $H^0(C,L^k)^\vee$, there exists some closed subscheme $Z \subset C$ of degree $\leq \frac{ke}{2}+1$ such that $\alpha$ factors through $Z$.  \end{lemma}
One should think of $ Z$ as analogous to the denominator of a rational number approximating $\alpha$ and $\overline{\alpha}_Z$ as analogous to the numerator.

\begin{proof}To make the paper self-contained, we repeat the proof of \cite[Lemma 12]{MHL}. 

Let $D$ be a divisor of degree $\lfloor \frac{ke}{2}\rfloor+1$ on $C$. There is a pairing 
\[ H^0(C, D) \times H^0(C, L^{k} (-D)) \to H^0(C, L^k) \to_{\alpha} \mathbb F_q \]
obtained by multiplying two sections and then applying $\alpha$.  The line bundle $L^{k}(-D)$ has degree $\lceil \frac{ke}{2} \rceil -1 \geq e-1 \geq 2g-2$, with equality only if $e=2g-1$ and $k=2$. By Riemann-Roch~\cite[\S8.6, Corollary 2]{Fulton2008}, it follows that 
\[ \dim H^0(C, L^{k}(-D)) = \deg L^{k} (-D) +1-g = ke-\deg D + 1-g \]
unless $e=2g-1$ and $k=2$. In the case $e=2g-1$ and $k=2$, a similar argument shows that $\dim H^0(C, L^{k}(-D)) $ is either $ke-\deg D + 1-g$ or $ke-\deg D + 2-g$, using the fact that a degree $0$ line bundle on a curve either has a zero-dimensional space of global sections or is trivial and hence has a one-dimensional space of global sections~\cite[IV, Lemma 1.2]{Hartshorne1977}.

We have
\[ \dim H^0(C, D) \geq \deg D + 1-g > ke - \deg D + 1-g = \dim H^0(C, L^{k}(-D))\] 
and in the case $e=2g-1$ and $k=2$ we have
\[ \dim H^0(C, D) \geq \deg D + 1-g > ke - \deg D +1 + 1-g \geq \dim H^0(C, L^{k}(-D))\]
so in either case, there must be a nonzero vector  $s\in  H^0(C, D)$ whose pairing with each element of $H^0(C, L^{k} (-D))$ is zero. Let $\operatorname{div}(s)$ be the divisor of $s$. (That is, the section $s$ corresponds to a rational function $f$ with $\operatorname{div}(f) +D$ an effective divisor. Let $\operatorname{div}(s)=\operatorname{div}(f)+D$.)  Let $Z$ be a closed subscheme corresponding to the effective divisor $\operatorname{div}(s)$. By definition, $\operatorname{div}(s)$ is linearly equivalent to $D$ and thus $Z$ has degree $\deg D \leq \frac{ke}{2}+1$~\cite[II, Corollary 6.10]{Hartshorne1977}.  Every section $u$ of $L^k$ which vanishes on $Z$ may be expressed as $s$ times an element of $H^0(C, L^{k} (-D))$ (since $s^{-1}$ is a meromorphic section of $\mathcal O(-D)$ and thus $u s^{-1}$ is a meromorphic section of $L^k (-D)$ but since $u$ vanishes on $Z$, the order of vanishing of $u$ at each point is at least the order of vanishing of $s$, so that $u s^{-1}$ is an ordinary section of $L^k(-D)$), thus $\alpha$ vanishes on every section of $L^k$ which vanishes on $Z$, and hence $\alpha $ factors through $Z$. \end{proof} 

We let $\deg \alpha$ be the minimum degree of a closed subscheme $Z$ through which $\alpha$ factors. Then $\deg \alpha$ is an integer with \[ 0 \leq \deg \alpha \leq \frac{ke}{2}+1.\] 

\begin{lemma}\label{mhl-two} If $\deg \alpha < \frac{ke}{2}-g+1$, then there exists a unique closed subscheme $Z$ of minimal degree through which $\alpha$ factors. \end{lemma}

\begin{proof}This follows immediately from \cite[Lemma 13]{MHL}, as noted, for a slightly different statement, in \cite[Remark 14]{MHL}. To make the paper self-contained, we repeat the proof of \cite[Lemma 13]{MHL}. 

Let $Z_1$ and $Z_2$ be two closed subschemes of minimal degree through which $\alpha$ factors. Assuming $Z_1 \neq Z_2$ and $\deg \alpha < \frac{ke}{2}-g+1$, we will obtain a contradiction.

Since $Z_1$ and $Z_2$ are closed subschemes of minimal degree, we have $\deg Z_1 =\deg Z_2 =\deg \alpha < \frac{ke}{2}-g+1$ so $\deg (Z_1 \cup Z_2) < ke-2g+2$.  For a closed subscheme $Z$, let $V_Z$ be the space of sections of $L^k$ vanishing on $Z$. Then $V_Z$ is isomorphic to $H^0(C, L^k(-Z))$ and hence has dimension 
\[\deg L^k(-Z) + 1-g= ke -\deg Z + 1-g \] by Riemann-Roch~\cite[\S8.6, Corollary 2]{Fulton2008} as long as $\deg Z < ke-2g+2$.

Since $\alpha$ vanishes on $V_{Z_1}$ and $V_{Z_2}$ by assumption, it vanishes on their sum $V_{Z_1} + V_{Z_2}$.

We have, since a section vanishes on both $Z_1$ and $Z_2$ if and only if it vanishes on their union,  \[ \dim ( V_{Z_1}+ V_{Z_2}) =\dim V_{Z_1}  + \dim V_{Z_2} - \dim (V_{Z_1} \cap V_{Z_2})  =\dim V_{Z_1}  + \dim V_{Z_2} - \dim V_{Z_1 \cup Z_2} \]\[= ke - \deg Z_1 +1-g + ke -\deg Z_2 + 1-g - ( ke -\deg (Z_1 \cup Z_2) +1-g) \]\[ = ke - ( \deg Z_1 + \deg Z_2 - \deg (Z_1 \cup Z_2)) +1-g = ke - \deg (Z_1 \cap Z_2) + 1-g = \dim V_{Z_1 \cap Z_2} \]
where we have used that $Z_1, Z_2, Z_1\cap Z_2, $ and $Z_1 \cup Z_2$ all have degree $<ke-2g+2$. Clearly $V_{Z_1} + V_{Z_2} \subseteq V_{Z_1 \cap Z_2} $ so $V_{Z_1 } + V_{Z_2} = V_{Z_1\cap Z_2} $ and thus $\alpha$ vanishes on $V_{Z_1 \cap Z_2}$, hence $\alpha$ factors through $Z_1 \cap Z_2$. If $Z_1\neq Z_2$ then $Z_1 \cap Z_2$ has lesser degree than one of $Z_1$ or $Z_2$, so they are not closed subschemes of minimal degree, giving the desired contradiction.\end{proof}

If $Z$ is a closed subscheme of minimal degree through which $\alpha$ factors, then $\overline{\alpha}$ is primitive. 

We will consider $\alpha$ to lie in the major arcs if $\deg(\alpha) \leq e-2g+1$, and otherwise to lie in the minor arcs. 

For $\overline{\alpha} \in H^0(Z,L^k)^\vee$, let \[ S_Z(\overline{\alpha} ) = \frac{1}{q^{\deg Z}} \sum_{\overline{a} \in H^0(Z,L)} \psi(\overline{\alpha}(\overline{a}^k)) .\]

We now, in Lemma \ref{lem-major-start}, evaluate $S(\alpha)$ for $\alpha$ in the major arcs in terms of $S_Z(\overline{\alpha})$, then, in Lemma \ref{major-factorization}, prove a multiplicativity property of $S_Z(\overline{\alpha})$, allowing us to reduce the study of $S_Z(\overline{\alpha})$ to the case of $Z$ supported at a single closed point, then, in Lemma \ref{major-local-bound}, bound $S_Z(\overline{\alpha})$ for $Z$ supported at a single closed point, and, in Lemma \ref{lv-es}, relate $S_Z(\overline{\alpha})$ for $Z$ supported at a single closed point to the local terms $\ell_v(f)$. These results enable us, in Lemma \ref{lem-major-mt}, to express the main term of Theorem \ref{intro-curve} as a sum of $S_Z(\overline{\alpha})$, so that we can, in Lemma \ref{lem-two-parts}, evaluate the number of representations of $f$ as equal to the main term plus two error terms, one of which we immediately handle in Lemma \ref{easier-euler} and the other we will handle in Section \ref{s-minor}.

We make some basic observations about $H^0(C,L)$ which will be used in the next three lemmas. Since $L$ is locally free of rank one on $C$, $L$ is free of rank one on $Z$, so $H^0(Z,L) \cong H^0(Z,\mathcal O_Z)$, though this isomorphism is noncanonical until we choose a local generator of $L$ at each point of $Z$. For $Z$ supported at a single closed point, i.e. $Z = m[v]$ for some $m \in \mathbb N$ and $v\in |C|$, we have $H^0(m [v], \mathcal O_{m[v]}) = \mathcal O_{C_v}/\pi^m$ where $\mathcal O_{C_v}$ is the local ring of $C$ at $v$ and $\pi$ is a uniformizer, because the ring of functions on a closed subscheme supported at a point is quotient of the local ring at that point by an ideal, $\mathcal O_{C_v}$ is a discrete valuation ring so its only nonzero ideals are powers of the uniformizer, and by definition the multiplicity $m$ is the length of the quotient ring, which equals the relevant power of the uniformizer.  Note that $\mathcal O_{C_v}/\pi^{m}$ is an $\mathbb F_q$-vector space of dimension $m \deg v = \deg ( m [v])$. For a general $Z = \sum_i m_i [v_i]$ we have $ H^0(Z,\mathcal O_Z) = \prod_i H^0(m _i[v_i], \mathcal O_{m_i[v_i]})$ is an $\mathbb F_q$-vector space of dimension $\sum_i m_i \deg v_i = \deg Z$.

\begin{lemma}\label{lem-major-start} Let $\alpha \in H^0(C, L^k)^\vee$ factor through a closed subscheme $Z$ of $C$ of degree $\leq e-2g+1$. Then
\begin{equation}\label{S1-major} S(\alpha) = q^{e-g+1} S_Z(\overline{\alpha}) .\end{equation}
\end{lemma}
\begin{proof}  The line bundle $L(-Z)$ has degree $e -\deg Z \geq 2g-1$ and thus $H^1(C, L(-Z))$ vanishes~\cite[III, Corollary 7.7 and IV, Lemma 1.2]{Hartshorne1977}. The short exact sequence $0 \to L(-Z) \to L \to L \mid_Z \to 0$ induces a long exact sequence (the $k=1$ case of \eqref{subscheme-les})  \[ \dots \to H^0(C, L) \to H^0(Z, L) \to H^1(C, L(-Z)) \to \dots.\] Combining these, we conclude that $H^0(C, L) \to H^0(Z, L)$ is surjective. It follows that each $\overline{a} \in H^0(Z,L)$ has exactly \[ \frac{ \# H^0(C, L)}{\#H^0(Z,L)} = \frac{ q^{\dim H^0(C,L)}}{q^{\dim H^0(Z,L)}} = \frac{ q^{e-g+1}}{q^{\deg Z}}\] preimages $a \in H^0(C,L)$ . For any of these preimages, the factorization of $\alpha$ through $Z$ implies that $\psi(\alpha({a}^k)) =\psi(\overline{\alpha}(\overline{a}^k)) $. So the sum in \eqref{S1-def} reduces to 
\[  \frac{q^{e-g+1} }{q^{\deg Z}} \sum_{\overline{a} \in H^0(Z,L)} \psi(\overline{\alpha}(\overline{a}^k)) = q^{e-g+1} S_Z(\overline{\alpha}) .\qedhere\] \end{proof}

The next lemma gives the multiplicativity properties of the sums $S_Z$. 

\begin{lemma}\label{major-factorization} Let $Z$ be a disjoint union of subschemes $Z_1$ and $Z_2$ and fix $\overline{\alpha} \in H^0(Z, L^k)^\vee$.  Let $\overline{\alpha}_{Z_i}$ be the restriction of $\overline{\alpha}$ to $H^0(Z_i,L^k)$. Then we have
\[ S_Z(\overline{\alpha})=S_{Z_1} (\overline{\alpha}_{Z_1}) S_{Z_2}(\overline{\alpha}_{Z_2})\]  and  for each $f \in H^0(Z, L^k)$ we have $\overline{\alpha}(f) = \overline{\alpha}_{Z_1}(f) +\overline{\alpha}_{Z_2}(f) .$ Furthermore, $\overline{\alpha}$ is primitive if and only if $\overline{\alpha}_{Z_1}$ and $\overline{\alpha}_{Z_2}$ are primitive. \end{lemma}
\begin{proof} The first two claims follow immediately from the ``Chinese remainder theorem" splitting $H^0(Z ,L^k) = H^0(Z_1,L^k) \times H^0(Z_2, L^k)$.

For the last claim, if $\overline{\alpha}$ factors through some proper subscheme $Z'$ then $\overline{\alpha}_{Z_1}$ factors through $Z_1 \cap Z'$ and $\overline{\alpha}_{Z_2}$ factors through $Z_2 \cap Z'$ and at least one of these is proper. Conversely, if $\overline{\alpha}_{Z_1}$ factors through a proper subscheme $Z_1'$ then $\overline{\alpha}$ factors through $Z_1' \cup Z_2$, which is proper, and similarly with $\overline{\alpha}_{Z_2}$. \end{proof}

Using Lemma \ref{major-factorization}, we can reduce the calculation of $S_Z$ to $Z$ supported at a single closed point.

\begin{lemma}\label{major-local-bound} Let $v \in |C|$ be a closed point, $m$ a nonnegative integer, and $\overline{\alpha } \in H^0(m[v], L^k)^\vee$ a primitive linear form. Assume $k$ is relatively prime to $p$.

If $m \not \equiv 1 \bmod k$ then \[S_{m[v]}(\overline{\alpha}) =\frac{1}{q^{\lceil \frac{m}{k} \rceil \deg v }}.\]

If $m \equiv 1 \bmod k$ then \[\abs{ S_{m[v]}(\overline{\alpha}) } \leq \frac{k-1}{ q^{ \left(\frac{m-1}{k} + \frac{1}{2}\right) \deg v}} .\]
\end{lemma}

\begin{proof} We have \[ S_{m[v]}(\overline{\alpha} ) = \frac{1}{q^{m \deg v }} \sum_{\overline{a} \in H^0(m[v] ,L)} \psi(\overline{\alpha}(\overline{a}^k)).\] We fix a local generator $g$ of $L$ at $v$, giving an isomorphism between $H^0(m[v], L)$ and $H^0(m[v],\mathcal O_{m[v]})$ by division by $g$ and inducing an isomorphism between $H^0(m[v], L^k)$ and $H^0( m[v], \mathcal O_{m[v]})$ by division by $g^k$. Furthermore we may write $H^0(m[v],\mathcal O_{m[v]})$ as $\mathcal O_{C_v}/\pi^m$ where $\mathcal O_{C_v}$ is the local ring of $C$ at $v$ and $\pi$ is a uniformizer. We may thus express $\overline{\alpha}$ as a primitive linear form on $\mathcal O_{C_v}/\pi^m$ and obtain
\[ S_{m[v]}(\overline{\alpha} ) = \frac{1}{q^{m \deg v }} \sum_{\overline{a} \in \mathcal O_{C_v}/\pi^m} \psi(\overline{\alpha}(\overline{a}^k)).\]  

We first handle the case $m=0$, which is trivial, and $m=1$, where $\mathcal O_{C_v}/\pi$ is a finite field $\mathbb F_{q^{\deg v}}$ and the sum $\sum_{\overline{a} \in \mathcal O_{C_v}/\pi} \psi(\overline{\alpha}(\overline{a}^k))$ is a Gauss sum of an additive character composed with the $k$th power map over that finite field. The bound
\[ \Bigl | \sum_{\overline{a} \in \mathcal O_{C_v}/\pi} \psi(\overline{\alpha}(\overline{a}^k))\Bigr|  \leq (k-1) q^{\frac{ \deg v}{2}}\] then follows from the classical bound for Gauss sums, and implies the $m=1$ case.

For $m \geq 2$ we split the sum into $\overline{a}$ that are multiples of $\pi$ and those that are not, i.e.
\[ \sum_{\overline{a} \in \mathcal O_{C_v}/\pi^m} \psi(\overline{\alpha}(\overline{a}^k))= \sum_{\substack{ \overline{a} \in \mathcal O_{C_v}/\pi^m\\ \pi \nmid \overline{a}}} \psi(\overline{\alpha}(\overline{a}^k))+ \sum_{\substack{ \overline{a} \in \mathcal O_{C_v}/\pi^m \\ \pi \mid \overline{a}}}\psi(\overline{\alpha}(\overline{a}^k)).\]
The first term vanishes for $m \geq 2$ by a stationary phase argument: For $b \in \mathbb F_{q^{\deg v}}$ we have $(\overline{a} + \pi^{m-1} b)^k = \overline{a}^k + k \overline{a}^{k-1} \pi^{m-1} b$ since $(\pi^{m-1})^2 =\pi^{2m-2}$ is divisible by $\pi^m$ and hence vanishes in $\mathcal O_{C_v}/\pi^m$. Thus
\[ \psi(\overline{\alpha}((\overline{a} + \pi^{m-1} b)^k))= \psi(\overline{\alpha} ( \overline{a}^k + k \overline{a}^{k-1} \pi^{m-1} b) ) =\psi(\overline{\alpha}(\overline{a}^k)) \psi(\overline{\alpha} ( k \overline{a}^{k-1} \pi^{m-1} b)).\]
Because  $k$ and $\overline{a}$ are nonzero modulo $\pi$, as $b$ runs over $\mathbb F_{q^{\deg v}}$, the product $k \overline{a}^{k-1} \pi^{m-1} b$ runs over all multiples of $\pi^{m-1}$ in $\mathcal O_{C_v}/\pi^m$. Since $\alpha$ is primitive, this implies the $\mathbb F_q$-linear map $b\mapsto \overline{\alpha} ( k \overline{a}^{k-1} \pi^{m-1} b)$ is nonzero and hence surjective. So there exists some $b$ such that $\psi(\overline{\alpha} ( k \overline{a}^{k-1} \pi^{m-1} b))\neq 1$. Furthermore, we may choose $b$ depending on $\overline{a}$ (mod $\pi$) only.  The change of variables $\overline{a} \mapsto \overline{a} + \pi^{m-1} b$ cancels the contributions of all $\overline{a}$ in a single nonzero residue class mod $\pi$. It follows that 
\[\sum_{\substack{ \overline{a} \in \mathcal O_{C_v}/\pi^m\\ \pi \nmid \overline{a}}} \psi(\overline{\alpha}(\overline{a}^k))=0.\]

For $2 \leq m \leq k$, if $\pi \mid \overline{a}$ then $\pi^m \mid \overline{a}^k$ so that \[\sum_{\substack{ \overline{a} \in \mathcal O_{C_v}/\pi^m \\ \pi \mid \overline{a}}}\psi(\overline{\alpha}(\overline{a}^k)) = \sum_{\substack{ \overline{a} \in \mathcal O_{C_v}/\pi^m \\ \pi \mid \overline{a}}}1 = q^{ \deg v (m-1)}\] and thus \[ S_{m[v]}(\overline{\alpha}) =\frac{1}{q^{\deg v}},\] handling this case.

Finally if $m>k$ then we can define $\overline{\beta} \in ( \mathcal O_{C_v}/\pi^{m-k} )^\vee$  by the rule $\overline{\beta} ( h) = \overline{\alpha} ( \pi^k h)$. Then $\overline{\beta}$ is primitive since $\overline{\alpha}$ is and

\[\sum_{\substack{ \overline{a} \in \mathcal O_{C_v}/\pi^m \\ \pi \mid \overline{a}}}\psi(\overline{\alpha}(\overline{a}^k))=\sum_{\substack{ \overline{a} \in \mathcal O_{C_v}/\pi^m \\ \pi \mid \overline{a}}}\psi(\overline{\beta}((\overline{a}/\pi)^k))= q^{ \deg v (k-1)} \sum_{\substack{ \overline{c} \in \mathcal O_{C_v}/\pi^{m-k} }}\psi(\overline{\beta}(\overline{c}^{k}))\] since each residue class $\overline{c} \in \mathcal O_{C_v}/\pi^{m-k} $ may be expressed as $\overline{a}/\pi$ for exactly $q^{ \deg v(k-1)}$ residue classes $\overline{a} \in \mathcal O_{C_v}/\pi^m$, all divisible by $\pi$, those being the $q^{\deg v(k-1)}$ solutions to the congruence $\overline{a} = \pi \overline{c}\bmod \pi^{m-k+1}$. This gives
\[ S_{m[v]}(\overline{\alpha}) =\frac{1}{q^{\deg v}} S_{(m-k)[v]}(\overline{\beta})\]
which handles all the cases $m>k$ by induction on $m$. \end{proof}

The next lemmas show that plugging the right-hand side of \eqref{S1-major} into a formula similar to \eqref{circle-setup} gives the desired main term.

\begin{lemma}\label{lv-es} For $f \in H^0(C, L^k)$ we have
\[ \ell_v(f)=  \sum_{m=0}^{\infty} \sum_{\substack{\overline{\alpha} \in H^0(m[v],L^k)^\vee\\ \textrm{primitive}}} S_{m[v]} (\overline{\alpha})^s \overline{\psi(\overline{\alpha}(f))}.\]
\end{lemma}

\begin{proof}It suffices to check that
\begin{equation}\label{lv-es-1} \sum_{m=0}^{r} \sum_{\substack{\overline{\alpha} \in H^0(m[v],L^k)^\vee\\ \textrm{primitive}}} S_{m[v]} (\overline{\alpha})^s \overline{\psi(\overline{\alpha}(f))}=   \frac{ \# \{ \mathbf b \in  (\mathcal O_{C_v} /\pi_v^r)^s \mid \sum_{i=1}^s b_i^k \equiv f \bmod \pi^r\}}{ q^{r (s-1) \deg v} }\end{equation}
as taking the limits of both sides as $r$ goes to $\infty$ gives the stated equality. We now verify this. (Compare \cite[Lemma 16]{MHL} which is the same argument in a slightly different setting.)

We first check that
\begin{equation}\label{lv-es-2} \sum_{\substack{\overline{\beta} \in H^0(r[v],L^k)^\vee }} S_{r[v]} (\overline{\beta})^s \overline{\psi(\overline{\beta}(f))}=   \frac{ \# \{ \mathbf b \in  (\mathcal O_{C_v} /\pi_v^r)^s \mid \sum_{i=1}^s b_i^k \equiv f \bmod \pi^r\}}{ q^{r (s-1) \deg v} } .\end{equation}

To do this we expand
\[\sum_{\substack{\overline{\beta} \in H^0(r[v],L^k)^\vee }} S_{r[v]} (\overline{\beta})^s \overline{\psi(\overline{\beta}(f))} = \sum_{\substack{\overline{\beta} \in H^0(r[v],L^k)^\vee }} \Biggl( \frac{1}{q^{r \deg  v} }\sum_{\overline{a} \in H^0( r[v], L) }\psi ( \overline{\beta} ( \overline{a}^k)) \Biggr)^ s \overline{\psi(\overline{\beta}(f))} \]
\[ = \frac{1}{ q^{ r s \deg v}}   \sum_{\substack{\overline{\beta} \in H^0(r[v],L^k)^\vee }}\overline{\psi(\overline{\beta}(f))}\sum_{\overline{\mathbf a} \in H^0( r[v], L)^s } \prod_{i=1}^s   \psi ( \overline{\beta} ( \overline{a}_i^k))   \]
\[ = \frac{1}{ q^{ r s \deg v}}   \sum_{\substack{\overline{\beta} \in H^0(r[v],L^k)^\vee }}\sum_{\overline{\mathbf a} \in H^0( r[v], L)^s }   \psi ( \overline{\beta} (\sum_{i=1}^s  \overline{a}_i^k - f ))   \]
\[ = \frac{1}{ q^{ r s \deg v}}   \sum_{\overline{\mathbf a} \in H^0( r[v], L)^s } \sum_{\substack{\overline{\beta} \in H^0(r[v],L^k)^\vee }}  \psi ( \overline{\beta} (\sum_{i=1}^s  \overline{a}_i^k - f ))   \]
and
\[ \sum_{\substack{\overline{\beta} \in H^0(r[v],L^k)^\vee }}  \psi ( \overline{\beta} (\sum_{i=1}^s  \overline{a}_i^k - f ))  = \begin{cases} q^{ r \deg v } & \textrm{if } \sum_{i=1}^s \overline{a}_i^k =f \\ 0 & \textrm{otherwise} \end{cases}\]
which gives
\[\sum_{\substack{\overline{\beta} \in H^0(r[v],L^k)^\vee }} S_{r[v]} (\overline{\beta})^s \overline{\psi(\overline{\beta}(f))} = \frac{\# \{ \mathbf a\in H^0(r[v], L)^s \mid \sum_{i=1}^s a_i^k =f \} } { q^{ r (s-1)  \deg v}}   \]
and then fixing a local generator of $L$ to obtain an isomorphism between $H^0(r[v], L)$ and $\mathcal O_{C_v} /\pi_v^r$ gives \eqref{lv-es-2}.

We next observe that for $\overline{\beta} \in H^0(r[v],L^k)^\vee$, if $\overline{\beta}$ factors through $m[v]$ then $\overline{\beta}$ factors through $(m+1)[v]$. There hence exists a unique $m$ such that $\overline{\beta}$ factors through $m [v]$ but not through $(m-1)[v]$. Furthermore, since every proper closed subscheme of $m [v]$ has the form $m'[v]$ for some $m'<m$, the factorization of $\overline{\beta}$ through $m[v]$ is a primitive form $\overline{\alpha} \in H^0( m[v], L^k)^\vee$. For this $\overline{\alpha}$ we trivially have $\overline{\alpha} (f) = \overline{\beta}(f)$. We also have   \[ S_{m [v] } (\overline{\alpha} )= S_{r [v]} ( \overline{\beta}) \] since the natural map $H^0( r [v], L^k) \to H^0(m [v] ,L^k)$ is surjective with fibers of size $q^ { (r-m) \deg v} = \frac{ q^{ \deg  r[v]}}{q^{\deg m[v]}} .$

Finally, each $\overline{\alpha}$ composes with the projection $H^0( r [v], L^k) \to H^0(m [v] ,L^k)$ to give a linear form $\overline{\beta} \in H^0(r[v],L^k)^\vee$, so each $\overline{\alpha}$ arises from exactly one such $\beta$. Combining all these observations, we obtain
\begin{equation}\label{lv-es-3} \sum_{m=0}^{r} \sum_{\substack{\overline{\alpha} \in H^0(m[v],L^k)^\vee\\ \textrm{primitive}}} S_{m[v]} (\overline{\alpha})^s \overline{\psi(\overline{\alpha}(f))}=  \sum_{\substack{\overline{\beta} \in H^0(r[v],L^k)^\vee }} S_{r[v]} (\overline{\beta})^s \overline{\psi(\overline{\beta}(f))}\end{equation}

Combining \eqref{lv-es-2} and \eqref{lv-es-3}, we obtain \eqref{lv-es-1}. \end{proof}

Let \[ \operatorname{MT}_e=  q^{ e(s-k) +(s-1)(1-g)  } \prod_{\substack{v \in \abs{C} }}  \ell_v(f) .\] For compactness we leave the dependence of the main term on $f$ and the other parameters implicit, but keeping track of the dependence on $e$ will be crucial later.

\begin{lemma}\label{lem-major-mt} Assume $s \geq 5$ and $k$ is coprime to $p$. Then we have
\[  \frac{q^{ s (e+1-g)} }{ q^{ke+1-g}} \sum_{\substack{ Z \subset C\\  \textrm{ finite closed }}} \sum_{\substack{\overline{\alpha} \in H^0(Z,L^k)^\vee\\ \textrm{primitive}}} S_Z(\overline{\alpha})^s \overline{\psi(\overline{\alpha}(f))} =  \operatorname{MT}_e.\]
\end{lemma}

\begin{proof} We have $\frac{q^{ s (e+1-g)} }{ q^{ke+1-g}} = q^{ e(s-k) +(s-1)(1-g)  }$ so it suffices to show the sum on the left hand side is $ \prod_{\substack{v \in \abs{C} }}  \ell_v(f) $.

Each finite closed subscheme $Z$ can be written uniquely as a disjoint union of subschemes of the form $m[v]$ for closed points $v$ and positive integers $m$. This and Lemma \ref{major-factorization} gives the factorization
\[   \sum_{\substack{ Z \subset C\\  \textrm{ finite closed }}} \sum_{\substack{\overline{\alpha} \in H^0(Z,L^k)^\vee\\ \textrm{primitive}}} S_Z(\overline{\alpha})^s \overline{\psi(\overline{\alpha}(f))}  = \prod_{v \in \abs{C}} \sum_{m=0}^{\infty} \sum_{\substack{\overline{\alpha} \in H^0(m[v],L^k)^\vee\\ \textrm{primitive}}} S_{m[v]} (\overline{\alpha})^s \overline{\psi(\overline{\alpha}(f))} .\] Here Lemma \ref{major-local-bound} implies, since $s\geq 5$, that the product of sums is absolutely convergent and so the sum is absolutely convergent and thus the formal manipulation is analytically valid. The result then follows from Lemma \ref{lv-es}. \end{proof}

\begin{lemma}\label{lem-two-parts} Assume $s\geq 5$, $k$ is coprime to $p$, and either $k>2$ or $k=2$ and $g>0$. For $f \in H^0(C, L^k)$ we have 
\[  \#\{ \mathbf a \in H^0( C, L )^s \mid  \sum_{i=1}^s a_i^k = f\}  = \operatorname{MT}_e + \operatorname{MA} - \operatorname{ESS}\] where
 
\begin{equation}\label{minor-symbol} \operatorname{MA} =  \frac{1}{ q^{ke+1-g}}\sum_{ \substack{ \alpha \in H^0(C,L^k)^\vee \\ \deg \alpha > e-2g+1 } } S(\alpha)^s \overline{\psi(\alpha(f))} \end{equation} 
and \begin{equation}\label{major-symbol} \operatorname{ESS}  =  \frac{q^{s (e-g+1)} }{ q^{ke+1-g}} \sum_{ \substack{ Z \subset C \\ \textrm{finite closed} \\ \deg Z > e-2g+1 } }\sum_{\substack{ \overline{\alpha} \in H^0(Z,L^k)^\vee \\ \textrm{primitive}} } S_Z(\overline{\alpha})^s  \overline{\psi(\overline{\alpha}(f))} .\end{equation} \end{lemma}

Here ``$\operatorname{MA}$'' stands for ``minor arcs'' and ``$\operatorname{ESS}$'' stands for ``excess singular series''.

\begin{proof} We apply \eqref{circle-setup} to obtain
\[  \#\{ \mathbf a \in H^0( C, L )^s \mid  \sum_{i=1}^s a_i^k = f\}  \] \[= \frac{1}{ q^{ke+1-g}} \sum_{\substack{  \alpha \in H^0(C,L^k)^\vee \\ \deg \alpha \leq e-2g+1}}  S(\alpha)^s \overline{\psi(\alpha(f))}+  \frac{1}{ q^{ke+1-g}} \sum_{\substack{  \alpha \in H^0(C,L^k)^\vee \\ \deg \alpha > e-2g+1}}  S(\alpha)^s \overline{\psi(\alpha(f))}\] where the second term is $\operatorname{MA}$.

We have $\frac{ke}{2}- g+1  = e -2g+1 + \frac{ (k-2) e}{2}+g $ which since $e>0$ and either $k>2$ or $g>0$ implies that $e-2g+1 < \frac{ke}{2}-2g+1$. Thus if $\deg \alpha \leq e-2g+1$ then there exists a unique closed subscheme $Z$ of minimal degree through which $\alpha$ factors, and $\overline{\alpha}$ is certainly primitive.  Conversely, if $Z$ is a closed subscheme of degree $\leq e-2g+1$ and $\overline{\alpha} \in H^0(Z,L^k)^\vee$ is primitive then $\alpha \colon H^0(C,L^k) \to H^0(Z, L^k) \to \mathbb F_q$ has degree $\leq e-2g+1$ and, by Lemma \ref{mhl-two}, $Z$ is the minimal subscheme through which $\alpha$ factors. Thus
\[ \frac{1}{ q^{ke+1-g}} \sum_{\substack{  \alpha \in H^0(C,L^k)^\vee \\ \deg \alpha \leq e-2g+1}}  S(\alpha)^s \overline{\psi(\alpha(f))}\]
\[ =\frac{q^{s(e-g+1)} }{ q^{ke+1-g}} \sum_{\substack{  Z \subset C\\  \textrm{finite closed} \\ \deg Z \leq e-2g+1} } \sum_{\substack{\overline{\alpha} \in H^0(Z,L^k)^\vee\\ \textrm{primitive}}} S_Z(\overline{\alpha})^s  \overline{\psi(\overline{\alpha}(f))} \]
\[=\frac{q^{s(e-g+1)} }{ q^{ke+1-g}} \sum_{\substack{  Z \subset C \\ \textrm{finite closed}} } \sum_{\substack{\overline{\alpha} \in H^0(Z,L^k)^\vee\\ \textrm{primitive}}}S_Z(\overline{\alpha})^s  \overline{\psi(\overline{\alpha}(f))}  -  \frac{q^{s (e-g+1)} }{ q^{ke+1-g}} \sum_{ \substack{ Z \subset C \\ \textrm{finite closed} \\ \deg Z > e-2g+1 } }\sum_{\substack{ \overline{\alpha} \in H^0(Z,L^k)^\vee \\ \textrm{primitive}} } S_Z(\overline{\alpha})^s  \overline{\psi(\overline{\alpha}(f))} \] where the second term is $\operatorname{ESS}$. Plugging in Lemma \ref{lem-major-mt} then gives the desired statement. \end{proof}

To prove Theorem \ref{intro-curve}, it follows from Lemma \ref{lem-two-parts} that it suffices to give an upper bound for $\operatorname{MA}$ and $\operatorname{ESS}$. The minor arcs $\operatorname{MA}$ will be handled in the next section. $\operatorname{ESS}$ will be handled in the next lemma.

\begin{lemma}\label{easier-euler} Assume $k\geq 2$  and $k$ is coprime to $p$. For each $\delta <  \frac{s-\max(k,3) -1}{k} $ we have
\[ \operatorname{ESS}= \frac{q^{s (e-g+1)} }{ q^{ke+1-g}}   O_{s,k,\delta} (  (1+ q^{-\frac{1}{2}} )^{O_{s,k}(g)} q^{ - \delta (e-2g+2)} ) = \frac{q^{s (e-g+1)} }{ q^{ke+1-g}} O_{s,k,\delta,g} ( q^{ - \delta (e-2g+2)} ) .\] \end{lemma}

\begin{proof} It suffices to show 
\begin{equation}\label{easier-euler-to-show} \sum_{ \substack{ Z \subset C\\  \textrm{finite closed} \\ \deg Z > e-2g+1 } }\sum_{\substack{ \overline{\alpha} \in H^0(Z,L^k)^\vee \\ \textrm{primitive}} }\abs  {S_Z(\overline{\alpha})}^s = O_{s,k,\delta} (  (1+ q^{-\frac{1}{2}} )^{O_{s,k}(g)} q^{ - \delta (e-2g+2)} ). \end{equation}
This gives the first bound of the statement, and the weaker $O_{s,k,\delta, g}(q^{ - \delta (e-2g+2)})$ claim is then immediate since $  (1+ q^{-\frac{1}{2}} )^{O_{s,k}(g) }= O_{s,k,\delta, g}(1)$. 

The quantity to be bounded in \eqref{easier-euler-to-show} is the sum over $d > e-2g+1$ of the coefficient of $u^d$ in
\[ \sum_{ \substack{ Z \subset C\\ \textrm{finite closed} }} u^{\deg Z} \sum_{\substack{ \overline{\alpha} \in H^0(Z,L^k)^\vee \\ \textrm{primitive}} }\abs{ S_Z(\overline{\alpha})}^s. \] 
The sum of the coefficients of $u^d$ in this power series for $d> e-2g+1$ is at most $q^{ - \delta (e-2g+2)}$ times the value of this power series at $u =q^{\delta}$.  So it suffices to show that the value of the power series at $u= q^{\delta}$ is  $ O_{s,k,\delta} (  (1+ q^{-\frac{1}{2}} )^{O_{s,k}(g)})$. Our first few bounds will use only that $u$ is a positive real, but after this we will use the assumption  $\delta < \frac{s-\max(k,3)-1}{k} $.

By Lemma \ref{major-factorization} we have
\[ \sum_{ \substack{ Z \subset C \\ \textrm{finite closed} }} u^{\deg Z} \sum_{\substack{ \overline{\alpha} \in H^0(Z,L^k)^\vee \\ \textrm{primitive}} }\abs{ S_Z(\overline{\alpha})}^s= \prod_{v \in \abs{C}} \sum_{m=0}^{\infty} u^{m \deg v}  \sum_{\substack{ \overline{\alpha} \in H^0(m[v] ,L^k)^\vee \\ \textrm{primitive}} }\abs{ S_{m[v]} (\overline{\alpha})}^s.\]
By Lemma \ref{major-local-bound} we have 
\[ \sum_{m=0}^{\infty} u^{m \deg v}  \sum_{\substack{ \overline{\alpha} \in H^0(m[v] ,L^k)^\vee \\ \textrm{primitive}} }\abs{ S_{m[v]} (\overline{\alpha})}^s\]
\[ \leq \sum_{m=0}^{\infty} u^{m \deg v} q^{m \deg v} \begin{cases} \frac{1}{q^{ s\lceil \frac{m}{k} \rceil \deg v }}& \textrm{if }m \not \equiv 1 \bmod k \\  \frac{(k-1)^s}{ q^{ s\left(\frac{m-1}{k} + \frac{1}{2}\right) \deg v}}  & \textrm{if }m\equiv 1\bmod k\end{cases} \]
\[ = \frac{1}{1- u^{k \deg v} q^{(k-s) \deg v}}  \Bigl( 1 +  u^{\deg v}  (k-1)^s q^{(1- \frac{s}{2} )\deg v } +  \sum_{m=2}^{k-1} u^{m \deg v} q^{(m-s) \deg v} \Bigr)\]
\[ \leq  \frac{1}{ (1- u^{\deg v} q^{(1-\frac{s}{2})\deg v} )^{(k-1)^s}  } \prod_{j=2}^{k}  \frac{1}{ 1- u^{j \deg v} q^{(j-s)  \deg v} }\] (where we used the fact that the $(m+k)$th term of the sum over $m$ is $ u^{k \deg v} q^{k \deg v - s \deg v}$ times the $m$th term to pull out a geometric series and then, in the last line, used the inequality $ 1 + \sum_{i} a_i \leq \prod_i \frac{1}{1-a_i}$ for nonnegative $a_i$ applied to a sequence $a_i$ consisting of  $(k-1)^s$ copies of $ u^{\deg v} q^{(1-\frac{s}{2}) \deg v}$ as well as $u^{j \deg v} q^{(j-s)\deg v}$ for $j$ from $2$ to $k-1$, and then absorbed the factor $\frac{1}{1- u^{k \deg v} q^{(k-s) \deg v}} $ as the last term of the product).

Thus\[  \prod_{v \in \abs{C}} \sum_{m=0}^{\infty} u^{m \deg v}  \sum_{\substack{ \overline{\alpha} \in H^0(m[v] ,L^k)^\vee \\ \textrm{primitive}} }\abs{ S_{m[v]} (\overline{\alpha})}^s\]
\[ \leq \prod_{v\in \abs{C}} \Bigl(  \frac{1}{ (1- u^{\deg v} q^{(1-\frac{ s}{2} )\deg v} )^{(k-1)^s} } \prod_{j=2}^{k}  \frac{1}{ 1- u^{j \deg v} q^{(j-s) \deg v} }\Bigr) \]
\[ \leq \zeta_{C} (  u q^{ 1- \frac{s}{2}})^{(k-1)^s}  \prod_{j=2}^{k} \zeta_C ( u^j q^{j-s} ) \leq   \frac{(  1+ u q^{ \frac{3}{2} - \frac{s}{2}})^{2g(k-1)^s} } {(  1- u q^{ 1- \frac{s}{2}})^{(k-1)^s} (  1- u q^{ 2- \frac{s}{2}})^{(k-1)^s} } \prod_{j=2}^{k} \frac{ (1+ u^j q^{j + \frac{1}{2}-s})^{2g} }{ (1-u^j q^{j-s}) (1- u^j q^{j+1-s})}.\]
by Lemma \ref{rh-bound}. The assumption $\delta <  \frac{s-\max(k,3) -1}{k}$ implies that $\delta < \frac{s-k-1}{k}$ so that $j \delta < s-j-1$ for all $j \leq k$ and also that  $\delta <  \frac{s-4}{k} \leq \frac{s-4}{2} = \frac{s}{2}-2$. Combining these, we see that all the terms in the denominator have the form $(1 - q^{ f})$ with $f<0$ depending on $s,k,\delta$ and so are lower bounded by $(1-2^{f})$ which depends only on $s,k,\delta$. Similarly, the terms in the numerator are bounded by $1+ q^{-1/2}$ and the number of terms appearing is $2g ((k-1)^s + (k-1)) =O_{s,k}(g)$. This gives the desired bound
$O_{s,k,\delta} (  (1+ q^{-\frac{1}{2}} )^{O_{s,k}(g)} q^{ - \delta (e-2g+2)} ) $ for the value of the power series.
  \end{proof}
  
  Finally, we prove a lower bound on the main term (equivalently, upper bound on the inverse of the main term) that will be needed for Theorem \ref{intro-curve}:
  
  \begin{lemma}\label{lv-lower-bound} Assume that $k$ is coprime to $p$, $k\geq 2$, $s>k+1$,  $s\geq 5$, and $q > (k-1)^4$.  We have
  \[ \Bigl( \prod_{\substack{v \in \abs{C} }} \ell_v(f)\Bigr)^{-1} = O_{s,k,g}(1).\]
  \end{lemma}
  
  The bound $q> (k-1)^4$ is not optimal, but certainly one must take $q > (k-1)^2$ since if $q$ is a perfect square, $k = \sqrt{q}+1$, and $v$ is a place of degree $1$, then $\ell_v(f)=0$ for $f$ whose restriction to that place does not lie in the subfield $\mathbb F_{\sqrt{q}}$.
  
  \begin{proof} By Lemmas \ref{lv-es} and \ref{major-local-bound} we have
  \[ \ell_v(f)=  \sum_{m=0}^{\infty} \sum_{\substack{\overline{\alpha} \in H^0(m[v],L^k)^\vee\\ \textrm{primitive}}} S_{m[v]} (\overline{\alpha})^s \overline{\psi(\alpha(f))}\]
  \[ \geq 1  - \sum_{ \substack{m >0 \\ m \not \equiv 1 \bmod k}}  (q^{m \deg v} -q^{(m-1)\deg v})\frac{1}{q^{s\lceil \frac{m}{k} \rceil \deg v }} - \sum_{\substack{m >0 \\ m  \equiv 1 \bmod k}} (q^{m \deg v} -q^{(m-1)\deg v}) \frac{(k-1)^s}{ q^{s \left(\frac{m-1}{k} + \frac{1}{2}\right) \deg v}} \]
  \[ = 1  -  \frac{1}{ 1-  q^{(k-s)\deg v}}  (q^{\deg v}-1) \Bigl(  \frac{(k-1)^s}{ q^{\frac{s}{2} \deg v}} +\sum_{m=2}^k \frac{q^{ (m-1) \deg v} }{q^{s \deg v}} \Bigr). \]
 
 We observe that $\frac{1}{ 1-  q^{(k-s)\deg v}}  (q^{\deg v}-1) \leq q^{\deg v}$ and that $\sum_{m=2}^k \frac{q^{ (m-1) \deg v} }{q^{s \deg v}} = O ( q^{ (k-1-s)\deg v})$ so that  \[ \ell_v(f) = 1 + O_{s,k} ( q^{ (1-\frac{s}{2}) \deg v} + q^{ (k-s) \deg v}) = 1 + O_{s,k} ( q^{-\frac{3}{2} \deg v} ).\] Thus, as long as $\ell_v(f)$ is bounded away from $0$ uniformly in $f$ and $q$, we have \[ \ell_v(f)^{-1} = 1 + O_{s,k} ( q^{-\frac{3}{2} \deg v} ).\] The product of this over $v$ is bounded by $\zeta_C ( q^{-3/2})^{O_{s,k}(1) } = O_{s,k,g}(1)  $  so it remains to check $\ell_v(f)$ is bounded away from $0$ uniformly in $f$ and $q$.
 
 To do this, we bound $\sum_{m=2}^k \frac{q^{ (m-1) \deg v} }{q^{s \deg v}} $ more precisely by $ \frac{1}{ 1- q^{-\deg v} } q^{ (k-1-s)\deg v }  \leq \frac{ 1}{ q^{\deg v} (q^{\deg v}-1)} $.
 
 On the other hand we have $(k-1)^{s} \leq q^{ \frac{s}{4} }$ so $ \frac{ (k-1)^s}{ q^{ \frac{s}{2} \deg v}} \leq \frac{1}{ q^{ \frac{s}{4} \deg v}} $. Thus
 
  \[ \ell_v(f) \geq 1  -  \frac{1}{ 1-  q^{(k-s)\deg v}}  (q^{\deg v}-1) \Bigl(  \frac{(k-1)^s}{ q^{\frac{s}{2} \deg v}} +\sum_{m=2}^k \frac{q^{ (m-1) \deg v} }{q^{s \deg v}} \Bigr) \] \[ \geq 1 - q^{ \deg v}  \Bigl(  \frac{1 }{ q^{\frac{s}{4} \deg v}} + \frac{1}{ q^{\deg v}( q^{\deg v}-1)}  \Bigr)   \geq 1 -  \Bigl( \frac{1}{ q^{ \frac{1}{4} \deg v}} + \frac{1}{q^{\deg v}-1} \Bigr) \]
 which is bounded away from $0$ uniformly in $q^{\deg v}$ as long as $q^{\deg v} \geq 4.3 $, which happens as long as $q \geq 4.3$ , which follows from the assumption $q > (k-1)^4$ if $k\geq 3$. If $k=2$, this does not follow, but we have $(k-1)^s=1$ so we can replace the last line by
   \[ \geq 1 - q^{ \deg v}  \Bigl(  \frac{1 }{ q^{\frac{s}{2} \deg v}} + \frac{1}{ q^{\deg v}( q^{\deg v}-1)}  \Bigr)   \geq 1 -  \Bigl( \frac{1}{ q^{ \frac{3}{2} \deg v}} + \frac{1}{q^{\deg v}-1} \Bigr)  \] which is bounded away from zero uniformly as long as $q^{ \deg v}\geq 2.4$ which happens because $k$ is prime to $p$ and so $q^{\deg v} \geq q \geq p \geq 3$.\end{proof}
   
  Finally, we remark briefly on how to handle the case $k=2, g=0$ dropped in Lemma \ref{lem-two-parts}. This case was essentially handled in \cite[Corollary 1.2]{Yamagishi2016}, but without an explicit power savings error term. We explain the same argument in our notation here. Since this case is easier than the other cases, we will be brief.
  
  \begin{lemma}\label{easy-quadratic-case} Assume $s \geq 5$, $k=2$ is coprime to $p$, and $g=0$. For each $\delta < \frac{s-4}{2}$ we have 
  \[  \#\{ \mathbf a \in H^0( C, L )^s \mid  \sum_{i=1}^s a_i^2= f\} -  q^{ e(s-2) +(s-1)  } \prod_{\substack{v \in \abs{C} }} \ell_v(f)  = O_{s,\delta} \Bigl( \frac{q^{s (e+1)} }{ q^{2e+1 }}  q^{-\delta(e+1)}\Bigr). \] 
  \end{lemma}

\begin{proof}  We fix a degree one point $\infty$ of $C \cong \mathbb P^1$ and claim that when $k=2$ and $g=0$, each $\alpha$ factors through a unique closed subscheme $Z$ of minimal degree such that $\deg (Z \cup \{\infty\})\leq e+1$, where union of subschemes corresponds to intersection of ideals. This existence can be obtained from classical Dirichlet approximation \cite[Lemma 3]{R1974}, or from observing that it suffices to check that the restriction of $\alpha$ from $H^0(C, L^2)$ to $H^0(C, L^2(-\infty))$ factors through a subscheme $Z'$ of degree $\leq e$ and take $Z = Z' + [\infty]$ (where sum of subschemes corresponds to multiplication of ideals), and then following the proof of Lemma \ref{mhl-one}. The uniqueness  follows from the argument of Lemma \ref{mhl-two} but using the bound
\[ \deg (Z_1 \cup Z_2) \leq \deg (Z_1 \cup Z_2 \cup \{\infty\})-1   \leq \deg(Z_1 \cup \{\infty\} ) + \deg(Z_2 \cup \{\infty\}) -1 \] instead of $\deg(Z_1 \cup Z_2) \leq \deg(Z_1) + \deg(Z_2)$ (or can also be obtained classically).

The existence and uniqueness immediately gives the following analogue of Lemma \ref{lem-two-parts}, where the minor arc term has disappeared and the other term is adjusted slightly:
  \[  \#\{ \mathbf a \in H^0( C, L )^s \mid  \sum_{i=1}^s a_i^2= f\} -  q^{ e(s-2) +(s-1)  } \prod_{\substack{v \in \abs{C} }} \ell_v(f)   \] 
\[ = -  \frac{q^{s (e+1)} }{ q^{2e+1}} \sum_{ \substack{ Z \subset C \\ \textrm{finite closed} \\ \deg Z \cup \{\infty\}> e+1 } }\sum_{\substack{ \overline{\alpha} \in H^0(Z,L^2)^\vee \\ \textrm{primitive}} } S_Z(\overline{\alpha})^s  \overline{\psi(\overline{\alpha}(f))} .\]

We can follow the proof of Lemma \ref{easier-euler} to bound this adjusted error term by $O( \frac{q^{s (e+1)} }{ q^{2e+1 }} q^{ -\delta (e+1)})$ for $\delta < \frac{s-4}{2}$, by considering the sum of the coefficients of $u^d$ for $d>e$ in an Euler product whose term for $v\neq \infty$ is unchanged and whose term for $v = \infty$ is \[\sum_{m=0}^{\infty} u^{ \max(m-1,0)} \sum_{ \substack{\overline{\alpha} \in H^0( m[\infty], L^2)^\vee \\ \textrm{primitive}}} \abs{ S_{m [\infty]}(\overline{\alpha})}^s.\] Applying Lemma \ref{major-local-bound} bounds this local factor by $\frac{1}{ 1- u^{\deg v} q^{\deg v- \frac{s}{2} \deg v}}$ for $v \neq \infty$ or $ 1 + \frac{q^{1-\frac{s}{2} }}{ 1- uq^{1- \frac{s}{2} }}$ for $v =\infty$. The product of this local bound is 
  \[ \prod_{v \in \abs{\mathbb P^1\setminus \{\infty\}}} \frac{1}{ 1- u^{\deg v} q^{\deg v- \frac{s}{2} \deg v}}  \times \Bigl( 1 + \frac{q^{1-\frac{s}{2} }}{ 1- uq^{1- \frac{s}{2} }}\Bigr) = \frac{1}{ 1-  u q^{ 2 - \frac{s}{2}}}  \Bigl( 1 + \frac{q^{1-\frac{s}{2} }}{ 1- uq^{1- \frac{s}{2} }}\Bigr). \]
For $u = q^{\delta}$ with $\delta< \frac{s-4}{2}$, all terms in both denominators are bounded away from $0$ by a constant depending only on $s,\delta$ and $q^{1-\frac{s}{2}} \leq 1$ so this expression is $O_{s,\delta}(1)$, giving the desired bound as in Lemma \ref{easier-euler}.\end{proof}

\section{Minor arcs via the singular locus of exponential sums}\label{s-minor}

For the minor arcs, it is crucial that we think of $H^0(C, L)$ as an algebraic variety, whose $R$-points are $H^0(C,L)\otimes_{\mathbb F_q} R$ for any $\mathbb F_q$-algebra $R$. In particular, this is an affine space of dimension $e+1-g$.  

For $\alpha \in H^0(C, L^k)^\vee$, let \[\operatorname{Sing}_\alpha = \{ a \in H^0(C, L) \mid  \alpha (a^{k-1} b)=0 \textrm{ for all } b\in H^0(C, L) \},\] which is a closed subscheme of $H^0(C, L)$ cut out by polynomials.

\begin{lemma}\label{minor-katz} Assume $k$ is coprime to $p$. Then for $S(\alpha)$ defined in \eqref{S1-def} we have
\[ \abs{ S(\alpha) } \leq 3 (k+1)^{e+1-g}  q^{ \frac{ e+1-g + \dim \operatorname{Sing}_\alpha}{2}} .\]
\end{lemma}
\begin{proof} This can be obtained from results of Katz \cite{Katz1999}, more specifically by a slight modification of the proof of \cite[Theorem 4(1)]{Katz1999}. To begin with, we explain Katz's setup and the assumptions of \cite[Theorem 4(1)]{Katz1999} and explain why they hold in our setting.

Fix a finite field $\mathbb F_q$ (which Katz calls $k$), positive integers $N$ and $r$, and a list of $r$ positive integers $D_1,\dots, D_r$. Inside $\mathbb P^N_{\mathbb F_q}$, consider a closed subscheme $X$ defined by a set of $r$ homogeneous equations of degrees $D_1,\dots, D_r$. Assume (``hypothesis (H1)'') that $X \otimes_{\mathbb F_q} \overline{\mathbb F_q}$ is irreducible and integral of dimension $n\geq 1$. Consider a nonzero linear form $L \in H^0(X, \mathcal O(1))$, an integer $d$, and a nonzero section $H\in H^0(X, \mathcal O(d))$. Both $L$ and $H$ may and will be viewed as hypersurfaces in $X$.  Assume also (``hypothesis (H2)'') that the scheme-theoretic intersection $X \cap L \cap H$ has dimension $n-2$. Let $\delta$ be the dimension of the singular locus of $X \cap L \cap H$ and $\epsilon$ the dimension of the singular locus of $X \cap L$.  Let $V$ be the complement of $L$ in $X$ and let $f$ be the polynomial function on $V$ defined by $\frac{H}{L^d}$. Fix $\psi \colon \mathbb F_q \to \mathbb C^\times$ a nontrivial additive character. If $\epsilon \leq \delta$, \cite[Theorem 4(1)]{Katz1999} states that we have
\begin{equation}\label{katz-theorem-4} \abs{ \sum_{x \in V(\mathbb F_q)} \psi (f(x))} \leq  (4\sup (1+D_1,\dots, 1+D_r, d) +5)^{ N+r}  q ^{ \frac{n+1+\delta}{2}} .\end{equation} 

We now explain how to apply the above to our setting. For our case $N = e+2-g$, $r=1$, $D_1=1$, and $X= \mathbb P^{e+1-g}$ defined by a single non-zero linear form. (The linear form is only included to satisfy the assumption $r\geq 1$ of Katz's setup, which seems not to be strictly necessary, but including it leads to no significant loss.) The dimension of $X$ is $n=e+1-g$. Hypothesis (H1) on $X$ is trivially satisfied. We take $L$ an arbitrary linear form so that the locus $V$ where $L$ is nonvanishing is $\mathbb A^{e+1-g}$. We identify this affine space with $H^0(C,L)$. On $H^0(C,L)$, the map $a \mapsto \alpha(a^k)$ may be identified with a homogeneous polynomial function of degree $k$. A polynomial function of degree $k$ on affine space gives by homogenization a section of $\mathcal O(k)$ on projective space, such that dividing the section by $L^k$ produces the original polynomial function. We take $d=k$ and $H$ to be this section. Hence $f = \frac{H}{L^d}$ is the function $a \mapsto \alpha(a^k)$.

Let us check that $\operatorname{Sing}_\alpha$ is the set of vectors where $f$ and its derivative in each direction vanish.  To check this, note that the derivative of $f$ in the direction given by a vector $b$ is \[ \frac{d}{du} \alpha ( (a+ ub)^k) = \alpha ( \frac{d}{du} (a+ub)^k)= \alpha ( k a^{k-1} b) = k \alpha(a^{k-1} b).\] Since $p \nmid k$, this is zero for all $b$ if and only if $ \alpha(a^{k-1} b)=0$ for all $b$. Furthermore, if $ \alpha(a^{k-1} b)=0$ for all $b$ then clearly $\alpha (a^k)=0$, verifying the characterization of $\operatorname{Sing}_\alpha$.

If $f$ is identically zero, it follows that $\operatorname{Sing}_\alpha$ is all of $H^0(C,L)$ so $\dim \operatorname{Sing}_\alpha= e+1-g$ and the stated bound follows from the trivial bound without applying \cite[Theorem 4(1)]{Katz1999}, so we may assume $f$ is not identically zero.

To check hypothesis (H2), $X \cap L$ is the projective space $\mathbb P^{e-g} =\mathbb P^{n-1}$, so it suffices to check that $H$ is nontrivial when restricted to this projective space, i.e. the leading coefficient of the polynomial function $f$ is nonzero. Since $f$ is homogeneous, this happens if and only if $f$ is not identically zero, which we have assumed.

We can now calculate $\epsilon$ and $\delta$. Since $X \cap L$ is just projective space, it is smooth, so $\epsilon=-1$. Since $X \cap L \cap H$ is the hypersurface in $\mathbb P^{e-g}$ defined by the leading term of the polynomial $f$, its singular locus is the set of nonzero vectors up to scaling where that leading term and its derivative in each direction vanish. Since $f$ is homogeneous of degree $k$, this is simply the set of nonzero vectors up to scaling where $f$ and its derivative in each direction vanish. In other words, the singular locus of $X \cap L \cap H$ is the set of nonzero vectors of $\operatorname{Sing}_\alpha$, modulo scaling. Since modding out by scaling reduces the dimension by $1$, this implies that $\delta =\dim \operatorname{Sing}_\alpha-1$.

We have $\epsilon = -1 \leq \dim \operatorname{Sing}_\alpha-1 =\delta$ so the final assumption $\epsilon \leq \delta$ of \cite[Theorem 4(1)]{Katz1999} is satisfied. 

As noted, $V$ is $\mathbb A^{e+1-g}$, so that $V(\mathbb F_q)$ may be identified with $H^0(C,L)$, and $f(a) = \alpha(a^k)$. Thus $S(\alpha) = \sum_{x \in V(\mathbb F_q)} \psi (f(x))$. Plugging this and  our fixed values of $N,r, D_1, n,\delta$ into  \eqref{katz-theorem-4} gives 
\[ \abs{ S(\alpha) } \leq (4k+5)^{e+3-g}  q^{ \frac{ e+1-g + \dim \operatorname{Sing}_\alpha}{2}} .\]

The constant $(4\sup (1+D_1,\dots, 1+D_r, d) +5)^{ N+r} $ arises in the proof as a bound for the total degree of the $L$-function of the exponential sum. (This step of the argument is given in \cite[Two paragraphs after the statement of Theorem 4 on p. 879]{Katz1999}, with the proof \cite[p. 892]{Katz1999} devoted to bounding the zeroes and poles of the $L$-function).

However, later work of Katz may be used to bound the total degree of the $L$-function by $3 (k+1)^{e+1-g}$. Indeed  \cite[Theorem 10]{Katz2001} states that for a polynomial $f$ of degree at most $d$ in $N$ variables over a finite field $F$ of characteristic $p$, the sum $\sigma_c( \mathbb  A^N_{\overline{F}}, \mathcal L_{\psi(f)})$ of the dimensions of the compactly supported cohomology groups $H^i_c(\mathbb  A^N_{\overline{F}}, \mathcal L_{\psi(f)})$ is at most $3(d+1)^N$. Applying this to $F=\mathbb F_q$ and  $f(a)= \alpha ( a^k)$ we will again take $d= k$ but now take $N= e+1-g$, the dimension of the relevant affine space. The Lefschetz fixed point formula shows that the $L$-function of the exponential sum $\sum_{x \in \mathbb A^N(F)} \psi(f(x))$ is \[\prod_i \det( 1-  u \operatorname{Frob}_F ,H^i_c(\mathbb  A^N_{\overline{F}}, \mathcal L_{\psi(f)}))^{ (-1)^i}.\]   Since $\det( 1-  u \operatorname{Frob}_F ,H^i_c(\mathbb  A^N_{\overline{F}}, \mathcal L_{\psi(f)}))$ is a polynomial in $u$ of degree $\dim H^i_c(\mathbb  A^N_{\overline{F}}, \mathcal L_{\psi(f)})$, this implies that the $L$-function has total degree at most  
\[ \sum_i \dim H^i_c(\mathbb  A^N_{\overline{F}}, \mathcal L_{\psi(f)})= \sigma_c( \mathbb A^N_{\overline{F}}, \mathcal L_{\psi(f)}) \leq 3 (d+1)^N = 3 (k+1)^{e+1-g}\] which gives the bound 
\[ \abs{ S(\alpha) } \leq  3(k+1)^{e+1-g}  q^{ \frac{ e+1-g + \dim \operatorname{Sing}_\alpha}{2}} .\qedhere\] 
 \end{proof} 
 
 Recall that $K_C $ is the line bundle of 1-forms on $C$. For $v$ a point of a finite closed subscheme $Z$, we have a residue map from sections of $K_C(Z)$ in a neighborhood of $v$ to $\mathbb F_q$. For $\pi$ a uniformizer at $v$, a section of $K_C(Z)$ in a neighborhood of $v$ may be written as $(f + \sum_{i=1}^m c_i \pi^{-i} ) d\pi$ where $f$ is a regular function in a neighborhood of $v$ and $c_i$ lie in the residue field $\kappa_v$ of $C$ at $v$. We then define $\operatorname{res}_v( (f + \sum_{i=1}^m c_i \pi^{-i} ) d\pi) = \operatorname{Tr}_{\kappa_v}^{\mathbb F_q} c_1$ where $\operatorname{Tr}$ is the trace. It is straightforward to check that, despite appearances, this does not depend on the choice of $\pi$. 
 
 Since $\operatorname{res}_v$ vanishes on sections of $K_C$, as for sections of $K_C$ we have $c_i=0$ for all $i>0$, $\operatorname{res}_v$ descends to a function on sections of $K_C(Z) / K_C$ and therefore to a function on $H^0(Z, K_C(Z))$, which we also denote by $\operatorname{res}_v$. 

\begin{lemma}\label{linear-is-residue}  Let \[\operatorname{res} \colon  H^0(Z, L^k) \times H^0 (Z, K_C (Z)  \otimes L^{-k}) \to \mathbb F_q \] be the bilinear map given by \[ \operatorname{res}(\beta, f) = \sum_{v \in Z} \operatorname{res}_v(\beta f).\]

Then $\operatorname{res}$ is a perfect pairing.

For $\overline{\alpha} \in H^0(Z, L^k)^\vee$, let  $\tilde{\alpha} \in H^0 (Z, K_C (Z)  \otimes L^{-k}) $ be the unique element that induces the linear form $\overline{\alpha}$ under this perfect pairing.

Finally, if  $\overline{\alpha}$ is primitive then $\tilde{\alpha}$ is invertible in the sense that there is a section $\tilde{\alpha}^{-1} \in H^0 (Z, K_C^{-1} (-Z)  \otimes L^{k}) $ with $\tilde{\alpha} \tilde{\alpha}^{-1} = 1\in H^0(Z, \mathcal O_Z)$. \end{lemma}

\begin{proof}  Since both $H^0(Z, L^k) $ and $ H^0 (Z, K_C (Z)  \otimes L^{-k})$   can be expressed as a direct sum over points of $v$, and the pairing of the summand of $H^0(Z, L^k) $ associated to $v_1$ with the summand of $ H^0 (Z, K_C (Z)  \otimes L^{-k})$ associated to $v_2$ vanishes unless $v_1 =v_2$, it suffices to check the pairing is perfect upon restriction to the summands associated to $v$ for each point $v$, i.e. it suffices to assume $Z = m[v]$.

To check a $\kappa_v$-linear pairing composed with the trace $ \operatorname{Tr}_{\kappa_v}^{\mathbb F_q}$ from $\kappa_v$ to $\mathbb F_q$ is perfect, it suffices to check that the original $\kappa_v$-linear pairing is perfect, so we may ignore the trace in the definition of residue.

In this case, we work in the local ring $\mathcal O_{C_v}$ with uniformizer $\pi_v$. We can fix a local trivialization of $L$ at $v$. This lets us write $H^0(m[v], L^k) = \mathcal O_{C_v} / \pi_v^m $ which has basis $1, \pi_v,\dots,\pi_v^{m-1}$ over $\kappa_v$. Similarly,  $H^0(m[v], K_C(m[v]) \otimes L^{-k}) = \pi_v^{-m} \mathcal O_{C_v} d \pi_v /   \mathcal O_{C_v} d \pi_v $ which has basis $\pi_v^{-1} d \pi_v ,\dots, \pi_v^{-m} d \pi_v$. The residue of the product $\pi_v^i \cdot \pi_v^{-j} d \pi_v$ is $1$ if $j=i+1$ and $0$ if $j>i+1$, which implies the pairing is perfect.

Finally, $\tilde{\alpha}$ fails to be a generator if and only if, restricted to the summand associated to some $v\in \abs{Z}$, the coefficient of $\pi_v^{-m}d \pi_v$ in $\tilde{\alpha}$ is zero. By the above calculation of the residue pairing, the coefficient is the residue of $\pi_v^{m-1} \tilde{\alpha}$ and thus is $\overline{\alpha} (\pi_v^{m-1})$. If this vanishes, then $\overline{\alpha}$ factors through the subscheme $Z'$ which corresponds to the divisor $Z - [v]$ and hence is not primitive. \end{proof}

\begin{lemma}\label{exists-c} Fix $\alpha \in H^0(C, L^k)^\vee$. Assume $\alpha$ factors through $Z$ and fix $\tilde{\alpha}$ as in Lemma \ref{linear-is-residue}.

Let $\rho_Z$ be the restriction map $H^0( C, K_C(Z) \otimes L^{-1})\to H^0(Z, K_C(Z) \otimes L^{-1})$ and let $\kappa_{\tilde{\alpha}}$ be the map $H^0(C, L) \to H^0(Z, K_C(Z) \otimes L^{-1})$ defined by  $a \mapsto \tilde{\alpha} a^{k-1} $.

Then $\operatorname{Sing}_\alpha$ is the preimage under $\kappa_{\tilde{\alpha}}$ of the image of $\rho_Z$. In other words, we have $a\in  \operatorname{Sing}_\alpha$ if and only if there exists $c \in H^0 (C, K_C(Z) \otimes L^{-1} ) $ such that $c \mid_Z = \tilde{\alpha} a^{k-1}$.

Furthermore, $\rho_Z$ is injective, i.e. if such a $c$ exists then it is unique. \end{lemma}

\begin{proof} 


Suppose such a $c$ exists. Then for $b\in H^0(C,L)$, 
\[  \alpha(a^{k-1} b)= \operatorname{res} ( a^{k-1} b, \tilde{\alpha} ) = \sum_{v\in Z} \operatorname{res}_v( \tilde{\alpha} a^{k-1} b) = \sum_{v \in Z} \operatorname{res}_v( cb) = 0\] because $cb$ 
is a global section of $H^0(C, K_C(Z))$, thus $cb$ is a meromorphic differential form with poles only in $Z$, hence, by the residue theorem, the sum of the residues of $cb$ over the points of $Z$ vanishes. This verifies ``if".



Next consider the short exact sequence $L (-Z) \to L \to L\mid_Z $ which induces a long exact sequence in cohomology \[H^0( C, L(-Z)) \to H^0(C, L) \to H^0(Z, L)\to H^1(C,L(-Z)) \to H^1(C, L).\] Hence the space of linear forms on $H^0(Z,L)$ that vanish on $H^0(C,L)$ is dual to the kernel of $H^1(C,L(-Z)) \to H^1(C, L)$, thus by Serre duality~\cite[III, Corollary 7.7]{Hartshorne1977} isomorphic to the cokernel  $M$ of $H^0( C, K_C \otimes L^{-1}) \to H^0(C, K_C(Z) \otimes L^{-1}) $. In particular, $M$ and the space of linear forms on $H^0(Z,L)$ that vanish on $H^0(C,L)$  have the same dimension. Every element of $H^0(C, K_C(Z) \otimes L^{-1})$ defines via the residue pairing a linear form on $H^0(Z,L)$ that vanishes on $H^0(C,L)$ and two elements define the same linear form if and only if their restriction to $Z$ is equal, i.e. if and only if their difference lies in $H^0(C, K_C \otimes L^{-1})$, so the residue pairing gives an injection from $M$ to the space of linear forms on $H^0(Z,L)$ vanishing on $H^0(C,L)$, which must therefore be a surjection. Thus, if the linear form $b \mapsto \alpha ( a^{k-1} b) = \operatorname{res} (a^{k-1} b, \tilde{\alpha} )$ vanishes on $b \in H^0(C,L)$, then there must exist $c \in H^0 (C, K_C(Z) \otimes L^{-1} ) $ such that $c \mid_Z = \tilde{\alpha} a^{k-1}$, verifying ``only if".

To check that $\rho_Z$ is injective, it suffices to observe that the kernel of $\rho_Z$ is $H^0(C, K_C \otimes L^{-1})$ which vanishes as $e>2g-2$.
\end{proof}

Let 
\[ \operatorname{Sing}_{\tilde{\alpha}, Z } =  \{ (a,c) \mid a \in H^0(C, L), c\in H^0 (C, K_C(Z) \otimes L^{-1} ), c \mid_Z = \tilde{\alpha} a^{k-1} \}.\] The map $\operatorname{Sing}_{\tilde{\alpha}, Z }  \to \operatorname{ Sing}_\alpha$ sending $(a,c) $ to $ a$ is well-defined and surjective by the first part of Lemma \ref{exists-c} and injective  by the last part of Lemma \ref{exists-c} so we have \begin{equation}\label{c-introduction} \dim \operatorname{ Sing}_\alpha = \dim \operatorname{Sing}_{\tilde{\alpha}, Z } .\end{equation} (We have checked that the map $(a,c) \mapsto a$ is a set-theoretic bijection, which suffices to show the dimension equality \eqref{c-introduction}. One could even check it is an isomorphism of schemes, but this is not necessary for us.)

For $(a,c) \in \operatorname{Sing}_{\tilde{\alpha}, Z }$ with  $a=0$ we must have  $c\mid_Z=0$ so, by the injectivity of $\rho_Z$, we have $c=0$. We thus split $\operatorname{Sing}_{\tilde{\alpha}, Z }$ into a closed set where $c=0$ and an open set where $a,c\neq 0$. We can further stratify the open set where $a,c\neq 0$ as follows: to each point which is a vanishing point of either $a$ or $c$, we associate a pair of nonnegative integers $(i,j)$, not both $0$, where $i$ is the order of vanishing of $a$ at that point and $j$ is the order of vanishing of $c$ at that point. A pair $a,c$ then provides a function $\mu$ from pairs of nonnegative integers  to nonnegative integers where $\mu(i,j)$ is the total degree of points associated to the pair $(i,j)$ and $\mu(0,0)=0$. We have \begin{equation}\label{mu-eq-1} \sum_{i,j \in \mathbb N^2 } \mu(i,j) i = e\end{equation} and  \begin{equation}\label{mu-eq-2} \sum_{i,j\in \mathbb N } \mu(i,j) j= 2g-2 + \deg Z - e\end{equation} because, as $a \in H^0(C, L)$, the total order of vanishing, weighted by degree, of all the zeroes of $a$ is $\deg L =e$, and the points where $a$ has order of vanishing $i$ and $c$ has order of vanishing $j$ contribute $\mu(i,j)$ to the total order of vanishing of $a$, while an analogous argument works for $c$.  Let $M_{e, 2g-2+ \deg Z-e}$ be the set of functions $\mu \colon \mathbb N^2 \to \mathbb N$ satisfying $\mu(0,0)=0$, \eqref{mu-eq-1}, and \eqref{mu-eq-2}. To each pair $(a,c)$ with $c\neq 0$ we then have a function $\mu\in M_{e, 2g-2+ \deg Z-e}$.

The set of all points $(a,c) \in \operatorname{Sing}_{\tilde{\alpha}, Z }^{ a,c\neq 0}$ providing a given function $\mu\in M_{e, 2g-2+ \deg Z-e}$ then forms a locally closed subset $ \operatorname{Sing}_{\tilde{\alpha}, Z }^{ \mu}$. We then have
\begin{equation}\label{c-dichotomy}  \operatorname{Sing}_{\tilde{\alpha}, Z } =  \operatorname{Sing}_{\tilde{\alpha}, Z }^{c=0} \cup\operatorname{Sing}_{\tilde{\alpha}, Z }^{ a,c\neq 0} \end{equation} and \begin{equation}\label{mu-chotomy}\operatorname{Sing}_{\tilde{\alpha}, Z }^{ a,c\neq 0}=  \bigcup_{ \substack{\mu\in M_{e, 2g-2+ \deg Z-e}}}  \operatorname{Sing}_{\tilde{\alpha}, Z }^{ \mu} .\end{equation}
We estimate the dimension of each locally closed subset in this decomposition separately.

\begin{lemma}\label{closed-dim} Let $Z$ be a proper closed subset of $C$ and fix $\tilde{\alpha} \in H^0 (Z, K_C (Z) \otimes L^{-k})$ invertible. For $v\in \abs{Z}$, let $m_v$ be the multiplicity of $v$ in $Z$. Let  \begin{equation}\label{dzkdef} d_k(Z) = \sum_{v \in \abs{Z}} \left\lceil \frac{m_v}{k-1}\right \rceil \deg v .\end{equation} Then
\[ \dim  \operatorname{Sing}_{\tilde{\alpha}, Z }^{c=0}  \leq  \max \Bigl( e -g+1 - d_k(Z) , \frac{1}{2}\left( e+2 - d_k(Z) \right) , 0\Bigr ).\]
\end{lemma}

\begin{proof} Plugging in $c=0$ to the definition of $\operatorname{Sing}_{\tilde{\alpha}, Z }$, we see that  \[ \operatorname{Sing}_{\tilde{\alpha}, Z } ^{c=0} =  \{ a \in H^0(C, L)\mid a^{k-1} \tilde{\alpha} =0  \}  =\{ a \in H^0(C, L) \mid a^{k-1} \mid_Z =0  \}  \] since $\tilde{\alpha}$ is invertible. We have $a^{k-1} \mid_Z=0$ if and only if $a^{k-1}$ vanishes to order at least $m_v$ at each point $v \in \abs{Z}$, in other words if $a$ vanishes to order at least $\lceil \frac{m_v}{k-1} \rceil$ at each point $v \in \abs{Z}$.  This happens if and only if $a$ is a global section of $L ( - \sum_{v \in \abs{Z}} \lceil \frac{m_v}{k-1} \rceil  [v])$, which is a line bundle of degree $e- d_k(Z) $. The result is then given by Clifford's theorem in the form of Lemma \ref{clifford}(1).  \end{proof}

One can understand  $\operatorname{Sing}_{\tilde{\alpha}, Z }^{ \mu} $ geometrically as follows. For $C$ a curve and $n$ a natural number, $\Sym^n C$ is the moduli space parameterizing degree $n$ divisors of $C$. It may also be written as the quotient $(C^n)/S_n$ where $S_n$ acts on $C^n$ by permuting the $n$ points.

There is a natural map from $\operatorname{Sing}_{\tilde{\alpha}, Z }^{ a,c\neq 0}$ to $\Sym^{e} (C) \times \Sym^{2g-2+\deg Z-e}(C)$ that sends a pair $(a,c)$ to the pair consisting of the divisor of zeroes of $a$ and the divisor of zeros of $c$. By definition, $\operatorname{Sing}_{\tilde{\alpha}, Z }^{ \mu} $ is the inverse image under this map of the subset of $\Sym^{e} (C) \times \Sym^{2g-2+\deg Z-e}(C)$ consisting of pairs of divisors $D_1,D_2$ such that the total degree of the points which have multiplicity $i$ in $D_1$ and multiplicity $j$ in $D_2$ is $\mu(i,j)$ for each pair $(i,j)$ of nonnegative integers, not both zero.  

We will first describe the tangent space of $\Sym^n C$. Using that, we will describe the tangent spaces of these subsets  of $\Sym^{e} (C) \times \Sym^{2g-2+\deg Z-e}(C)$, and finally the tangent spaces of their inverse images $\operatorname{Sing}_{\tilde{\alpha}, Z }^{ \mu} $, which will ultimately allow us to bound the dimension of $\operatorname{Sing}_{\tilde{\alpha}, Z }^{ \mu} $.

$\Sym^n C$ can be equivalently expressed as the component of the Hilbert scheme of $C$, i.e. the moduli space parameterizing closed subschemes of $C$, or equivalently of ideal sheaves on $C$, consisting of finite closed subschemes of length $n$~\cite[p. 274]{Grothendieck1961}. The tangent space to the Hilbert scheme of a projective scheme $X$ at an ideal sheaf $I$ is canonically identified with $\Hom (I ,\mathcal O_X/I)$~\cite[Corollary 5.4]{Grothendieck1961}, with the map given explicitly on a family of ideals, in other words an ideal on a product $X \times S$ that is flat over $S$, by lifting elements of the ideal $I$ parameterized by a point $x\in S$ to elements of the family of ideals over a neighborhood of $x$ and then modding out by $I$ to produce an element of the tensor product of $\mathcal O_X/I$ with the maximal ideal of $x$. This gives a map from the tangent space, which is the dual of the maximal ideal modulo the maximal ideal squared, to $\Hom (I, \mathcal O_X/I)$. In the case of a curve $C$, at a point corresponding to an effective divisor $D \subset C$ of degree $n$, the ideal $I$ is $ \mathcal O_C(-D)$ so the tangent space is identified with \[ \Hom ( \mathcal O_C(-D), \mathcal O_C / \mathcal O_C(-D)) \cong \Hom (\mathcal O_C(-D), \mathcal O_D) \cong H^0(D, \mathcal O_C(D))\] where the first isomorphism comes from the short exact sequence $0 \to \mathcal O_C(-D) \to \mathcal O_C \to \mathcal O_D$ and the second is the isomorphism $\Hom (L ,\mathcal F) = H^0(C, L^{-1} \otimes \mathcal F)$ valid for any line bundle $L$ and sheaf of $\mathcal O_C$-modules $\mathcal F$.

By construction, this identification is local on $C$. In particular, we have the following compatibility with multiplication maps:

\begin{lemma}\label{compatibility-lemma} Fix divisors $D_1$ and $D_2$ on $C$, of degrees $n_1$ and $n_2$ respectively. The differential of the multiplication map $\Sym^{n_1} C \times \Sym^{n_2} C \to \Sym^{n_1+ n_2} C$ sends the product of the tangent space of $\Sym^{n_1}C $ at $D_1$ with the tangent space of $\Sym^{n_2} C$ at $D_2$ to the tangent space of $\Sym^{n_1+ n_2} C$ at $D_1 + D_2$. If we compose this map by the identification of the tangent space of $\Sym^n C$ at $D$ with $H^0(D, \mathcal O_C(D))$ for $(n,D) = (n_1, D_1), (n_2, D_2)$, and $(n_1+n_2, D_1 + D_2)$, we obtain a map $H^0(D_1, \mathcal O_C( D_1)) \times H^0(D_2, \mathcal O_C(D_2)) \to H^0(D_1+ D_2, \mathcal O_C(D_1 + D_2))$.

When $D_1$ and $D_2$ have disjoint support, this map is inverse to the natural isomorphism obtained by restricting to $D_1$ and using the natural isomorphism $\mathcal O_C(D_1+D_2) \to \mathcal O_C(D_1)$ away from $D_2$, and by restricting to $D_2$ and using the natural isomorphism $\mathcal O_C(D_1+D_2) \to \mathcal O_C(D_2)$ away from $D_1$.

\end{lemma}

\begin{proof} Since the map by restricting to $D_1$ and $D_2$ is an isomorphism, it suffices to check that its composition with the differential is the identity, i.e. it suffices to compute the restriction of the image of a given tangent vector under the differential to $D_1$ and to $D_2$. Since the construction is local, when computing the restriction to $D_1$, $D_2$ is irrelevant, and thus we may assume $D_2=0$, in which case statement is trivial. Similarly, when computing the restriction to $D_2$, we may assume $D_1=0$, in which case the statement is trivial.\end{proof}

Also by construction, for $L$ a line bundle of degree $n$, the derivative at a section $f$ of the natural map $H^0(C, L) \setminus 0 \to \Sym^n(C)$ that sends a section to  its vanishing locus sends a tangent vector $\partial f$ to $\frac{\partial f}{f} \in H^0(D, \mathcal O_C(D))$, since the tangent vector corresponds to the family of ideals generated by $f + \epsilon \partial f$ and thus the induced element of $\Hom ( \mathcal O_C(-D), \mathcal O_C / \mathcal O_C(-D)) $ sends $f$ to $\partial f$.

For $(i,j) \in \mathbb N^2\setminus \{(0,0)\}$ let $p^{w(i,j)} $ be the greatest power of $p$ dividing both $i$ and $j$. In other words, $w(i,j)$ is the $p$-adic valuation of $\gcd(i,j)$. For a natural number $w$, let $C_w$ be the unique smooth projective curve with a totally inseparable map of degree $p^w$ from $C$, in other words the unique curve whose function field is the field of $p^w$th powers of elements of $\mathbb F_q(C)$. 

The next lemma will involve the product $\prod_{i,j \in \mathbb N} C^{\mu(i,j)}$ for $\mu \in  M_{e, 2g-2+ \deg Z-e}$. The function $\mu$ is necessarily finitely-supported so all but finitely many of the factors in the product are a single point. Since the point is the identity for the product of schemes, this makes the product well-defined (and, concretely, isomorphic to the product of the finitely many terms that are not a single point).

\begin{lemma}\label{tangent-in-sym} Fix $\mu \in M_{e, 2g-2+ \deg Z-e}$. The map
\[ \prod_{i,j \in \mathbb N} C^{\mu(i,j)} \to \Sym^{e} (C) \times \Sym^{2g-2+\deg Z-e}(C) \]
that sends a tuple $ (x_{i,j,t} )_{1\leq t \leq \mu(i,j)} $ to the pair of divisors \[ \sum_{i,j \in \mathbb N}   \sum_{t=1}^{\mu(i,j)} i [ x_{i,j,t}] ,  \sum_{i,j \in \mathbb N}   \sum_{t=1}^{\mu(i,j)} j[ x_{i,j,t}] \] factors through \[ \prod_{i,j \in \mathbb N}  C^{\mu(i,j)}_{w(i,j)}.\]

Fix a point of $\prod_{i,j \in \mathbb N}  C^{\mu(i,j)}_{w(i,j)}$, expressed as a tuple $(x_{i,j,t})$ of points of $C$, using that closed points on $C$ are naturally in bijection with closed points of $C_{w(i,j)}$. Assume that the $x_{i,j,t}$ are distinct. Fix a tangent vector of $C^{\mu(i,j)}_{w(i,j)}$ at $(x_{i,j,t})$,  expressed as a tuple $(\partial x_{i,j,t})$ of tangent vectors on $C_{w(i,j)}$. Let $\langle, \rangle$ denotes the natural pairing between tangent vectors of $C_{w(i,j)}$ and elements of the maximal ideal of $C_{w(i,j)}$.

The differential of the map \begin{equation}\label{differentiable-map} \prod_{i,j \in \mathbb N}  C^{\mu(i,j)}_{w(i,j)} \to \Sym^{e} (C) \times \Sym^{2g-2+\deg Z-e}(C) \end{equation} at $(x_{i,j,t})$ sends the tangent vector $(\partial x_{i,j,t})$ to the element of $H^0(D_1, \mathcal O_C(D_1)) \times H^0(D_2, \mathcal O_C(D_2))$ whose first component, restricted to the maximal closed subscheme of $D_1$ supported at $x_{i,j,t}$, takes the value \begin{equation}\label{tangent-vector-value-1}- \frac{i }{p^{w(i,j)} }  \frac{ \langle \partial x_{i,j,t}, \pi_{x_{i,j,t}}^{p^{w(i,j)}} \rangle }{ \pi_{x_{i,j,t}}^{p^{w(i,j)}}} \end{equation} and whose second component, restricted to the maximal closed subscheme of $D_2$ supported at $x_{i,j,t}$, takes the value \begin{equation}\label{tangent-vector-value-2} - \frac{j}{ p^{w(i,j)}} \frac{ \langle \partial x_{i,j,t}, \pi_{x_{i,j,t}}^{p^{w(i,j)}} \rangle }{ \pi_{x_{i,j,t}}^{p^{w(i,j)}}} .\end{equation} Note that $\pi^{ p^{w(i,j)}}$ lies in this maximal ideal of $C_{w(i,j)}$ at $x_{i,j,t}$ so that the pairing is well-defined.

Finally, again restricted to a point $(x_{i,j,t})$ where the $x_{i,j,t}$ are all distinct, the differential of \eqref{differentiable-map}  induces an isomorphism of the tangent space of $\prod_{i,j \in \mathbb N}  C^{\mu(i,j)}_{w(i,j)}$ to the tangent space of the image of \eqref{differentiable-map}. \end{lemma}

\begin{proof} We handle each claim in turn.

For the factorization statement, by the existence of various multiplication maps $\Sym^{n_1} C \times \Sym^{n_2} C \to \Sym^{n_1+n_2} C$, it suffices to prove that the map $C \to \Sym^i C \times \Sym^j C$ sending $x$ to $(i [x], j[x])$ factors through $C_{w(i,j)}$. Since by definition $i$ and $j$ are divisible by $p^{w(i,j)}$, it suffices to check that the map $C \to \Sym^{p^w} C$ sending $x$ to $p^w[x]$ factors through $C_w$. The factorization of a map from a curve through a totally inseparable map of degree $p^w$ may be checked \'etale-locally. Since \'{e}tale-locally $C$ is isomorphic to $\mathbb A^1$, we may work in $\mathbb A^1$, where  the map $\mathbb A^1 \to \Sym^{p^w} \mathbb A^1$ sending $x$ to $p^w[x]$  sends a point $x$ to the vanishing locus of the polynomial $(T-x)^{p^w} = T ^{p^w}-x^{p^w}$ which may be expressed only in terms of $x^{p^w}$ and thus factors through the unique totally inseparable map of degree $p^w$.

Let $e_w$ be the map $C_w \to \Sym^{p^w} C$ whose composition with the totally inseparable map $C \to C_w$ of degree $p^w$ sends $x$ to $p^w [x]$, whose existence we have just checked.

To calculate the differential of \eqref{differentiable-map}, by Lemma \ref{compatibility-lemma}, it suffices to prove that the differential of $e_{w(i,j)}$  sends a tangent vector  $\partial x_{i,j,t}$ to the pair consisting of \eqref{tangent-vector-value-1} and \eqref{tangent-vector-value-2}. The ideal-theoretic description of the identification of the tangent space of the symmetric power with the global sections of a sheaf makes it clear that it is \'{e}tale-local. (We can formulate the ideal-theoretic description without mentioning the Hilbert scheme, as simply a map which takes a family of ideals on $X$ parameterized by a scheme $S$ to a section of a sheaf on $X \times S$. It is necessary to do this when working \'{e}tale-locally as the Hilbert scheme does not exist in general.) 

For compactness of notation write $w=w(i,j)$. The uniformizer $\pi_{x_{i,j,t}}$ gives an \'{e}tale map to $\mathbb A^1$, and working \'{e}tale-locally, $e_w$ sends $x^{p^w}$ to 
\[  ( (T-x)^i, (T-x)^j)=  ( ( T^{p^{w}} - x^{ p^{w}})^{ \frac{i}{p^w}}, ( T^{p^{w}} - x^{ p^{w}})^{ \frac{j}{p^w}}) .\] Recalling that the identification of tangent spaces sends the family of ideals generated by a family of functions $f$ to $\frac{\partial  f}{f}$, we may calculate the differential of $e_w$ by taking the generator of the family of ideals, differentiating with respect to $x^{p^{w}}$, and then dividing by the generator, obtaining
\[\Bigl(  \frac{ - \frac{i}{p^{w}}  ( T^{p^{w} }- x^{ p^{w}})^{ \frac{i}{p^{w}}-1 } dx^{p^{w} } }{( T^{p^{w}} - x^{ p^{w}})^{ \frac{i}{p^{w}}}} ,  \frac{ - \frac{j}{p^{w}}  ( T^{p^{w}} - x^{ p^{w}})^{ \frac{j}{p^{w}}-1 }dx^{p^{w}}}{( T^{p^{w}} - x^{ p^{w}})^{ \frac{j}{p^{w}}}}\Bigr)\]
\[ =   \Bigl( - \frac{i}{p^{w}} \frac{   dx^{p^{w}}} { T^{p^{w} }- x^{ p^{w}}} ,  -\frac{j}{p^{w}} \frac{ dx^{p^{w}}} { T^{p^{w} }- x^{ p^{w}}}  \Bigr) .\]
Pulling $T^{p^{w}} - x^{ p^{w}}$ back along the \'{e}tale-local map $\pi_{i,j,t}$ gives $\pi_{i,j,t}^{p^{w}}$ and pulling $  dx^{p^{w}}$ back gives $\langle \partial x_{i,j,t} , \pi_{i,j,t}^{p^{w}}\rangle$. This gives \eqref{tangent-vector-value-1} and \eqref{tangent-vector-value-2}.

Finally, to check that the differential of \eqref{differentiable-map} gives an isomorphism on tangent spaces onto the image of \eqref{differentiable-map}, it suffices to check that \eqref{differentiable-map} is a map between smooth varieties which is injective on tangent spaces and open onto its image. The fact that both varieties are smooth is straightforward.

The fact that \eqref{differentiable-map} is injective on tangent spaces follows from the explicit formula for the differential of \eqref{differentiable-map}: Since the tangent space of the source and target both split as a product over $x_{i,j,t}$, with the map compatible with that splitting, it suffices to check the restriction to a given $x_{i,j,t}$ is injective. Since $w(i,j)$ is the highest power of $p$ dividing both $i$ and $j$, at least one of $\frac{i}{p^{w(i,j)}}$ and $\frac{j}{p^{w(i,j)}}$ must be coprime to $p$. Hence it suffices to check that for $\partial x_{i,j,t}$ nonzero, $\frac{ \langle  \partial x_{i,j,t}, \pi_{x_{i,j,t}}^{p^{w(i,j)}} \rangle }{ \pi_{x_{i,j,t}}^{p^{w(i,j)}}}$ is nonzero, which is clear since $\pi_{x_{i,j,t}}^{p^{w(i,j)}} $ is a uniformizer of $C_{w(i,j)}$ at $x_{i,j,t}$ so that $ \langle  \partial x_{i,j,t}, \pi_{x_{i,j,t}}^{p^{w(i,j)}} \rangle\neq 0$ for $\partial x_{i,j,t}\neq 0$ and $\frac{1}{ \pi_{x_{i,j,t}}^{p^{w(i,j)}}}$ is nonzero modulo $\mathcal O_C$ and thus nonzero as an element of $H^0( i [x_{i,j,t}], \mathcal O_C(i [x_{i,j,t}]))$.

To check that \eqref{differentiable-map} is open onto its image, we first calculate its image, which consists of pairs of divisors $(D_1,D_2)$ on $C$ such that the number of geometric points whose multiplicity in $D_1$ is $i$ and whose multiplicity in $D_2$ is $j$ is exactly $\mu(i,j)$.  By \cite[Tag 01U1]{stacks-project}, to check that a map $f$ is open onto its image, it suffices to check that for $a$ a point of the domain of $f$, $b=f(a)$, and $c$ a point of the image of $f$ that specializes to $b$, there exists a point $d$ of the domain of $f$ that specializes to $a$ and satisfies $f(d)=c$. Two points $b,c$ of a scheme, $b$ the specialization of $c$, always arise from a map from $\operatorname{Spec} R$ to the scheme for some integral local ring $R$, with $b$ the image of the special point $s$ and $c$ the image of the general point $\eta$. (One can take $R$ to be the local ring at $b$ modulo the ideal corresponding to $c$.) So we may consider a family of pairs of divisors over $\operatorname{Spec} R$, represented by a pair $Z_1,Z_2$ of closed subschemes of $C \times \operatorname{Spec} R$, where for each $r\in  \operatorname{Spec} R $  in the family, the number of geometric points whose multiplicity in $(Z_1)_r$ is $i$ and whose multiplicity in $(Z_2)_r$ is $j$ is exactly $\mu(i,j)$.

Lifting the point $((Z_1)_r, (Z_2)_r)$ of $\Sym^{n_1}(C) \times \Sym^{n_2}(C)$ along the map \eqref{differentiable-map} corresponds to labeling the geometric points $(Z_1)_r \cup (Z_2)_r$ by triples $(i,j,t)$ in such a way that points labeled $(i,j,t)$ have multiplicity $i$ in $(Z_1)_r$ and $j$ in $(Z_2)_r$ and distinct points get distinct labels. This produces a tuple, indexed by triples $(i,j,t)$, of points of $C$ and thus a point of  $\prod_{i,j \in \mathbb N}  C^{\mu(i,j)}_{w(i,j)}$. To check the criterion for openness, we have a labeling $\ell_s$ of  $(Z_1)_s \cup (Z_2)_s$ giving a point $(x_{i,j,t}^{\ell_s})$ of $\prod_{i,j \in \mathbb N}  C^{\mu(i,j)}_{w(i,j)}$ and must find a labeling $\ell_\eta$ of $ (Z_1)_\eta \cup (Z_2)_\eta$ whose associated point $(x_{i,j,t}^{\ell_\eta})$ of $\prod_{i,j \in \mathbb N}  C^{\mu(i,j)}_{w(i,j)}$ specializes to $(x_{i,j,t}^{\ell_s})$.

By passing to an extension of $R$, we may assume the field of fractions of $R$ is algebraically closed. (This removes the distinction between points and geometric points.) Then there is a specialization map from points of $(C \times \operatorname{Spec} R)_\eta$ to points of $(C \times \operatorname{Spec} R)_s$ that sends $x$ in $(C \times \operatorname{Spec} R)_\eta$ to the intersection with $(C \times \operatorname{Spec} R)_s$ of the closure of the image of $x$ in $C \times \operatorname{Spec} R$. Since $Z_1 \cup Z_2$ is closed, the specialization map sends points of $ (Z_1)_\eta \cup (Z_2)_\eta$ to points of $ (Z_1)_s \cup (Z_2)_s$.  Since $Z_1 \cup Z_2$ is flat over $\operatorname{Spec} R$, the specialization map is surjective. Since $ (Z_1)_\eta \cup (Z_2)_\eta$ and $ (Z_1)_s \cup (Z_2)_s$ both contain exactly $\sum_{i,j \in \mathbb N} \mu(i,j)$ points, it follows that the specialization map is a bijection. This lets us define a labeling $\ell_\eta$ of $ (Z_1)_\eta \cup (Z_2)_\eta$ by labeling each point by its image under the specialization map. By construction, for each triple $i,j,t$, the  point $x_{i,j,t}^{\ell_\eta}$ labeled $(i,j,t)$ in the labeling of $ (Z_1)_\eta \cup (Z_2)_\eta$ specializes to the point $x_{i,j,t}^{\ell_s}$ labeled $(i,j,t)$ in the labeling of $ (Z_1)_s \cup (Z_2)_s$, so the tuple $(x_{i,j,t}^{\ell_\eta})$ specializes to the tuple $(x_{i,j,t}^{\ell_s})$. \end{proof}

Let \[ \abs{\mu}_{k,p} = \sum_{ \substack{ i,j \in \mathbb N \\ p \mid \frac{ (k-1) i - j} { \gcd(i,j)}}} \mu(i,j) .\]

\begin{lemma}\label{strata-dim} For $\mu\in M_{e, 2g-2+ \deg Z-e}$, we have \[ \dim \operatorname{Sing}_{\tilde{\alpha}, Z }^{ \mu}  \leq g +1+  \abs{\mu}_{k,p} .\] \end{lemma}

\begin{proof} Write $V(f)$ for the set of vanishing points of a section $f$.
We begin by further stratifying $ \operatorname{Sing}_{\tilde{\alpha}, Z }^{ \mu}$ into strata $ \operatorname{Sing}_{\tilde{\alpha}, Z }^{ \mu,r}$ indexed by integers $r$ consisting of points $(a,c)$ of $ \operatorname{Sing}_{\tilde{\alpha}, Z }^{ \mu}$ such that $\deg ((V ( a) \cup V(c)) \cap Z)$ is $r$. On each component of each stratum, the set of points of $Z$ where $a$ or $c$ vanishes must be constant. It suffices to prove the same upper bound for the dimension of each stratum, and hence suffices to prove the same upper bound for the dimension of each stratum at each point. Fix now a point $(a,c)$ of $ \operatorname{Sing}_{\tilde{\alpha}, Z }^{ \mu,r}$. The tangent space of $\operatorname{Sing}_{\tilde{\alpha}, Z }$ at $(a,c)$ consists of pairs of $\partial a \in H^0( C, L)$ and $\partial c\in H^0(C, K_C(Z) \otimes L^{-1} ) $ such that \[ \partial c \mid_Z =(k-1)  \tilde{\alpha} a^{k-2} \partial a .\] 

Because $\operatorname{Sing}_{\tilde{\alpha}, Z }^\mu$ is the inverse image of the image in $\Sym^{e} (C) \times \Sym^{ 2g-2 + \deg Z -e}(C)$ of the map discussed in Lemma \ref{tangent-in-sym}, the  tangent space of $\operatorname{Sing}_{\tilde{\alpha}, Z }^\mu$ is the subset of pairs $(\partial a,\partial c)$ satisfying the additional condition that $(\frac{ \partial a}{a}, \frac{\partial c}{c} )$ lies in the tangent space of that image. By Lemma \ref{tangent-in-sym}, this means that there must be a tuple of scalars $s_x$ indexed by geometric points $x$ where either $a$ and $c$ vanishes, such that at a point $x$ at which $a$ vanishes to multiplicity $i$ and $c$ vanishes to multiplicity $j$, we have
\begin{equation}\label{movable-direction}\Bigl(\frac{ \partial a}{a}, \frac{\partial c}{c} \Bigr)= s_x \Bigl(  \frac{i }{p^{w(i,j)} }  \frac{ 1 }{ \pi_{x_{i,j,t}}^{p^{w(i,j)}}} ,  \frac{j }{p^{w(i,j)} }  \frac{ 1 }{ \pi_{x_{i,j,t}}^{p^{w(i,j)}}}\Bigr).\end{equation} (In the notation of Lemma \ref{tangent-in-sym}, for $x=x_{i,j,t}$ we take $s_x =- \langle \partial x_{i,j,t}, \pi_{x_{i,j,t}}^{p^{w(i,j)}} \rangle.$)

Finally, because on each component of  $ \operatorname{Sing}_{\tilde{\alpha}, Z }^{\mu,r}$ the vanishing points of $a$ and $c$ in $Z$ are fixed, if the vector lies in the tangent space of  $ \operatorname{Sing}_{\tilde{\alpha}, Z }^{ \mu,r}$ then for $x$ a geometric point of $Z$ to which $a$ vanishes to multiplicity $i$ and $c$ vanishes to multiplicity $j$, the scalar $s_x$ must vanish and so the pair $(\frac{ \partial a}{a}, \frac{\partial c}{c} )$ cannot have a pole at these points.

Since $c =  \tilde{\alpha} a^{k-1}$ on restriction to $Z$, we have
\begin{equation}\label{log-deriv-eq} \frac{\partial c}{c} c = \partial c = (k-1) \tilde{\alpha} a^{k-2} \partial a = (k-1) \frac{\partial a}{a} \tilde{\alpha} a^{k-1} = (k-1) \frac{\partial a}{a} c \end{equation} on restriction to $Z$.
  
 Fix now a geometric point $x$ of $Z$, with multiplicity $m$. Let $d$ be the order of vanishing of $c$ at $x$. Then $\frac{\partial c}{c} -(k-1)  \frac{\partial a}{a} $ vanishes to order at least $m-\min(d,m)$ at $x$ by \eqref{log-deriv-eq}. Thus, viewed as a section of $K_C( Z)$, 
\begin{equation}\label{cleared-denoms} ac \Bigl( \frac{\partial c}{c} -(k-1)  \frac{\partial a}{a}\Bigr)  = a \partial c - (k-1) c \partial a\end{equation} vanishes to order at least $m$ at $x$ since the order of vanishing of $c$ at least cancels the $-\min(d,m)$ term. Applying this for all $x$, we see that \eqref{cleared-denoms} is in fact a section of $K_C(Z-Z) =K_C$. Let $S_1$ be the subspace of the tangent space on which \eqref{cleared-denoms} vanishes. Since the space of sections of $K_C$ is $g$-dimensional, $S_1$ has codimension at most $g$ in the tangent space. It suffices to show that $S_1$ has dimension at most $1+ \abs{\mu}_{k,p}$.

 Let $S_2$ be the subspace of $S_1$ on which $s_x=0$ whenever $p \mid  \frac{ (k-1) i - j} { \gcd(i,j)}$. By definition, the number of $x$ such that $p \mid  \frac{ (k-1) i - j} { \gcd(i,j)}$ is $\abs{\mu}_{k,p}$, so $S_2$ has codimension at most $\abs{\mu}_{k,p}$ in $S_1$.  So it suffices to show that $S_2$ has dimension at most $1$.

For $(\partial a, \partial c)\in S_1$, we have the equation \begin{equation}\label{simpler-eq} \frac{\partial c}{c} = (k-1)  \frac{\partial a}{a}\end{equation}  of meromorphic functions on $C$. Combining \eqref{movable-direction} and \eqref{simpler-eq}, we obtain \[ (k-1) s_x \frac{i}{p^{w(i,j)}} = s_x \frac{j}{ p^{w(i,j)}}.\] This forces $s_x=0$ unless \[ (k-1)  \frac{i }{p^{w(i,j)} } - \frac{j}{p^{w(i,j)} } =0\] which happens if and only if $p \mid  \frac{ (k-1) i - j} { \gcd(i,j)}$ since the greatest power of $p$ dividing $\gcd(i,j)$ is $p^{w (i,j)}$. Thus, by definition of $S_2$, for $(\partial a, \partial c)\in S_2$ we have $s_x=0$ for all $x$. Thus, for $(\partial a, \partial c) \in S_2$, the rational functions $\frac{\partial a}{a}$ and $\frac{\partial c}{c}$ have poles nowhere and hence are constant. \eqref{simpler-eq} shows that $\frac{\partial c}{c}$ is determined by $\frac{\partial a}{a}$ so that $(\partial a, \partial c) \in S_2$ is determined by the constant $\frac{\partial a}{a}$  and thus $S_2$ has dimension at most $1$ (the dimension of the space of constant functions on a curve), as desired. \end{proof}


Now recall that $\Delta_{k,p}$ is the convex hull of the set of points $\begin{pmatrix} i \\ j \end{pmatrix} \in \mathbb N^2$ such that $\gcd(i,j,p)=1$ but $p \mid  (k-1) i - j$.   Also recall that $\gamma_{k,p}$ is the supremum of all $\gamma \in (0,\infty)$  such that $ \begin{pmatrix}  1\\ \frac{k-2}{2} \end{pmatrix}  \in  \gamma \Delta_{k,p}$, or $0$ if no such $\gamma$ exists.

\begin{lemma}\label{open-dim} We have
\[ \dim \operatorname{Sing}_{\tilde{\alpha}, Z }^{a,c\neq 0} \leq g+1 + \sup \left\{ \lambda \in (0,\infty) \mid  \begin{pmatrix}  e \\  2g-2 + \deg Z -e \end{pmatrix}  \in \lambda \Delta_{k,p} \right\} \] where we adopt the convention that the supremum of an empty set is $0$.  
\end{lemma}

\begin{proof} In view of \eqref{mu-chotomy} and Lemma \ref{strata-dim}, it suffices to prove that for $\mu\in M_{e, 2g-2+ \deg Z-e} $ we have
\begin{equation}\label{lambda-relation}  \sup \left\{ \lambda \mid  \begin{pmatrix}  e \\  2g-2 + \deg Z -e \end{pmatrix} \in \lambda \Delta_{k,p} \right\} \geq  \abs{\mu}_{k,p}  . \end{equation} We may assume $\abs{\mu}_{k,p}>0$ since if $\abs{\mu}_{k,p}=0$, \eqref{lambda-relation} is clear. 
We have
\[ \begin{pmatrix}  e \\  2g-2 + \deg Z -e \end{pmatrix} = \sum_{i,j \in \mathbb N}    \mu(i,j) \begin{pmatrix} i \\ j \end{pmatrix} \] \begin{equation}\label{mu-decomposition}=  \abs{\mu}_{k,p}   \sum_{ \substack{  i,j \in \mathbb N \\ p \mid \frac{ (k-1) i - j} { \gcd(i,j)}}} \frac{ \mu(i,j)}{  \abs{\mu}_{k,p} }  \begin{pmatrix} i \\ j \end{pmatrix}  +    \sum_{\substack{ i,j \in \mathbb N \\ p \nmid \frac{ (k-1) i - j} { \gcd(i,j)}}} \mu(i,j) \begin{pmatrix} i \\ j \end{pmatrix}. \end{equation}

By definition of $ \abs{\mu}_{k,p} $ we have \[  \sum_{ \substack{  i,j \in \mathbb N \\ p \mid \frac{ (k-1) i - j} { \gcd(i,j)}}} \frac{ \mu(i,j)}{  \abs{\mu}_{k,p} }  = 1.\] Hence the first term $   \abs{\mu}_{k,p}  \sum_{ \substack{ i,j \in \mathbb N  \\ p \mid \frac{ (k-1) i - j} { \gcd(i,j)}}}  \frac{\mu(i,j)}{ \abs{\mu}_{k,p} }  \begin{pmatrix} i \\ j \end{pmatrix} $ of \eqref{mu-decomposition}  is $ \abs{\mu}_{k,p} $ times a convex combination of the vectors $ \begin{pmatrix} i \\ j \end{pmatrix} $ which are all either in $\Delta_{k,p}$ by definition or are positive integer multiples of points of $\Delta_{k,p}$ and thus lie in $\Delta_{k,p}$ by Lemma \ref{Delta-stability}. Thus the first term of \eqref{mu-decomposition} lies in $\abs{\mu}_{k,p}  \Delta_{k,p}$. Since the second term is a vector with nonnegative entries, we have  \[   \begin{pmatrix}  e \\  2g-2 + \deg Z -e \end{pmatrix} \in  \abs{\mu}_{k,p}  \Delta_{k,p} \] by Lemma \ref{Delta-stability}, which implies \eqref{lambda-relation}. \end{proof}

Let $e'= e+  \frac{2 \max(2g-1,0)}{k-2}$ if $k>2$ or $e'=e$ if $k=2$. 

\begin{lemma}\label{open-dim-simplified}   Fix $\alpha \in H^0(C, L^k)^\vee$. For  $Z$ a closed subscheme of minimal degree through which $\alpha$ factors, we have
   \[  \dim \operatorname{Sing}_{\tilde{\alpha}, Z }^{a,c\neq0} \leq g+1 + e' \gamma_{k,p}.\] \end{lemma}

\begin{proof}First assume $k>2$. We observe that $\deg Z \leq \frac{ke}{2} +1$  so by Lemma \ref{Delta-stability} we have
\[ \max \left\{ \lambda \mid  \begin{pmatrix}  e \\  2g-2 + \deg Z -e \end{pmatrix}  \in\lambda \Delta_{k,p} \right\} \leq \max \left\{ \lambda \mid  \begin{pmatrix}  e \\  2g-2 + \frac{ke}{2}+1-e \end{pmatrix}  \in \lambda \Delta_{k,p} \right\} \] \[\leq \max \left\{ \lambda \mid  \begin{pmatrix}  e  + \frac{ 2\max(2g-1,0)}{k-2}  \\  \max( 2g-1,0) + \frac{(k-2) e}{2} \end{pmatrix}  \in \lambda \Delta_{k,p} \right\} = (e + \frac{2 \max(2g-1,0)}{k-2}) \gamma_{k,p} = e'\gamma_{k,p}.\] The result then follows from Lemma \ref{open-dim}. 

In the $k=2$ case, we must prove a sharper bound. In fact we will prove this bound for $ \dim \operatorname{Sing}_{\alpha}$ directly, after observing that $\operatorname{Sing}_{\tilde{\alpha}, Z }^{a,c\neq0} $ is nonempty only if $H^0( C, K_C(Z) \otimes L^{-1})$ is nonzero which requires $\deg Z \geq  e+2-2g$. The minimality of $\deg Z$ implies $\deg Z \leq \frac{ke}{2} +1= e+1$ and hence $\deg Z \leq e$ since $\deg Z$ is an integer.

By Lemma \ref{exists-c} we have $a \in \operatorname{Sing}_\alpha$ if and only if there exists $c \in H^0(C, K_C(Z) \otimes L^{-1})$ such that $c\mid_Z = \tilde{\alpha} a$. Since $\tilde{\alpha}$ is invertible, the restriction of $a$ to $H^0(Z,L)$ is uniquely determined by $c$. The choices for $a$ given a fixed restriction to $H^0(Z,L)$ are a torsor for $H^0(C, L(-Z))$. Thus
\[ \dim \operatorname{Sing}_\alpha \leq \dim H^0(C, K_C(Z) \otimes L^{-1}) + \dim H^0(C, L(-Z)) \leq g+1 \]
by Clifford's theorem in the form of Lemma \ref{clifford}(3) since $\deg L(-Z)  = e -\deg Z$ lies in $[-1, 2g-2] \subseteq [-2,2g]$.\end{proof}

\begin{lemma}\label{overall-dim}  Fix $\alpha \in H^0(C, L^k)^\vee$. For  $Z$ the minimum closed subscheme through which $\alpha$ factors, we have
\[ \dim \operatorname{Sing}_\alpha \leq \max \left( e -g+1 -d_k(Z),  g+1 +  e' \gamma_{k,p} \right).\]
\end{lemma}
\begin{proof} By \eqref{c-introduction} and \eqref{c-dichotomy}, we have \[ \dim \operatorname{Sing}_\alpha= \dim \operatorname{Sing}_{\tilde{\alpha},Z} \leq  \max(  \dim \operatorname{Sing}_{\tilde{\alpha}, Z }^{c=0},  \dim \operatorname{Sing}^{a,c\neq0}_{\tilde{\alpha}, Z} ).\]
It follows from Lemma \ref{closed-dim} that
\[ \dim  \operatorname{Sing}_{\tilde{\alpha}, Z }^{c=0}  \leq  \max \left( e -g+1 - d_k(Z),g+1\right ).\]
where $g+1 \leq g+1 +  e' \gamma_{k,p} $. Combining this with Lemma \ref{open-dim-simplified}, we obtain the bound. \end{proof}

\begin{lemma}\label{harder-Euler-bound} For each positive $\delta < \frac{ \frac{s}{2} - k}{ k-1} $ we have 
\[ \sum_{\substack{ Z \subset C\\ \textrm{finite closed} \\ \deg Z > e-2g+1} } q^{ \deg Z -  \frac{1}{2} s d_k(Z)} \leq O_{s,k, \delta} (  (1+q^{-1/2})^{O_{k}(g)} q^{- \delta (e-2g+2)})   \leq O_{s,k, \delta, g} ( q^{- \delta (e-2g+2)}) .\] \end{lemma}

\begin{proof} For $u = q^{\delta}$ we have 
\[ \sum_{\substack{ Z \subset C\\  \textrm{finite closed} \\\deg Z > e-2g+1} } q^{ \deg Z - \frac{1}{2} s d_k(Z)}   \leq  q^{- \delta (e-2g+2)} \sum_{ \substack{Z \subset C\\  \textrm{finite closed}}} u^{\deg Z} q^{ \deg Z - \frac{1}{2} s d_k(Z)}  \] so it suffices to prove  
\[ \sum_{ \substack{Z \subset C\\  \textrm{finite closed}}} u^{\deg Z} q^{ \deg Z -\frac{1}{2} s d_k(Z)} = O_{s,k, \delta} (  (1+q^{-1/2})^{O_{s,k}(g)}  ) .\] (The second inequality in the statement follows from  $  (1+ q^{-\frac{1}{2}} )^{O_{s,k}(g) }= O_{s,k,\delta, g}(1)$.)
We have an Euler product expansion, recalling the definition \eqref{dzkdef} of $d_k(Z)$,
\[\sum_{ \substack{Z \subset C\\  \textrm{finite closed}} } u^{\deg Z} q^{ \deg Z -\frac{1}{2} s d_k(Z)} = \sum_{ Z \subset C} u^{\deg Z} q^{ \deg Z - s \frac{\sum_{v \in \abs{Z}}\left \lceil \frac{ m_v}{k-1}  \right \rceil \deg v}{2}} = \prod_{v \in \abs{C} } \sum_{m=0}^{\infty} (u q)^{ m \deg v }  q^{- \frac{ s \left\lceil \frac{m}{k-1} \right\rceil \deg v}{2}} \]
\[ \leq  \prod_{v \in \abs{C}}\prod_{j=1}^{k-1} \frac{1}{ 1- (uq)^{j \deg v} q^{- \frac{ s \deg v }{2}}} \]
\[ =\prod_{j=1}^{k-1} \zeta_C ( u^j q^{ j - \frac{s}{2}})  \leq \prod_{j=1}^{k-1}   \frac{ (1 + u^j q^{j - \frac{s-1}{2}})^{2g}}{(1- u^j q^{ j -\frac{s}{2}}) (1- u^j q^{ j+1 -\frac{s}{2}}) }\] by Lemma \ref{rh-bound}.

If $\delta < \frac{  \frac{s}{2}- k}{k-1}$ then  $j \delta + j+1 < \frac{s}{2} $ for all $j \leq k-1$ so all the terms in the denominator have the form $(1-q^f)$ with $f<0$ depending on $s,k,\delta$ and so are lower bounded by $(1-2^f)$ which depends only on $s,k,\delta$. Similarly, the terms in the numerator are bounded by $1 + q^{-1/2}$ and the number of terms appearing is $2g (k-1) = O_k(g)$. \end{proof} 

\begin{lemma}\label{messy-minor-bound} Assume $k$ is coprime to $p$. For all positive $\delta < \frac{  \frac{s}{2}- k}{k-1}$ we have 
\[\sum_{ \substack{ \alpha \in H^0(C,L^k)^\vee \\ \deg \alpha > e-2g+1 } } \abs{S(\alpha)}^s \]
\[ \leq  k q^{(k+1)e + 2-2g  }  3^{s-2} (k+1)^{(s-2) (e+1-g)}  q^{(s-2)  \frac{ e+2+  e'\gamma_{k,p}}{2}} \] 
\[ + O_{s,k, \delta} ( (k+1)^{s (e+1-g)}  q^{ s (e+1-g)}  (1+q^{-1/2})^{O_{k}(g)} q^{- \delta (e-2g+2)}).\]\end{lemma}

\begin{proof} Let $\alpha$ factor minimally through a closed subscheme $Z$. Then by Lemma \ref{minor-katz} and Lemma \ref{overall-dim} we have

\[ \abs{S(\alpha)}\leq  3(k+1)^{(e+1-g)}  q^{ \frac{ e+1-g + \dim \operatorname{Sing}_\alpha}{2}} \]
\[\leq 3(k+1)^{(e+1-g)}  q^{ \frac{ e+1-g +  \max ( e -g+1 - d_k(Z),  g+1 +  e'\gamma_{k,p} ) }{2}} \]
and thus either
\begin{equation}\label{messy-minor-bound-1}  \abs{S(\alpha)} \leq 3(k+1)^{(e+1-g)}  q^{ \frac{ 2e+2-2g - d_k(Z)}{2}} \end{equation}
or
\begin{equation} \label{messy-minor-bound-2} \abs{S(\alpha)} \leq 3(k+1)^{(e+1-g)}  q^{ \frac{ e+2+  e'\gamma_{k,p} }{2}} .\end{equation}

In case \eqref{messy-minor-bound-1} we have
\[ \abs{S(\alpha)}^s \leq  3^s (k+1)^{s(e+1-g)}  q^{s \frac{ 2e+2-2g - d_k(Z)}{2}} \]
and in case \eqref{messy-minor-bound-2} we have
\[ \abs{S(\alpha)}^s \leq   3^{s-2} (k+1)^{(s-2)(e+1-g)}  q^{(s-2)  \frac{e+2+  e'\gamma_{k,p} }{2}}  \abs{S(\alpha)}^2\]
so that in either case we have
\[ \abs{S(\alpha)}^s \leq    3^s (k+1)^{s(e+1-g)}  q^{s \frac{ 2e+2-2g - d_k(Z)}{2}} +  3^{s-2} (k+1)^{(s-2)(e+1-g)}  q^{(s-2)  \frac{e+2+  e'\gamma_{k,p} }{2}}  \abs{S(\alpha)}^2.  \]

Thus (choosing for each $\alpha$ a minimal closed subscheme $Z_\alpha$)
\[\sum_{ \substack{ \alpha \in H^0(C,L^k)^\vee \\ \deg \alpha > e-2g+1 } } \abs{S(\alpha)}^s \]
\[ \leq  \sum_{ \substack{ \alpha \in H^0(C,L^k)^\vee \\ \deg \alpha > e-2g+1 } }   3^s(k+1)^{s(e+1-g)}  q^{s \frac{ 2e+2-2g - d_k(Z_\alpha) }{2}} +\sum_{ \substack{ \alpha \in H^0(C,L^k)^\vee \\ \deg \alpha > e-2g+1 } }   3^{s-2}(k+1)^{(s-2)(e+1-g)}  q^{(s-2)  \frac{e+2 + e'\gamma_{k,p} }{2}}  \abs{S(\alpha)}^2. \]

For the second term we use the Plancherel formula estimate
\[ \sum_{ \substack{ \alpha \in H^0(C,L^k)^\vee  } } \abs{S(\alpha)}^2 = q^{ke+1-g} \#{ \{a_1, a_2 \in H^0(C,L) \mid a_1^k = a_2^k\}} \leq  k q^{(k+1)e + 2-2g  } \]
to obtain
\[   \sum_{ \substack{ \alpha \in H^0(C,L^k)^\vee \\ \deg \alpha > e-2g+1 } }  3^{s-2} (k+1)^{(s-2) (e+1-g)}  q^{(s-2) \frac{ e+2+  e'\gamma_{k,p}}{2}}  \abs{S(\alpha)}^2\]
\[ \leq  \sum_{ \substack{ \alpha \in H^0(C,L^k)^\vee  } } 3^{s-2} (k+1)^{(s-2) (e+1-g)}  q^{(s-2) \frac{ e+2+  e'\gamma_{k,p}}{2}}  \abs{S(\alpha)}^2\]
\[\leq    k q^{(k+1)e + 2-2g  } 3^{s-2} (k+1)^{(s-2) (e+1-g)}  q^{(s-2)  \frac{ e+2+  e'\gamma_{k,p}}{2}}.\]

For the first term we observe that there are at most $q^{\deg Z}$ choices of $\alpha$ for each subscheme $Z$ to obtain
\[  \sum_{ \substack{ \alpha \in H^0(C,L^k)^\vee \\ \deg \alpha > e-2g+1 } } 3^s   (k+1)^{s(e+1-g)}  q^{s \frac{ 2e+2-2g - d_k(Z_\alpha)}{2}} \]
\[ \leq 3^s(k+1)^{s(e+1-g)} \sum_{ \substack{ Z \subset C\\ \textrm{finite closed} \\\deg Z > e-2g+1}} q^{ \deg Z+ s \frac{ 2e+2-2g - d_k(Z), }{2}} \]
\[ \leq 3^s(k+1)^{s (e+1-g)}  q^{ s (e+1-g)} O_{s,k, \delta} (  (1+q^{-1/2})^{O_{k}(g)} q^{- \delta (e-2g+2)})\] for all $\delta < \frac{  \frac{s}{2}- k}{k-1}$ by Lemma \ref{harder-Euler-bound}. We can then absorb $3^s$ into the big $O$. \end{proof}

We first state a version of the main theorem with a complicated estimate that preserves as much uniformity as possible in the variables $q,g$. We will then state a simpler version that drops uniformity in $q,g$ and clarifies when obtain an asymptotic as $e\to\infty$ with other parameters fixed.

\begin{theo}\label{main-complicated} Assume $k\geq 2$, and $s> 2k$. Fix $\delta < \frac{ \frac{s}{2}-k}{k-1}$. For any finite field $\mathbb F_q$ of characteristic $p \nmid k$, curve $C$ of genus $g$ over $\mathbb F_q$, and natural number $e>2g-2$ we have
\[\#\{ \mathbf a \in H^0( C, L )^s \mid  \sum_{i=1}^s a_i^k = f\} -  \operatorname{MT}_e \] 
\[ =  O_{s,k} ( q^{e+1-g  } (k+1)^{(s-2) (e-g)}  q^{(s-2)  \frac{ e+2+  e'\gamma_{k,p}}{2}} )\] \[ + O_{s,k, \delta} ( (k+1)^{s (e-g)}  \frac{q^{s (e-g+1)} }{ q^{ke+1-g}}  (1+q^{-1/2})^{O_{s,k}(g)} q^{- \delta (e-2g+2)}) \]
\end{theo}

\begin{proof} Note first that $\delta < \frac{ \frac{s}{2}-k}{k-1}  $ implies $\delta<\frac{s-\max(k,3)- 1}{k} $ since if $k=2$ we have \[ \frac{ \frac{s}{2}-k}{k-1}   = \frac{s}{2}-2 = \frac{s-4}{2}= \frac{s-\max(k,3)- 1}{k} \]
and if $k >2$ we have  $k/2 < k-1$ and thus   \[  k (s/2-k)= k s/2 -k ^2  <  (k-1) s - k^2+1 = (k-1)( s-k-1).\]

In the case $k=2, g=0$, the stated error term is worse than the error term in Lemma \ref{easy-quadratic-case} and we can conclude immediately. Otherwise,  combining Lemmas \ref{lem-two-parts} (noting $s>2k\geq 4$ so $s\geq 5$), \ref{easier-euler}, and \ref{messy-minor-bound} we obtain  for $\delta < \frac{ \frac{s}{2}-k}{k-1}  $.
\[  \abs{ \#\{ \mathbf a \in H^0( C, L )^s \mid  \sum_{i=1}^s a_i^k = f\} -  \operatorname{MT}_e}\] 
\[ \leq k q^{e+1-g  } 3^{s-2} (k+1)^{(s-2) (e+1-g)}  q^{(s-2)  \frac{ e+2+  e'\gamma_{k,p}}{2}} \] 
\[ + O_{s,k, \delta} ( (k+1)^{s (e+1-g)}  \frac{q^{s (e-g+1)} }{ q^{ke+1-g}}  (1+q^{-1/2})^{O_{k}(g)} q^{- \delta (e-2g+2)})\]   
\[ +  \frac{q^{s (e-g+1)} }{ q^{ke+1-g}} O_{s,k, \delta} ( (1+q^{-1/2})^{O_{k,s}(g)} q^{- \delta (e-2g+2)} ).\]
For the first term, absorbing $k$ and $3^{s-2} (k+1)^{(s-2) } $ into the constant in the big $O$, we obtain the error term $O_{s,k} ( q^{e+1-g  } (k+1)^{(s-2) (e-g)}  q^{(s-2)  \frac{ e+2+e' \gamma_{k,p}}{2}} )$.

The last two terms can be combined into the single term $O_{s,k, \delta} ( (k+1)^{s (e-g)}  \frac{q^{s (e-g+1)} }{ q^{ke+1-g}}  (1+q^{-1/2})^{O_{s,k}(g)} q^{- \delta (e-2g+2)})$. \end{proof}

\begin{theo}\label{main-simplified} Assume $k\geq 2$, and $s> 2k$.  Fix a finite field $\mathbb F_q$ of characteristic $p > k$ and curve $C$ of genus $g$ over $\mathbb F_q$.  Let $\lqko= \frac{\log (k+1)}{\log q}$.  Fix $\theta $ such that \[ \theta < \frac{ \frac{s}{2}-k}{k-1} - s\lqko \] and \[ \theta \leq s-k- (s-2) \left(\frac{1+\gamma_{k,p}}{2} +  \lqko\right )  -1 .\]

Then for any number $e>2g-2$ we have
\[\#\{ \mathbf a \in H^0( C, L )^s \mid  \sum_{i=1}^s a_i^k = f\} -  \operatorname{MT}_e\] 
\[ = O_{s,k, q, g,  \theta} (  q^{ (s-k- \theta )e}    ) .\]

Such a $\theta >0$ exists if and only if \[ q > ( k+1)^{2 (k-1)} \] and
 \[ s > \max\left(  \frac{2 (k- \gamma_{k,p} -2 \lqko)}{ 1-\gamma_{k,p}  -2\lqko } ,  \frac{2k}{ 1- 2 (k-1)\lqko } \right) .\]
\end{theo}

\begin{proof} We choose $\delta = \theta + s\lqko $ and the first assumed upper bound on $\theta$ gives $\delta< \frac{ \frac{s}{2}-k}{k-1} $.
 Taking Theorem \ref{main-complicated}, observing by definition that $e'=e+O(g)$, and then dropping every term that depends only on $q$ and $g$, we obtain 
\[\#\{ \mathbf a \in H^0( C, L )^s \mid  \sum_{i=1}^s a_i^k = f\} -  \operatorname{MT}_e\] 
\[ =  O_{s,k} ( q^{e  } (k+1)^{(s-2)e}  q^{(s-2)   \frac{ 1+  \gamma_{k,p}}{2}e}) + O_{s,k,g, q ,\theta} ( (k+1)^{se }   q^{ (s-k- \delta )e}    ) .\]

With this choice of $\delta$ we have $(k+1)^{se }   q^{ (s-k- \delta )e}= q^{ (s-k-\theta)e}$.  Our second assumption on $\theta$ is equivalent to $ q^{e  } (k+1)^{(s-2)e}  q^{(s-2)   \frac{ 1+  \gamma_{k,p}}{2}e} \leq q^{ (s-k-\theta)e}$. So both terms are $O_{s,k, q, g,  \theta} (  q^{ (s-k- \theta )e}    )$, as desired.

A positive $\theta$ exists if and only if both upper bounds for $\theta$ are positive. These upper bounds are each affine functions of $s$ and take negative values respectively at $s=0$ and $s=2$, so for $s>2k > 2$ they can only be positive if the slope in $s$ is positive. The slope in $s$ is respectively $\frac{1}{2(k-1)} - \lqko  $ for the first bound and $ \frac{1- \gamma_{k,p}}{2}-  \lqko$ for the second bound.  The positivity of the slopes is equivalent to the lower bound
 \[ q > \max( ( k+1)^{2 (k-1)} ,  (k+1)^{  \frac{2}{ 1- \gamma_{k,p}}} )\]
 but since $p>k$ we have $\gamma_{k,p} \leq \frac{k-2}{2k-2} \left( 1+ \frac{k}{p} \right) \leq \frac{k-2}{k-1}$  so $\frac{2}{ 1-\gamma_{k,p}}  \leq 2 (k-1)$ and hence the maximum is always equal to $ ( k+1)^{2 (k-1)}$, which is the stated lower bound on $q$.

Given positive slopes, we can solve for the minimum value of $s$ where each upper bound on $\theta$ is positive. This gives the stated lower bounds on $s$. \end{proof}

\begin{proof}[Proof of Theorem \ref{intro-curve}] We apply Theorem \ref{main-simplified} to obtain
\[\#\{ \mathbf a \in H^0( C, L )^s \mid  \sum_{i=1}^s a_i^k = f\} -  \operatorname{MT}_e  = O_{s,k, q, g,  \theta} (  q^{ (s-k- \theta )e}    ) \] for some $\theta>0$, since the conditions in Theorem \ref{main-simplified} for $\theta>0$ to exist are exactly the conditions in Theorem \ref{intro-curve}. The condition $ k \geq 2$ and $p>k$ of Theorem \ref{main-simplified} are also assumed in Theorem \ref{intro-curve}. The condition $s>2k$ of Theorem \ref{main-simplified} follows from the condition 
 \[ s > \max\left(  \frac{2 (k- \gamma_{k,p} -2 \lqko)}{ 1-\gamma_{k,p}  -2\lqko } ,  \frac{2k}{ 1- 2 (k-1)\lqko } \right) \]
since $1- 2 (k-1)\lqko< 1$. The condition $e>2g-2$ can be dropped since the small $e$ case can be handled by increasing the implicit constant.

To obtain the desired asymptotic \[\#\{ \mathbf a \in H^0( C, L )^s \mid  \sum_{i=1}^s a_i^k = f\} = (1+o(1)) \operatorname{MT}_e \] 
it suffices to show that $\prod_{\substack{v \in \abs{C} }}\ell_v(f) ^{-1} = O(1)$ so that $\operatorname{MT}_e$ has size $\gg q^{ e(s-k)}$. This is accomplished by Lemma \ref{lv-lower-bound} as long as $k\geq 2, s>k+1, s\geq 5$, and $q>(k-1)^4$. All these conditions follow from our assumptions: $k\geq 2$ is simply assumed itself, while $s>k+1$ and $s\geq 5$ follow from the earlier checked $s>2k$, and $q> (k-1)^4$ is trivial if $k=2$ and otherwise follows from $q > ( k+1)^{2 (k-1)} $ which since $k\geq 3$ implies $q> (k+1)^4> (k-1)^4$. \end{proof}

\begin{proof}[Proof of Theorem \ref{intro-poly}]

Theorem \ref{intro-poly} is obtained from Theorem \ref{intro-curve} by setting $C = \mathbb P^1$. In this case, every line bundle of degree $e$ is isomorphic, so we take $L = \mathcal O(e [\infty])$ to be the line bundle whose sections are rational functions with poles of order at most $e$ at $\infty$ and no poles away from $\infty$. 

For each $n$ we may interpret the sections $H^0( \mathbb P^1, \mathcal O(n[\infty]))$ of the line bundle $\mathcal O(n[\infty])$ as the space of polynomials of degree $\leq n$ in $T$. To do this, let $T$ be the standard coordinate of $\mathbb P^1$, and observe that every rational function on $\mathbb P^1$ with poles of order at most $n$ at $\infty$ and no poles away from $\infty$ is a polynomial in $T$ of degree at most $n$.

In particular, the sections of $L$ are the polynomials in $T$ of degree $\leq e$, and the sections of $L^k=\mathcal O(ke[\infty])$ are the polynomials of degree $\leq ke$.  Thus the polynomial $f$ of degree $\leq ke$ gives a section of $L^k$. The closed points of $\mathbb P^1$ consist of the closed points of $\mathbb A^1$, which are in one-to-one correspondence with the primes $\pi\in \mathbb F_q[T]$, together with the point at $\infty$.

The local ring at $\pi$ is $\mathbb F_q[T]_{(\pi)}$, with uniformizer $\pi$, and we have $\mathbb F_q[T]_{(\pi)}/\pi^r = \mathbb F_q[T]/\pi^r$. We can choose our local generator of $L$ to be the section given by the polynomial $1$. When we do this, interpreting $f$ as an element of $\mathcal O_{C_v}$ is simply the embedding $\mathbb F_q[T] \to \mathbb F_q[T]_\pi$, so the formula for $\ell_v(f)$ matches the formula for $\ell_\pi(f)$. 

The local ring at $\infty$ is $\mathbb F_q[u]_{(u)}$ where $u=T^{-1}$ so $T=u^{-1}$. We can choose $u$ as the uniformizer of this local ring. The degree of the point $\infty$ is $1$, so we can drop the $\deg v$ term in $\ell_v$. We cannot choose $1$ again as a local generator of $L$ as $1$ is not a local generator at $\infty$. Instead, we choose $T^e$ as a local generator,   because every rational function with a pole of order at most $e$ at $\infty$ is $T^e$ times a rational function with a pole of order at most $0$ at $\infty$ and thus $T^e$ generates the local sections of $L$.  We interpret $f$ as an element of $\mathcal O_{C_v} = \mathbb F_q[u]_{(u)}$ by substituting $T = u^{-1}$ and then dividing by the $k$th power of this local generator, i.e. by substituting $T = u^{-1}$ and then multiplying by $u^{ke}$. Having done this, the formula for $\ell_v(f)$ now matches the formula for $\ell_\infty(f)$.

Combining these observations, $\operatorname{MT}_e $ becomes $q^{ e(s-k) +s-1  }  \ell_\infty (f)  \prod_{\substack{\pi \in \mathbb F_q[T] \\  \textrm{prime} }} \ell_\pi(f) $. \end{proof}

\section{Manin's conjecture for Fermat hypersurfaces}

Let $X$ be the hypersurface in $\mathbb P^n_{\mathbb F_q}$ defined by the equation $\sum_{i=0}^n x_i^d=0$ for fixed positive integers $n,d$. The goal of this section is to prove Theorem \ref{intro-manin}, estimating $\# \{ f\colon C \to X \mid \textrm{degree }e \}$ for a smooth projective curve $C$ of genus $g$ over $\mathbb F_q$. 

It turns out that this can be expressed as a sum over effective divisors $D$ on $C$. This sum involves the M\"obius function of a divisor $D$: We define the M\"obius function $\mu(D)$ to equal $0$ if $D$ has multiplicity $>1$ at any point and otherwise to equal $(-1)$ raised to a power equal to the number of closed points in $D$.

\begin{lemma}\label{manin-setup}Fix a finite field $\mathbb F_q$ and positive integers $n$ and $d$. Let $X$ be the hypersurface in $\mathbb P^n_{\mathbb F_q}$ defined by the equation $\sum_{i=0}^n x_i^d=0$. Let $C$ be a smooth projective curve of genus $g$ over $\mathbb F_q$.  For a nonnegative integer $e$ we have
\[\# \{ f\colon C \to X \mid \textrm{degree }e \}=  \frac{1}{q-1}  \sum_{\substack{D \textrm{ effective} }} \mu(D) \sum_{ \substack{L\textrm{ on }C \\ \textrm{degree }e- \deg(D) }} \bigl( \#{ \{ \mathbf a \in H^0(C, L)^{n+1} \mid \sum_{i=0}^n a_i^d=0\}} -1 \bigr). \]
\end{lemma}

\begin{proof} Every map $f$ of degree $e$ from $C$ to $X$ defines a line bundle of degree $e$ on $C$, that being the pullback of the line bundle $\mathcal O(1)$ from $\mathbb P^n$ to $C$. The space of global sections of the line bundle $\mathcal O(1)$ on $\mathbb P^n$ is $(n+1)$-dimensional, generated by sections corresponding to the coordinates $x_0,\dots,x_n$ of $\mathbb P^n$. Precisely, the correspondence is that the values of these sections at a point is the $(n+1)$-tuple of coordinates at the point. (The value of a tuple of sections of a line bundle at a point is well-defined only up to scalar multiplication, since we must fix an identification of the fiber of the line bundle at that point with the base field, and the projective coordinates are well-defined only up to scalar multiplication by definition, and these two indeterminacies match up.) Thus, these $n+1$ sections of $\mathcal O(1)$ on $\mathbb P^n$ do not all vanish at any point, as that point would have projective coordinates $(0:\dots : 0)$ which no point of projective space does, and restricted to $X$, the $d$'th powers of these sections sum to $0$ by the defining equation $\sum_{i=0}^n x_i^d=0$ of $X$. 

Thus, from $f$ we obtain an $(n+1)$-tuple of sections of the pullback line bundle, by pulling back the sections $x_0,\dots, x_n$ of $\mathbb P^n$. These sections cannot all vanish at the same point, and their $d$th powers sum to zero, since they are pullbacks of sections with the same two properties. The sections are well-defined up to the action of the automorphisms $\mathbb F_q^\times$ of the line bundle. Conversely, given a line bundle of degree $e$ and $n+1$ sections satisfying these two conditions, we obtain a map $C \to X$. This implies
\[ \#{\{ f \colon C \to X \mid \textrm{degree } e\} }= \frac{1}{q-1} \sum_{ \substack{L\textrm{ on }C \\ \textrm{degree }e}} \#{ \{ \mathbf x \in H^0(C, L)^{n+1} \mid \sum_{i=0}^n x_i^d=0,  \textrm{ no common zero}\}}.\]

The condition that $x_0,\dots,x_n$ have no common zero is the condition that for each closed point $v$ of $C$ the $x_i$ do not all vanish at $v$. This condition can be detected by M\"obius inversion, equivalently, inclusion-exclusion, as an alternating sum over divisors of $C$. We must first remove the tuples which are all zero since otherwise the sum over $D$ would be infinite. Hence
\[  \sum_{ \substack{L\textrm{ on }C \\ \textrm{degree }e}} \# \{ \mathbf x \in H^0(C, L)^{n+1}  \mid \sum_{i=0}^n x_i^d=0,  \textrm{ no common zero}\}\]
\[= \sum_{ \substack{L\textrm{ on }C \\ \textrm{degree }e}}  \sum_{\substack{D \textrm{ effective} }} \mu(D)  \#{ \{ \mathbf x \in H^0(C, L)^{n+1}  \mid \sum_{i=0}^n x_i^d=0,  \textrm{not all }0\textrm{, vanishing at each point of }D\}}\]
\[=  \sum_{ \substack{L\textrm{ on }C \\ \textrm{degree }e}}  \sum_{\substack{D \textrm{ effective} }} \mu(D)  \#{ \{ \mathbf a \in H^0(C, L(-D) )^{n+1} \mid \sum_{i=0}^n a_i^d=0, \textrm{ not all }0\}} \]
\[=  \ \sum_{ \substack{L\textrm{ on }C \\ \textrm{degree }e}}  \sum_{\substack{D \textrm{ effective} }} \mu(D) \bigl( \#{ \{ \mathbf a \in H^0(C, L(-D) )^{n+1} \mid \sum_{i=0}^n a_i^d=0\}}-1 \bigr) \]
\[=   \sum_{\substack{D \textrm{ effective} }} \mu(D) \sum_{ \substack{L\textrm{ on }C \\ \textrm{degree }e- \deg(D) }}   \bigl(\#{ \{ \mathbf a \in H^0(C, L)^{n+1} \mid \sum_{i=0}^n a_i^d=0\}} -1 \bigr). \]
 \end{proof}

To obtain Theorem \ref{intro-manin}, we will plug Theorem \ref{main-simplified} into Lemma \ref{manin-setup} to estimate the individual terms \[ \#{ \{ \mathbf a \in H^0(C, L)^{n+1} \mid \sum_{i=0}^n a_i^d=0\}}.\] The main term of Theorem \ref{main-simplified}, in these variables, is $\operatorname{MT}_{e-\deg D} $. Summing this main term turns out to give the main term of Theorem \ref{intro-manin}.

Recall that $\tau_C(X)=\prod_{v \in \abs{C} }  \Bigl ( (1 - q^{ -\deg v})\frac{ \# X(\mathbb F_{q^{\deg v} })}{ q^{(n-1) \deg v}} \Bigr)$

\begin{lemma}\label{manin-mt} Fix a finite field $\mathbb F_q$ and positive integers $n$ and $d$. Assume $p$ does not divide $d$ and that $n>\max(3,d)$. Let $X$ be the hypersurface in $\mathbb P^n_{\mathbb F_q}$ defined by the equation $\sum_{i=0}^n x_i^d=0$. Let $C$ be a smooth projective curve of genus $g$ over $\mathbb F_q$.  For a nonnegative integer $e$ we have
\[  \frac{1}{q-1}  \sum_{\substack{D \textrm{ effective} }} \mu(D) \sum_{ \substack{L\textrm{ on }C \\ \textrm{degree }e- \deg(D) }} \operatorname{MT}_{e-\deg D} \]
\[ = \tau_C(X) \frac{q^{e (n+1-d) + n(1-g) } \#{\operatorname{Pic}^0(C)}  }{q-1} .\] \end{lemma}

\begin{proof} We have (with $n>d$ ensuring absolute convergence so that these manipulations are justified)
\[  \sum_{\substack{D \textrm{ effective} }} \mu(D) \sum_{ \substack{L\textrm{ on }C \\ \textrm{degree }e- \deg(D) }} \operatorname{MT}_{e - \deg D} \]
\[  =  \sum_{\substack{D \textrm{ effective} }} \mu(D) \#\Pic^0(C) \operatorname{MT}_{e - \deg D} \]
\[  =  \sum_{\substack{D \textrm{ effective} }} \mu(D) \#\Pic^0(C) q^{ (e -\deg D) (n+1-d) +n (1-g)  } \prod_{\substack{v \in \abs{C} }} \ell_v(0) \]
\[ =  q^{e (n+1-d) + n(1-g) } \#{\operatorname{Pic}^0(C)}  \sum_{\substack{D \textrm{ effective} }} \mu(D)  q^{ -\deg D (n+1-d)   }   \prod_{\substack{v \in \abs{C} }}\ell_v(0) \]

\[ =  q^{e (n+1-d) + n(1-g) } \#{\operatorname{Pic}^0(C)} \prod_{v \in \abs{C} } (1 - q^{ - \deg v (n+1-d)})   \ell_v(0) .\]

Now for $v\in \abs{C}$, after defining \[ N_r(v) = \# \{ \mathbf b \in (\mathcal O_{C_v} /\pi_v^r )^{n+1}\mid \sum_{i=0}^n b_i^d \equiv 0 \bmod \pi^r\},\]  we have
\begin{equation}\label{limit-manipulation} \begin{aligned}
 (1 - q^{ - \deg v (n+1-d)}) \ell_v(0)&  \\
= (1 - q^{ - \deg v (n+1-d)})  \lim_{r\to\infty} & \frac{ N_r(v) }{ q^{r n \deg v} }\\
 =  \lim_{r\to\infty}&  \frac{N_r(v)}{ q^{r n \deg v} } -   q^{ - \deg v (n+1-d)} \lim_{r\to\infty}  \frac{ N_r(v)}{ q^{r n \deg v} }\\
 =  \lim_{r\to\infty}  &\frac{ N_r(v)}{ q^{r n \deg v} } -   q^{ - \deg v (n+1-d)} \lim_{r\to\infty}  \frac{ N_{r-d}(v) }{ q^{(r-d) n \deg v} }\\
 = \lim_{r\to\infty} \Bigl( & \frac{N_r(v)} { q^{r n \deg v} } -   q^{ - \deg v (n+1-d)} \frac{ N_{r-d}(v)}{ q^{(r-d) n \deg v} } \Bigr).\end{aligned}\end{equation}

For $\mathbf b \in (\mathcal O_{C_v} /\pi_v^r )^{n+1}$ such that $\sum_{i=0}^n b_i^d \equiv 0 \bmod \pi_v^r$, if  $b_0,\dots, b_n$ are all divisible by $\pi$ then $\frac{b_0}{\pi},\dots, \frac{b_n}{\pi}$ are all well-defined in $\mathcal O_{C_v}/\pi_v^{r-1}$ and satisfy  $\sum_{i=0}^n \left( \frac{b_i}{\pi}\right)^d =0 \bmod \pi_v^{r-d}$. Modding these out by $\pi_v^{r-d}$, they give a tuple in $\mathcal O_{C_v} /\pi_v^{r-d} $ of solutions to $\sum_{i=0}^n b_i^d \equiv 0 \bmod \pi_v^{r-d}$, and each solution to that equation defines $q^{ (d-1) (n+1) \deg v}$ solutions to $\sum_{i=0}^n b_i^d \equiv 0 \bmod \pi_v^{r}$. After defining
\[ N_r'(v) =  \# \{ \mathbf b \in (\mathcal O_{C_v} /\pi_v^r )^{n+1}\mid \sum_{i=0}^n b_i^d \equiv 0 \bmod \pi^r, \pi \nmid \gcd(b_0,\dots, b_n)\},\] the above $q^{(d-1)(n+1)\deg v}$-to-one map gives
\begin{align} N_r(v)= & \# \{ \mathbf b \in (\mathcal O_{C_v} /\pi_v^r )^{n+1}\mid \sum_{i=0}^n b_i^d \equiv 0 \bmod \pi_v^r\} \nonumber \\ =& \# \{ \mathbf b \in (\mathcal O_{C_v} /\pi_v^r )^{n+1}\mid \sum_{i=0}^n b_i^d \equiv 0 \bmod \pi_v^r, \pi \nmid \gcd(b_0,\dots, b_n)\} \nonumber \\ +&  \# \{ \mathbf b \in (\mathcal O_{C_v} /\pi_v^r )^{n+1}\mid \sum_{i=0}^n b_i^d \equiv 0 \bmod \pi_v^r, \pi \mid \gcd(b_0,\dots, b_n)\} \nonumber \\
=& N_r'(v)+q^{ (d-1) (n+1) \deg v}\# \{ \mathbf b \in (\mathcal O_{C_v} /\pi_v^{r-d} )^{n+1} \mid \sum_{i=0}^n b_i^d \equiv 0 \bmod \pi_v^{r-d} \} \nonumber\\
=&\label{divisible-or-not}  N_r'(v)+q^{ (d-1) (n+1) \deg v}N_{r-d}(v) \end{align}
From \eqref{divisible-or-not} we obtain
\begin{equation}\label{divisible-cancellation} \frac{N_r(v)} { q^{r n \deg v} } -   q^{ - \deg v (n+1-d)} \frac{ N_{r-d}(v)} { q^{(r-d) n \deg v} } =   \frac{ N_r'(v)}{q^{rn \deg v}} \end{equation}
after observing that  \[ (d-1)(n+1) - rn = - (n+1-d) - (r-d) n.\] Combining  \eqref{limit-manipulation} and \eqref{divisible-cancellation} we have 
\[ (1 - q^{ - \deg v (n+1-d)})  \ell_v(0)=  \lim_{r\to\infty}   \frac{N_r'(v)} { q^{r n \deg v} } = \frac{N_1'(v)} { q^{n \deg v} }    \]
\[ = \frac{ \# \{ b_0,\dots, b_n \in \mathbb F_{q^{\deg v}} \mid \sum_{i=0}^n b_i^d =0, b_0,\dots,b_n \textrm{ not all zero} \}}{ q^{n \deg v}}  = (1 - q^{ -\deg v}) \frac{ \# X(\mathbb F_{q^{\deg v} })}{ q^{(n-1) \deg v}} \]
since the limit is attained already at $r=1$ by Hensel's lemma (where we use $p \nmid d$ to ensure that $X$ is smooth)

Plugging this in gives
\[   q^{e (n+1-d) + n(1-g) } \#{\operatorname{Pic}^0(C)}  \prod_{v \in \abs{C} }  \Bigl ( (1 - q^{ -\deg v})\frac{ \# X(\mathbb F_{q^{\deg v} })}{ q^{(n-1) \deg v}} \Bigr).  \] and then we may add back in the factor of $\frac{1}{q-1}$ to obtain the statement.\end{proof}

However, a subtlety is that Theorem \ref{main-simplified} may only be applied to estimate \[\sum_{ \substack{L\textrm{ on }C \\ \textrm{degree }e- \deg(D) }}  \#{ \{ \mathbf a \in H^0(C, L)^{n+1} \mid \sum_{i=0}^n a_i^d=0\}}\] if $e - \deg(D) > 2g-2$. If $ e- \deg(D) \in [0,2g-2]$, we will instead use a ``trivial bound" based on upper bounding  $\# \{ \mathbf a \in H^0(C, L)^{n+1} \}$ using Clifford's theorem. If $e -\deg(D) <0$ then $H^0(C,L)$ consists only of the zero vector so $  \#\{ \mathbf a \in H^0(C, L)^{n+1} \mid \sum_{i=0}^n a_i^d=0\}=1$ and these terms cancel. Breaking up into different ranges, we will obtain an estimate with several different error terms, and we will then check in turn that each error term is smaller than the main term.

\begin{lemma}\label{manin-splitting} Fix a finite field $\mathbb F_q$ of characteristic $p$ and positive integers $n$ and $d$ such that $2\leq d<p$. Let $X$ be the hypersurface in $\mathbb P^n_{\mathbb F_q}$ defined by the equation $\sum_{i=0}^n x_i^d=0$. Let $C$ be a smooth projective curve of genus $g$ over $\mathbb F_q$. Assume that
\[ q > ( d+1)^{2 (d-1)} \] and
 \[ n+1 > \max\left(  \frac{2 (d- \gamma_{d,p} -2 \frac{\log (d+1)}{\log q})}{ 1-\gamma_{d,p}  -2\frac{\log (d+1)}{\log q} } ,  \frac{2d}{ 1- 2 (d-1)\frac{\log (d+1)}{\log q} } \right) .\] 
 
 Then there exists $\theta>0$ such that we have
\[ \Bigl| \# \{ f\colon C \to X \mid \textrm{degree }e \} - \tau_C(X) \frac{q^{e (n+1-d) + n(1-g) } \#{\operatorname{Pic}^0(C)}  }{q-1}\Bigr| \leq \]
\[O_{n,d,q,g,\theta}  \Bigl( \frac{\#{\operatorname{Pic}^0(C)}  }{q-1}  \sum_{\substack{D \textrm{ effective} \\ \deg D < e+\min(2-2g,1)  }} q^{  ( n+1-d-\theta)(e- \deg D ) } \Bigr)  \]
\[ +   \frac{\#{\operatorname{Pic}^0(C)}  }{q-1}  \sum_{\substack{D \textrm{ effective} \\ e +2-2g \leq  \deg D \leq e  }}  q^{  ( n+1)\frac{ e- \deg D+2 }{2}  }+ \Bigl|   \frac{\#{\operatorname{Pic}^0(C)}  }{q-1}   \sum_{\substack{D \textrm{ effective} \\ \deg D < e+\min(2-2g,1)  }} \mu(D) \Bigr| \] \[+ \Bigl|  \frac{\#{\operatorname{Pic}^0(C)}  }{q-1}  \sum_{\substack{D \textrm{ effective} \\ \deg D \geq  e+\min(2-2g,1)   }} \mu(D)  \operatorname{MT}_{e- \deg D} \Bigr|.\]
\end{lemma}

\begin{proof} For a line bundle $L$ let \[ N(L) = \#\{\mathbf a \in H^0(C, L)^{n+1} \mid \sum_{i=0}^n a_i^d=0\}.\]

We let $S_{e}$ be the set of pairs $(D,L)$ with $D$ an effective divisor and $L$ a line bundle on $C$ of degree $e-\deg D$. For compactness of notation, for any function $\Phi$ of a line bundle $L$ we  define
\[  \sum_{(D,L)\in S_e}^\mu \Phi(L)= \sum_{(D,L)\in S_e} \mu(D) \Phi(L).\]
  We apply Lemma \ref{manin-setup} to obtain, in this new notation,
\begin{equation}\label{m-s-eq} \# \{ f\colon C \to X \mid \textrm{degree }e \}=  \frac{1}{q-1}  \sum_{\substack{(D,L)\in S_e}}^\mu(N(L) -1). \end{equation} If $ \deg (D)>e$ then for $L$ of degree $e - \deg(D)<0$ we have $H^0(C, L) =\{0\}$ so $\{\mathbf a \in H^0(C, L)^{n+1} \mid \sum_{i=0}^n a_i^d=0\} = \{ 0^{n+1} \}$ has cardinality $1$, i.e. $N(L)=1$. Thus
\begin{equation} \begin{aligned}\label{ms-1} \sum^\mu_{\substack{(D,L)\in S_e}} & (N(L) -1 )=  \sum_{\substack{(D,L)\in S_e \\ \deg D < e+\min(2-2g,1)   }}^\mu &(N(L) -1 )  + \sum_{\substack{(D,L)\in S_e\\ e +2-2g \leq  \deg D \leq e  }}^\mu& (N(L)-1 ) \end{aligned} \end{equation} 
as the terms with $\deg D>e$ vanish. For the second term of \eqref{ms-1}, we observe that 
\begin{equation}\label{middle-bound-clifford} 0 \leq N(L)-1  \leq N(L) \leq  \# \{ \mathbf a \in H^0(C, L)^{n+1} \} = q^{ (n+1) \dim H^0(C,L)} \leq q^{  ( n+1)\frac{ e- \deg D+2 }{2}  }\end{equation} by Clifford's theorem in the form of Lemma \ref{clifford}(2). For the first term of \eqref{ms-1}, we first split off the $-1$, getting
\begin{equation}\label{ms-2} \sum^\mu_{\substack{(D,L)\in S_e  \\ \deg D < e+\min(2-2g,1)   }} \bigl(N(L) -1 \bigr) 
 = \sum^\mu_{\substack{(D,L)\in S_e  \\ \deg D < e+\min(2-2g,1) }}   N(L)  - \sum^\mu_{\substack{(D,L)\in S_e  \\ \deg D < e+2-2g  }}1\end{equation}
 
  Theorem \ref{main-simplified} gives
 \[N(L) = \operatorname{MT}_{e -\deg D} +  O_{n,d,q,g,\theta} (  q^{  ( n+1-d-\theta)(e- \deg D ) } )\]
 which summed over $D$ and $L$ gives 
  \begin{equation}\label{ms-3}  \sum^\mu_{\substack{(D,L)\in S_e  \\ \deg D < e+\min(2-2g,1)   }}N(L)
 = \sum^\mu_{\substack{(D,L)\in S_e \\ \deg D < e+\min(2-2g,1)   }}   \operatorname{MT}_{e -\deg D}  + O_{n,d,q,g,\theta}\Bigl(  \sum_{\substack{(D,L)\in S_e  \\ \deg D < e+\min(2-2g,1)   }}  q^{  ( n+1-d-\theta)(e- \deg D ) } \Bigr) \end{equation}

Adding terms back in  gives
\begin{equation}\label{ms-4} 
  \sum^\mu_{\substack{(D,L)\in S_e  \\ \deg D < e+\min(2-2g ,1)  }}   \operatorname{MT}_{e -\deg D} 
= \sum_{(D,L)\in S_e }^\mu \operatorname{MT}_{e -\deg D}  -  \sum^\mu_{\substack{(D,L)\in S_e \\ \deg D \geq  e+\min(2-2g ,1) }}  \operatorname{MT}_{e -\deg D}  \end{equation}
and Lemma \ref{manin-mt} gives
\begin{equation}\label{ms-5} \sum_{\substack{(D,L)\in S_e }}^\mu \operatorname{MT}_{e -\deg D}  = \tau_C(X) \frac{q^{e (n+1-d) + n(1-g) } \#{\operatorname{Pic}^0(C)}  }{q-1}.  \end{equation}
Combining \eqref{m-s-eq}, \eqref{ms-1}, \eqref{middle-bound-clifford}, \eqref{ms-2}, \eqref{ms-3}, and \eqref{ms-4}, and adding back the factor of $\frac{1}{q-1}$, we obtain the statement, after observing that in each case the term summed over $(D,L) \in S_e$ and thus we may replace the sum over $S_e$ with a sum over $D$ times the number $\#\Pic^0(C)$ of $L$ for a given $D$. 
\end{proof}


\begin{lemma}\label{met-main}  Fix a finite field $\mathbb F_q$ and positive integers $n,d$ with $n>d$. Let $C$ be a smooth projective curve of genus $g$ over $\mathbb F_q$.   Let $\theta>0$ be a real number. For a nonnegative integer $e$ we have
\[  \sum_{\substack{D \textrm{ effective} \\ \deg D < e+\min(2-2g ,1)  }} q^{  ( n+1-d-\theta)(e- \deg D ) } = o_{n,d,q,g,\theta} ( q^{ (n+1-d)e}) .\]\end{lemma}

\begin{proof} Since $n>d$ and $\theta>0$, we can always replace $\theta$ with a smaller value such that $\theta<n-d$ and $\theta>0$ . Since this can only grow the left-hand side, it suffices to handle the case $\theta<n-d$. We have \[ \sum_{\substack{D \textrm{ effective} \\ \deg D< e+ \min(2-2g ,1) }} q^{  ( n+1-d-\theta)(e- \deg D ) } \leq \sum_{D \textrm{ effective}  } q^{  ( n+1-d-\theta)(e- \deg D ) } \] \[= q^{ (n+1-d-\theta)e} \prod_{v \in \abs{C}} \frac{1}{1 - q^{ - (n+1-d-\theta) \deg v} }= q^{ (n+1-d-\theta)e}  \zeta_C ( q^{ - (n+1-d-\theta )}) \] which is $\ll q^{ (n+1-d-\theta)e } $ and thus is $o( q^{(n+1-d )e})$ since $\theta < n-d$ ensures $\zeta_C(q^{n+1-d-\theta})=O_{n,d,q,g,\theta}(1)$.
 \end{proof}

\begin{lemma}\label{met-middle} Fix a finite field $\mathbb F_q$ and positive integers $n,d$ with $n>d$. Let $C$ be a smooth projective curve of genus $g$ over $\mathbb F_q$. We have
\[ \sum_{\substack{D \textrm{ effective} \\ e +2-2g \leq  \deg D \leq e  }}   q^{  ( n+1)\frac{ e- \deg D+2 }{2}  }= o_{n,q,g} ( q^{ (n+1-d)e}) .\]\end{lemma}

\begin{proof}  Since $e-\deg D\in [0, 2g-2]$ is bounded, the term $ q^{  ( n+1)\frac{ e- \deg D+2 }{2}  }$ is $O_{n,g,q}(1)$. So it suffices to show that  \[\sum_{\substack{D \textrm{ effective} \\ e +2-2g \leq  \deg D \leq e  }} 1= o_{n,q,g} ( q^{ (n+1-d)e}) .\]
But the left-hand side is bounded by the number of effective divisors of degree at most $e$, which by a zeta function argument is $O_{q,g}(q^{e})$, and hence is $o_{q,g}(q^{ (n+1-d)e})$ since $n>d$.\end{proof}

\begin{lemma}\label{met-ones}  Fix a finite field $\mathbb F_q$ and positive integers $n,d$ with $n>d$. Let $C$ be a smooth projective curve of genus $g$ over $\mathbb F_q$. We have
\[  \sum_{\substack{D \textrm{ effective} \\ \deg D < e+2-2g  }} \mu(D)  = o_{q,g} ( q^{ (n+1-d)e}) .\]\end{lemma}

\begin{proof}  The left-hand side is bounded by the number of effective divisors of degree at most $e+2-2g$, which by a zeta function argument is $O_{q,g} (q^{e+2-2g})=O_{q,g}(q^e)$, and hence is $o_{q,g}(q^{ (n+1-d)e})$ since $n>d$. \end{proof}

\begin{lemma}\label{met-tail}  Fix a finite field $\mathbb F_q$ and positive integers $n,d$ with $n>d$. Let $C$ be a smooth projective curve of genus $g$ over $\mathbb F_q$. We have
\[ \frac{1}{q-1}  \sum_{\substack{D \textrm{ effective} \\ \deg D \geq  e+2-2g  }} \mu(D)  \operatorname{MT}_{e - \deg D} = o_{n,d,q,g} ( q^{ (n+1-d)e}) .\]\end{lemma}

\begin{proof}We have $\operatorname{MT}_{e-\deg D} = q^{ (e -\deg D) (n+1-d) +n (1-g)  } \prod_{\substack{v \in \abs{C} }} \ell_v(0)$ but the  $\prod_{\substack{v \in \abs{C} }} \ell_v(0)$ and  $q^{n(1-g)}$ factors are $O_{n,d,q,g}(1)$ and can be ignored. So it suffices to prove
\[  \sum_{\substack{D \textrm{ effective} \\ \deg D \geq  e+2-2g  }} \mu(D)  q^{ (e -\deg D) (n+1-d)   } = o_{n,d,q,g} ( q^{ (n+1-d)e}) \]
which dividing both sides by $q^{e (n+1-d)}$ is equivalent to
\[  \sum_{\substack{D \textrm{ effective} \\ \deg D \geq  e+2-2g  }} \mu(D)  q^{ -\deg D (n+1-d)   } = o_{n,d,q,g} ( 1) .\]
We have
\[ \Bigl| \sum_{\substack{D \textrm{ effective} \\ \deg D \geq  e+2-2g  }} \mu(D)  q^{ -\deg D (n+1-d)   } \Bigr| \leq \sum_{\substack{D \textrm{ effective} \\ \deg D \geq  e+2-2g  }}\left| \mu(D) \right|  q^{ -\deg D (n+1-d)   }  \leq   \sum_{\substack{D \textrm{ effective} \\ \deg D \geq  e+2-2g  }}  q^{ -\deg D (n+1-d)   } \]
so it suffices to prove that $\sum_{\substack{D \textrm{ effective} \\ \deg D \geq  e+2-2g  }}  q^{ -\deg D (n+1-d)   } =o(1)$, which is equivalent to the convergence of the sum $\sum_{\substack{D \textrm{ effective} }} q^{ -\deg D (n+1-d)   }  = \zeta_C ( q^{ - (n+1-d)})$, which indeed converges as $n>d$ so $n+1-d>1$.
   \end{proof}

\begin{proof}[Proof of Theorem \ref{intro-manin}] This follows upon combining Lemmas \ref{manin-splitting}, \ref{met-main}, \ref{met-middle}, \ref{met-ones}, and \ref{met-tail}, observing in each case that $\frac{ \#\Pic^0(C)}{q-1} = O_{q,g}(1)$ and can be ignored. \end{proof}

\section{The singular locus in the arbitrary hypersurface case}\label{s-arbitrary}

Let $F$ be a polynomial of degree $d$ in $n+1$ variables $x_0,\dots, x_n$ whose vanishing locus in $\mathbb P^n$ is a smooth hypersurface $X$. Assume that the characteristic of $\mathbb F_q$ does not divide $d$. We consider in this section what happens if we apply similar techniques to those in the remainder of the paper to count maps from $C$ to $X$ of degree $e$ or to count tuples $a_0,\dots, a_n$ in $H^0(C,L)$ such that $F(a_0,\dots, a_n)$ takes any fixed value.

In this case, for $\alpha \in H^0(C, L^d)^\vee$, the relevant exponential sum is
\[ S^F(\alpha) =  \sum_{a_0,\dots, a_n \in H^0(C,L)} \psi ( \alpha (F(a_0,\dots, a_n))),\]

To estimate this sum, the analogue of Lemma \ref{minor-katz} concerns the singular locus 
\[\operatorname{Sing}^F_\alpha = \{ a_0,\dots, a_n  \in H^0(C, L)\mid  \alpha(b  \frac{\partial F}{\partial x_i} (a_0,\dots, a_n ) )=0 \textrm{ for all }i \in \{0,\dots,n\}, b\in H^0(C, L) \},\] which is a closed subscheme of $H^0(C, L)^{n+1}$ for a line bundle $L$ of degree $e$.

It turns out to be more natural to estimate the dimension of $\operatorname{Sing}^F_\alpha$ ``on average" over different $\alpha$, i.e. to understand
\[ \hspace{-.5in} \operatorname{Sing}^{F,m} = \{ Z \subset C, \overline{\alpha} \in H^0(Z, L^d)^\vee, a_0,\dots, a_n \in H^0(C,L) \mid \deg Z = m ,  \overline{\alpha } \textrm{ primitive}, (a_0,\dots, a_n)\in \operatorname{Sing}^F_{\overline{\alpha} }  \} \] where we abuse notation slightly by using $\overline{\alpha}$ to refer to the linear form $H^0(C, L^d) \to H^0(D, L^d) \to \mathbb F_q$ induced by $\overline{\alpha}$ in $\operatorname{Sing}^F_{\overline{\alpha} } $   

We now recall notation from the introduction. Let $\nabla F$ be the tuple of polynomials $\frac{\partial F}{\partial x_0},\dots, \frac{\partial F}{\partial x_n}$. Since $\nabla F$ is an $(n+1)$-tuple of polynomials in $n+1$ variables with no common zeroes except $0$, $\nabla F$ defines a map $\mathbb P^n \to \mathbb P^n$. Let $Y$ be the blowup of $\mathbb P^n \times \mathbb P^n$ along the graph of this map. Let $E$ be the exceptional divisor of this blowup. 

We will relate $\operatorname{Sing}^{F,m}$ to moduli spaces of maps from $C$ to $Y$. The degree of a map $C \to Y$ can be expressed as a triple of nonnegative integers: in order, the degree of the induced map to the first $\mathbb P^n$ (the source of the map $\nabla F$), the degree of the induced map to the second $\mathbb P^n$ (the target of the map $\nabla F$), and the degree of intersection with the exceptional divisor. We let $\operatorname{Mor}_{ i_1,i_2,i_3} (C ,Y)$ be the moduli space parameterizing morphisms $C \to Y$ whose degree is the triple $(i_1,i_2,i_3)$, and we let $\operatorname{Mor}'_{ i_1,i_2,i_3} (C ,Y)$ be the open subspace parameterizing only those maps whose image is not entirely contained in the exceptional divisor $E$.

The goal of this subsection is to prove the following proposition.

\begin{prop}\label{general-dim-bound} For $n,d,e,m$ natural numbers satisfying $d>2$,  \[ n+1 > 2d,\]
\[e> \max ( (4g-4 )\frac{d-1}{d-2}  + \frac{2d}{d-2},  \frac{ 2 ( (n+3)d-4 )}{( n+1-2d)  (d-2)} ) \] and \[ m \leq \frac{de}{2}+1,\]  we have
\[ \dim  \operatorname{Sing}^{F,m} \leq \max \Bigl( \max_{ \substack{  0 \leq j_1 \leq e \\ 0 \leq j_2 \leq m+ 2g-2-e  \\  j_3 \leq m \\ j_3 \leq j_2 \\  j_3 \leq (d-1) j_1}}   \dim \operatorname{Mor}'_{e-j_1 , m+2g-2-e-j_2, m-j_3 } (C ,Y) + 2+ j_1 + j_2, \] \[ \hspace{-.5in} m+1+ \lceil \frac{m}{d-1} \rceil + (n+1)  ( e+1 - g- \lceil \frac{m}{d-1} \rceil ) , m+1+ \frac{ de-m-2g+2}{d-1} + (n+1) ( e+1-g- \frac{ de-m-2g+2}{d-1}) , 2g+ 2m  \Bigr) .\]
\end{prop} 

We note that the expected dimension of $\operatorname{Mor}'_{ i_1,i_2,i_3} (C ,Y)$ is $(n+1) i_1 + (n+1) i_2  - (n-1) i_3 -   2n (g-1) $ since the anticanonical divisor is $n+1$ times the hyperplane class of the first $\mathbb P^n$ plus $n+1$ times the hyperplane class of the second $\mathbb P^n$ minus $(n-1)$ times the exceptional divisor and the dimension is $2n$.  If the true dimension is equal to the expected dimension, then the maximum over $j_1,j_2,j_3$ is attained for $j_1=j_2=j_3=0$ at a dimension bound of \[ (n+1) (m+2g-2) - (n-1) m -2n (g-1)= 2m +2g-2\] which is dominated by the other terms.

\begin{remark} To make use of Proposition \ref{general-dim-bound}, we could use the bound on $\dim  \operatorname{Sing}^{F,m}$ to bound $\operatorname{Sing}^F_\alpha $ for typical $\alpha$, and obtain a bound for $S^F(\alpha)$. To obtain interesting arithmetic consequences, we need the sum over $\alpha$ of $S^F(\alpha)$ to be dominated by $S^F(0)$. In this remark, we will explain under what conditions that might be possible.

The analogue of Lemma \ref{minor-katz} will give a bound for $S^F(\alpha)$ of \[q^{ \frac{ (n+1) (e+1-g) + \dim  \operatorname{Sing}^F_\alpha +1}{2}}\] times a Betti number bound factor. If we stratify the space of possible linear forms $\alpha$ of degree $m$ by $\dim \operatorname{Sing}^F_\alpha$, then a codimension $c$ stratum will consist of $\alpha$ with $\dim \operatorname{Sing}^F_\alpha \leq \dim \operatorname{Sing}^{F,m} +c -2m$. The number of points in the codimension $c$ stratum will be $q^{2m-c}$ times a Betti number bound factor, so the total contribution of this stratum is
\[ q^{2m-c} q^{ \frac{ (n+1) (e+1-g) + \dim  \operatorname{Sing}^{F,m}+c-2m +1}{2}}\] which is maximized for $c=0$ with a value of
\[q^{  \frac{ (n+1) (e+1-g) + \dim  \operatorname{Sing}^{F,m}+2m +1}{2}}.\]
Since $S^F(0) = q^{ (n+1)(e+1-g)}$, a bound for the sum over $\alpha$ of $S^F(\alpha)$ of the form $ q^{ (n+1)(e+1-g) - \frac{ e \delta}{2}}$  times a Betti number factor would suffice for the main term to dominate the error term as long as the Betti number bound is exponential and $q$ is sufficiently large, which are reasonable assumptions to make. This requires
\begin{equation}\label{wanted-dim-bound} \dim  \operatorname{Sing}^{F,m} \stackrel{?}{<}  (n+1) ( e+1-g) - 2m -e \delta +O(1) \textrm{ for all } m\in (e+1-2g, \frac{de}{2}+1].\end{equation}
A linear programming calculation (Corollary \ref{lin-prog}) shows that we obtain \eqref{wanted-dim-bound} for $e$ sufficiently large as long as $n> 3d-3+(d-1)\delta $ and for all tuples $i_1,i_2,i_3$ of nonnegative integers such that $i_1 +i_2 \leq i_3 +2g-2$ and $d (i_2+1-2g) \leq (d-2) (i_3-1)$ we have
\begin{equation}\label{i1i2i3-remark}  \dim \operatorname{Mor}'_{i_1,i_2, i_3 } \leq O(1) + 
 \begin{cases}  (n+2-\delta) i_1   +i_2 -3 i_3  & \textrm{if } i_3 \leq \frac{ d i_1}{2}+1  \\
(n+2 - \frac{3d}{2} -\delta)i_1   + i_2 & \textrm{if } i_3 > \frac{ d i_1}{2} +1 \textrm{ and } \frac{(d-2)i_1}{2} +2g-1 \geq i_2 \\ 
i_1 +  \frac{ 2 (n+1-\delta) -2d -2}{d-2}  i_2  & \textrm{if }  \frac{(d-2)i_1}{2} +2g-1 < i_2   \\\end{cases}. \end{equation}  

Depending on $n$, \eqref{i1i2i3-remark} may be considerably weaker than the claim that $\dim \operatorname{Mor}'_{ i_1,i_2,i_3} (C ,Y)= (n+1) i_1 + (n+1) i_2  - (n-1) i_3 -   2n (g-1) $, but it is not clear when we can establish this weaker statement.

Even with strong assumptions on $\dim \operatorname{Mor}'_{ i_1,i_2,i_3} (C ,Y)$, the lower bound on $n$ is worse than in the Fermat case because the estimate on the $\ell^2$ norm of $S(\alpha)$ that is used in the Fermat case does not have an analogue in the general case. It might be possible to rectify this by understanding the dimensions of the individual $\dim \operatorname{Sing}^{F}_\alpha$, showing that in fact in low codimension strata the dimensions are smaller than expected from  $\dim \operatorname{Sing}^{F,m} $.
\end{remark}

%
%

\vspace{15pt}

We now begin the proof of Proposition \ref{general-dim-bound}. We will keep the assumptions of Proposition \ref{general-dim-bound} on $n,d,m,e$ throughout. We first introduce a modified form of the singular locus. Let
\[\widetilde{ \operatorname{Sing}}^{F,m} = \{ Z \subset C, \tilde{\alpha} \in H^0(Z, K_C(Z) \otimes L^{-d}), a_0,\dots, a_n \in H^0(C,L) , c_0,\dots, c_n \in H^0(C, K_C(Z) \otimes L^{-1}) \mid \]
\[\deg Z = m , \tilde{ \alpha} \textrm{ invertible}, c_i \mid_Z = \tilde{\alpha}\frac{  \partial F}{\partial x_i} (a_0,\dots, a_n) \textrm{ for all }i \} \]

\begin{lemma}\label{general-dim-lift}We have
\[  \dim \operatorname{Sing}^{F,m} =\dim  \widetilde{ \operatorname{Sing}}^{F,m} .\] \end{lemma}

\begin{proof} Lemma \ref{linear-is-residue} and an argument identical to Lemma \ref{exists-c} together imply that each primitive $\overline{\alpha} \in H^0(Z, L^{d})^\vee$ is associated to an invertible $\tilde{\alpha} \in H^0 (Z, K_C( Z) \otimes L^{-d} )$ and  we have $ (a_0,\dots, a_n)\in \operatorname{Sing}^F_\alpha$ if and only if there exist $c_0,\dots, c_n \in H^0 (C, K_C(Z) \otimes L^{-1})$ such that  $c_i \mid_Z = \tilde{\alpha} \frac{\partial F}{\partial x_i} (a_0,\dots, a_n)$ for all $i$ from $0$ to $n$. \end{proof}

We divide $ \widetilde{ \operatorname{Sing}}^{F,m}$ into various locally closed subsets and bound the dimension of each one.

\begin{lemma}\label{general-c0} The dimension of the locus in $ \widetilde{ \operatorname{Sing}}^{F,m}$ where $c_0,\dots, c_n=0$ is 
\[m + \lceil \frac{m}{d-1} \rceil + (n+1) (e+1-g- \lceil \frac{m}{d-1} \rceil).\]\end{lemma}

\begin{proof} On this locus, the condition $c_i \mid_Z = \tilde{\alpha}\frac{ \partial F}{\partial x_i} (a_0,\dots, a_n)$ simply forces  $\frac{\partial F}{\partial x_i} (a_0,\dots, a_n) \mid_Z =0$. In particular, the condition is independent of the choice of $\tilde{\alpha}$. Furthermore, since $F$ defines a smooth hypersurface, we have $\frac{\partial F}{\partial x_i} (a_0,\dots, a_n) \mid_Z =0$ for all $i$ if and only if $a_0,\dots,a_n$ all vanish at each point of $Z$ to order at least $\frac{1}{d-1}$ times the multiplicity of that point in $Z$. The ``if" is relatively clear and the ``only if" follows from the fact that if $a_0,\dots,a_n$ do not all vanish on a point then  $\frac{\partial F}{\partial x_i} (a_0,\dots, a_n) $ cannot vanish at that point for all $i$. It follows that the dimension of the locus in $ \widetilde{ \operatorname{Sing}}^{F,m}$ where $c_0,\dots, c_n=0$ is 
\[m+\dim   \{ Z \subset C, a_0,\dots, a_n \in H^0(C,L)   \mid \deg Z = m ,  \gcd(a_0,\dots,a_n)^{d- 1} \mid_Z =0 \}. \]

For the tuple $a_0,\dots,a_n$ all zero the space of possible $Z$ has degree $m$. For any other tuple $a_0,\dots, a_n$ there are finitely many possible $Z$, and such a $Z$ exists if and only if $ \gcd(a_0,\dots,a_n)$ has degree at least $\frac{m}{d-1}$. Hence the dimension of the locus in $ \widetilde{ \operatorname{Sing}}^{F,m}$ where $c_0,\dots, c_n=0$ is 
\[\max(2m, m+\dim   \{ a_0,\dots, a_n \in H^0(C,L)    \mid   \deg \gcd(a_0,\dots,a_n) \geq \frac{m}{d-1}  \} ) . \]
To obtain a tuple $a_0,\dots, a_n$ whose gcd has degree at least $\frac{m}{d-1} $, we can choose a divisor $D$ of degree $\lceil \frac{m}{d-1} \rceil $ and then choose $n+1$ sections of $H^0(C,L(-D))$. Every such tuple arises this way from at least one $D$.  The dimension of the space of divisors of degree $\lceil \frac{m}{d-1} \rceil $ is $\lceil \frac{m}{d-1} \rceil$ and each such divisor has a space of global sections of dimension at most $e+1-g- \lceil \frac{m}{d-1} \rceil$ since we have
\begin{equation}\label{eminm}  e- \lceil \frac{m}{d-1} \rceil \geq  e- \lceil \frac{ \frac{de}{2}+ 1 }{d-1} \rceil  \geq e-  \frac{ \frac{de}{2}+ 1 }{d-1}  -1= e \frac{ d-2}{2d-2} - \frac{d}{d-1} > 2g-2 \end{equation} by assumption on $e$.
So the total dimension of the locus in $ \widetilde{ \operatorname{Sing}}^{F,m}$ where $c_0,\dots, c_n=0$ is
\[ \max( 2m, m + \lceil \frac{m}{d-1} \rceil + (n+1) (e+1-g- \lceil \frac{m}{d-1} \rceil )   ) .\] 
We furthermore have
\[ (n+1) (e+1-g- \lceil \frac{m}{d-1} \rceil )   > (n+1) ( e \frac{ d-2}{2d-2} - \frac{d}{d-1} +1-g) > \frac{n+1}{2} ( e \frac{ d-2}{2d-2} - \frac{d}{d-1} )\]
while \[ m - \lceil \frac{m}{d-1} \rceil \leq \frac{m (d-2)}{d-1} \leq e \frac{ d (d-2)}{2d-2} + \frac{d-2}{d-1} \] and by assumption on $e$ we have  \[e \frac{ d (d-2)}{2d-2} + \frac{d-2}{d-1} < \frac{n+1}{2} ( e \frac{ d-2}{2d-2} - \frac{d}{d-1} )\] so the maximum is always 
\[m + \lceil \frac{m}{d-1} \rceil + (n+1) (e+1-g- \lceil \frac{m}{d-1} \rceil).\qedhere\] \end{proof}

On the complementary locus where some $c_i\neq 0$, we must have $c_i \mid_Z \neq 0$, which forces us to have some $a_i\neq 0$. Then $a_0,\dots,a_n$ define a map $C \to \mathbb P^n$ of degree $e-j_1$, where $j_1$ is the total degree of the common vanishing locus of $a_0,\dots,a _n$. Similarly $c_0,\dots, c_n$ define a map $C \to \mathbb P^n$ of degree $2g-2+m-e-j_2$, where $j_2$ is the total degree of the common vanishing locus of $ c_0,\dots, c_n$. Combining these maps, we obtain a map $C \to \mathbb P^n \times \mathbb P^n$. There are two possibilities: either the image of this map is contained in the graph of $\nabla F$, or not.

\begin{lemma}\label{general-contained-comp} The total dimension of the locus in $ \widetilde{ \operatorname{Sing}}^{F,m}$ where some $c_i\neq 0$ and the induced map $C \to \mathbb P^n \times \mathbb P^n$ has image in the graph of $\nabla F$ is at most
\[ \max_{ \substack{ 0 \leq j_1 \leq e \\ 0 \leq j_2 \leq 2g-2+m-e \\ 2g-2+m-e-j_2 = (d-1) (e-j_1)}} m+j_1+ j_2+1 + (n+1)  ( \max( e+1-g-j_1,  \frac{e+2-j_1}{2}, 0)) .\] \end{lemma}

\begin{proof} If the image of the map $C \to  \mathbb P^n \times \mathbb P^n$ is contained in the graph of $\nabla F$, then the map $C \to \mathbb P^n$ defined by $c_0,\dots,c_n$ is obtained as the composition of the map $C \to \mathbb P^n$ defined by $a_0,\dots,a_n$ with $\nabla F$. In particular, this forces \[ 2g-2+m-e-j_2 = (d-1) (e-j_1) .\]  To choose a point in this locus, we first choose the common vanishing locus of the $a_0,\dots,a_n$, a divisor $D$ on $C$ of degree $j_1$. We then choose the $a_0,\dots,a_n$ as sections of $H^0 (C, L(-D))$, nowhere all vanishing.  Applying $\nabla F$, we get a tuple of sections of $H^0(C, L^{d-1} (- (d-1) D))$, nowhere all vanishing, and then multiply them all by a section of the line bundle $K_C(Z + (d-1)D) \otimes L^{-d}$ of degree $j_2$ to obtain $c_0,\dots, c_n$. The space of choices of $Z$ and $\tilde{\alpha}$ compatible with a given $a_0,\dots,a_n, c_0,\dots,c_n$ has dimension at most $m$: For each point $v$ of $Z$ outside the common vanishing locus of the $a_i$, the equation $c_i \mid_Z = \tilde{\alpha} \frac{\partial F}{\partial x_i}(a_0,\dots,a_n)$ determines the restriction of $\tilde{\alpha}$ to $v$. The dimension of the space of valid choices of $\tilde{\alpha}$ for a given $Z$ is thus at most the total multiplicity of the points of $Z$ that are also vanishing points of $a_i$. However, the dimension of the space of divisors of degree $m$ where the total multiplicity of a given finite set is $r$ is $m-r$, so the total dimension is at most $m$ regardless of the choice of total multiplicity. 

Adding together this dimension $m$, the dimension $j_1$ of the space of divisors of degree $j_1$, the maximum dimension $j_2+1$ of the space of sections of a line bundle of degree  $j_2$, and $n+1$ times the maximum possible dimension $\max( e+1-g-j_1,  \frac{e+2-j_1}{2}, 0)$ of the space of sections of $H^0 (C, L(-D))$, we obtain the stated formula.\end{proof}

\begin{lemma}\label{general-contained-simp} We have
\[ \max_{ \substack{ 0 \leq j_1 \leq e \\ 0 \leq j_2 \leq 2g-2+m-e \\ 2g-2+m-e-j_2 = (d-1) (e-j_1)}} m+j_1+ j_2+1 + (n+1)  ( \max( e+1-g-j_1,  \frac{e+2-j_1}{2}, 0)) \]
\[\hspace{-.9in} \leq \max ( m+1+ \lceil \frac{m}{d-1} \rceil + (n+1)  ( e+1 - g- \lceil \frac{m}{d-1} \rceil ) , m+1+ \frac{ de-m-2g+2}{d-1} + (n+1) ( e+1-g- \frac{ de-m-2g+2}{d-1}) , 2g+2m) .\] \end{lemma}

\begin{proof} Since $n+1  \geq 2d$, increasing $j_1$ by $1$ and $j_2$ by $d-1$ always reduces the expression \begin{equation}\label{general-contained-expression}  m+1+j_1+ j_2 + (n+1)  ( \max( e+1-g-j_1,  \frac{e+2-j_1}{2}, 0)) \end{equation}  unless $e-j_1 \leq -2$ and $\max( e+1-g-j_1,  \frac{e+2-j_1}{2}, 0)=0$ already, in which case increasing $j_1$ by $1$ and $j_2$ by $d-1$ increases the expression \eqref{general-contained-expression}. Hence the maximum value of \eqref{general-contained-expression} is always attained at either the minimum value of $j_1$ or the maximum value of $j_1$.  For the minimum value, the two constraints are $0 \leq j_1$ and $ 0\leq j_2 = 2g-2+m -e - (d-1) (e-j_1)$, in other words  $de-m-2g+2 \leq (d-1) j_1$. The second bound is always stricter since $m \leq \frac{de}{2}+1$ so \[ de-m-2g+2 \geq \frac{de}{2} -2g+1 > -1 \] since $e \geq 2g-2$ and $d\geq 2$.  Hence the minimum value of $j_1$ is $ \lceil \frac{ de-m-2g+2}{d-1} \rceil $. We can simplify by plugging in $j_1 = \frac{ de-m-2g+2}{d-1} $ which gives a slightly worse bound.

The maximum value of $j_1$ is $e$, which forces $j_2 = 2g-2+m-e$. Hence the maximum over $j_1, j_2$ can be bounded by
\[\max_{ (j_1,j_2) =(  \frac{ de-m-2g+2}{d-1} ,  0) \textrm{ or } (e, 2g-2+m-e) } m+1+j_1+ j_2 + (n+1)   \max( e+1-g-j_1,  \frac{e+2-j_1}{2}, 0) .\]
We can simplify this expression. In the case $j_1 =  \frac{ de-m-2g+2}{d-1} $ , $j_2=0$, we can assume the maximum value of \eqref{general-contained-expression} is attained at the minimum value of $j_1$. Thus $\max( e+1-g-j_1,  \frac{e+2-j_1}{2}, 0) >0$. If in addition $ j_1> \lceil \frac{m}{d-1} \rceil$ then in any case the expression $m+1+j_1 + (n+1)  ( \max( e+1-g-j_1,  \frac{e+2-j_1}{2})) $ is bounded by $m+1+ \lceil \frac{m}{d-1} \rceil + (n+1)  ( e+1 - g- \lceil \frac{m}{d-1} \rceil ) $ since the expression $m+1+j+ (n+1)  ( \max( e+1-g-j_1,  \frac{e+2-j_1}{2})) $ decreases as a function of $j$ and  from \eqref{eminm} we have  $e - \lceil \frac{m}{d-1} \rceil > 2g-2$ so for $j=\lceil \frac{m}{d-1}\rceil $ the maximum is given by $e+1 - g- \lceil \frac{m}{d-1} \rceil$. If we make the opposite assumption that $j_1 \leq   \lceil \frac{m}{d-1} \rceil$, which by  \eqref{eminm} implies $j_1 < e+2-2g$, then we have $\max( e+1-g-j_1,  \frac{e+2-j_1}{2}, 0)= e+1-g-j_1$. These cases give the first two terms in the statement.

In the case $(j_1,j_2)=  (e, 2g-2+m-e)$ we have $\max( e+1-g-j_1,  \frac{e+2-j_1}{2}, 0) =\max (1-g,1,0)=1$ so the value of \eqref{general-contained-expression} is  $2g+2m$. This gives the last term in the statement.\end{proof}

\begin{lemma}\label{general-outside} The dimension of the locus in $ \widetilde{ \operatorname{Sing}}^{F,m}$ where some $c_i\neq 0$ and  the map $C \to  \mathbb P^n \times \mathbb P^n$ is not contained in the graph of $\nabla F$  is at most \[ \max_{ \substack{  0 \leq j_1 \leq e \\ 0 \leq j_2 \leq m+ 2g-2-e  \\   j_3 \leq m \\ j_3 \leq j_2 \\  j_3 \leq (d-1) j_1}}   \dim \operatorname{Mor}'_{e-j_1 , m+2g-2-e-j_2, m-j_3 } (C ,Y) + 2+ j_1 + j_2 .  \]  \end{lemma}

\begin{proof} If the image of the map $C \to  \mathbb P^n \times \mathbb P^n$ is not contained in the graph of $\nabla F$, then it lifts uniquely to a map $f\colon C \to Y$ by the valuative criterion of properness applied to the blowup $Y \to \mathbb P^n \times \mathbb P^n$. The degree of this map is $(e-j_1, 2g-2+m-e-j_2, m-j_3)$ for some integer $j_3$, where $m-j_3$ is the degree of $f^{-1}(E)$, i.e. the length as a scheme of the inverse image of the graph of $\nabla F$ under $C \to  \mathbb P^n \times \mathbb P^n$. Given such a map $f$,  $(a_0,\dots,a_n)$ are determined up to scaling by $f$ and the choice of a divisor $D_1$ of degree $j_1$ on which $a_0,\dots,a_n$ all vanish, and $(c_0,\dots,c_n)$ are determined up to scaling by $f$ and the choice of divisor $D_2$ of degree $j_2$ on which $c_0,\dots,c_n$ all vanish. The scaling factors add $2$ to the dimension. Not all divisors work, as we must have the inverse image of the hyperplane class of the first $\mathbb P^n$ plus $D_1$ agree with the class of $L$, and a similar criterion involving $D_2$, but forgetting this restriction still gives a valid upper bound for the dimension. 

The divisor $Z$ is, in this case, tightly constrained by the map. Let $v$ be a point in the support of $Z$ with uniformizer $\pi$. Let $o_1$ be the multiplicity of $v$ in $D_1$, $o_2$ be the multiplicity of $v$ in $D_2$, $w$ be the multiplicity of $v$ in $f^{-1}(E)$, and $z$ be the multiplicity of $v$ in $Z$. Then we can check that:
\begin{equation}\label{higher-dimensional-local-bound} z \leq w + \min ( (d-1) o_1, o_2) \end{equation}
since $ \frac{\partial F}{\partial x_i}(a_0,\dots,a_n)$ all vanish to order $(d-1)o_1$,  $(c_0,\dots,c_n)$ all vanish to order $o_2$, the first $w$ nonvanishing coefficients in their $\pi$-adic expansions all agree up to a common scalar, and their remaining coefficients do not agree, so in total at most $w + \min ( (d-1) o_1, o_2) $ coefficients agree. Furthermore, the dimension of the space of choices for $\tilde{\alpha}$, restricted to the subscheme $z[v]$, is at most $ \min ( (d-1) o_1, o_2) $ since if $o_2 < (d-1) o_1$ then a choice for $\tilde{\alpha}$ only exists if $z \leq o_2$, in which case the dimension of the space of choices is $z$, and otherwise the equation $\tilde{\alpha} \frac{\partial F}{\partial x_i}(a_0,\dots,a_n) = c_i \mid_Z$ for some $i$ such that $\frac{\partial F}{\partial x_i}(a_0,\dots,a_n) $ vanishes to order exactly $(d-1)o_1$ has a space of solutions of dimension at most $(d-1)o_1$.

In particular, every point in the support of $Z$ must either lie in  $f^{-1}(E)$ or lie in both $D_1$ and $D_2$, so there are certainly at most finitely many possible choices for $Z$.  The dimension of the space of total choices for $\alpha$ is at most $j_1 + j_2$ minus the number of distinct points in the support of $D_1 \cup D_2$, since the local contribution to $\deg D_1 + \deg D_2$ minus the number of points in the support is $o_1 +o_2 -1 \geq \min( (d-1)o_1,o_2)$ unless $o_1=o_2=0$ in which case the local contribution is $0$.  Since the dimension of  the space of pairs of divisors of degrees $j_1,j_2$ supported in a set of size $r$ is at most $r$, the dimension of the space of choices of $D_1,D_2, Z,\alpha$ is at most $j_1+ j_2$, giving that the dimension of the locus in $ \widetilde{ \operatorname{Sing}}^{F,m}$ corresponding to a given map  $f\colon C\to Y $ of multidegree $e-j_1 , m+2g-2-e-j_2, m-j_3 $ is at most $ 2+ j_1 + j_2$.

However, summing \eqref{higher-dimensional-local-bound} over all points gives the constraint \[ m \leq m-j_3 + \min ( (d-1) j_1, j_2) \] which must be satisfied if any $Z$ exists. This constraint implies $j_3 \leq (d-1) j_1$ and $j_3 \leq j_2$. \end{proof}

\begin{proof}[Proof of Proposition \ref{general-dim-bound}] By Lemma \ref{general-dim-lift} it suffices to bound the dimension of $ \widetilde{ \operatorname{Sing}}^{F,m}$. We divide $ \widetilde{ \operatorname{Sing}}^{F,m}$ into three loci, based on whether some $c_i\neq 0$, and, if so, whether the image of the map $C \to  \mathbb P^n \times \mathbb P^n$ is contained in the graph of $\nabla F$,  and its dimension is the maximal dimension of each locus. These dimensions are bounded in Lemmas \ref{general-c0}, \ref{general-contained-comp}, and \ref{general-outside}, with the expression of Lemma \ref{general-contained-comp} simplified in Lemma \ref{general-contained-simp}. Combining the expressions of Lemmas \ref{general-c0}, \ref{general-contained-simp}, and \ref{general-outside}, we obtain the statement of Proposition \ref{general-dim-bound}. \end{proof}

\bibliographystyle{plain}

\bibliography{references.bib}

\begin{thebibliography}{10}

\bibitem{Bourgain2017}
J.~Bourgain.
\newblock On the {V}inogradov mean value.
\newblock {\em Proceedings of the Steklov Institute of Mathematics},
  296(1):30–40, January 2017.
\newblock
  \href{http://dx.doi.org/10.1134/S0081543817010035}{doi:10.1134/S0081543817010035}.

\bibitem{BDG}
Jean Bourgain, Ciprian Demeter, and Larry Guth.
\newblock Proof of the main conjecture in {V}inogradov’s mean value theorem
  for degrees higher than three.
\newblock {\em Annals of Mathematics}, 184(2):633–682, September 2016.
\newblock
  \href{http://dx.doi.org/10.4007/annals.2016.184.2.7}{doi:10.4007/annals.2016.184.2.7}.

\bibitem{Fulton2008}
Wiliam Fulton.
\newblock {\em Algebraic Curves: An Introduction to Algebraic Geometry}.
\newblock William Fulton, 2008.
\newblock available at
  \href{https://dept.math.lsa.umich.edu/~wfulton/CurveBook.pdf}{https://dept.math.lsa.umich.edu/\~wfulton/CurveBook.pdf}.

\bibitem{Grothendieck1961}
Alexander Grothendieck.
\newblock Techniques de construction et th\'{e}or\`{e}mes d'existence en
  g\'{e}om\'{e}trie alg\'{e}brique {IV} : les sch\'{e}mas de {H}ilbert.
\newblock {\em S\'eminaire Nicolas Bourbaki}, 6(221):249--276, 1961.

\bibitem{Hartshorne1977}
Robin Hartshorne.
\newblock {\em Algebraic {{Geometry}}}.
\newblock {Springer}, {New York, NY}, 1st ed. 1977. corr. 8th printing 1997
  edition edition, December 1977.

\bibitem{MHL}
Matthew Hase-Liu.
\newblock A higher genus circle method and an application to geometric
  {M}anin's conjecture.
\newblock \href{https://arxiv.org/abs/2402.10498}{arXiv:2402.10498}, 2024.

\bibitem{Katz1999}
Nicholas Katz.
\newblock Estimates for ``singular" exponential sums.
\newblock {\em International Mathematics Research Notices}, 1999(16):875, 1999.
\newblock
  \href{http://dx.doi.org/10.1155/S1073792899000458}{doi:10.1155/S1073792899000458}.

\bibitem{Katz2001}
Nicholas~M. Katz.
\newblock Sums of {B}etti numbers in arbitrary characteristic.
\newblock {\em Finite Fields and Their Applications}, 7(1):29–44, January
  2001.
\newblock
  \href{http://dx.doi.org/10.1006/ffta.2000.0303}{doi:10.1006/ffta.2000.0303}.

\bibitem{R1974}
R.~M. Kubota.
\newblock {\em Waring’s problem for {$ F_q[x]$}}.
\newblock Instytut Matematyczny Polskiej Akademi Nauk, 1974.

\bibitem{LiuWooley}
Yu-Ru Liu and Trevor~D. Wooley.
\newblock Waring's problem in function fields.
\newblock
  \href{https://staff.math.su.se/shapiro/ProblemSolving/LiuWooley.pdf}{https://staff.math.su.se/shapiro/ProblemSolving/LiuWooley.pdf},
  2005.

\bibitem{Liu2010}
Yu-Ru Liu and Trevor~D. Wooley.
\newblock {W}aring’s problem in function fields.
\newblock {\em Journal f\"{u}r die reine und angewandte Mathematik (Crelles
  Journal)}, 2010(638):1–67, January 2010.
\newblock
  \href{http://dx.doi.org/10.1515/CRELLE.2010.001}{doi:10.1515/CRELLE.2010.001}.

\bibitem{Pugin}
{Pugin, Thibaut}.
\newblock {\em An Algebraic Circle Method}.
\newblock Ph{D} thesis, Columbia University, 2011.
\newblock \href{https://doi.org/10.7916/D8G166VP}{doi:10.7916/D8G166VP}.

\bibitem{Rosen2002}
Michael Rosen.
\newblock {\em Number Theory in Function Fields}.
\newblock Springer New York, 2002.

\bibitem{stacks-project}
The {Stacks project authors}.
\newblock The stacks project.
\newblock \url{https://stacks.math.columbia.edu}, 2026.

\bibitem{Vaughan1986}
R.C. Vaughan.
\newblock On {W}aring's problem for cubes.
\newblock {\em Journal f\"ur die reine und angewandte Mathematik},
  365:122--170, 1986.
\newblock \href{http://eudml.org/doc/152808}{http://eudml.org/doc/152808}.

\bibitem{Vinogradov1934}
I.~M. Vinogradov.
\newblock A new solution of {W}aring’s problem.
\newblock {\em , Dokl. Akad. Nauk SSSR}, 2:337--341, 1934.

\bibitem{Wooley1992}
Trevor~D. Wooley.
\newblock Large improvements in {W}aring’s problem.
\newblock {\em The Annals of Mathematics}, 135(1):131, January 1992.

\bibitem{Wooley2012}
Trevor~D. Wooley.
\newblock The asymptotic formula in {W}aring’s problem.
\newblock {\em International Mathematics Research Notices},
  2012(7):1485–1504, 2012.
\newblock
  \href{http://dx.doi.org/10.1093/imrn/rnr074}{doi:10.1093/imrn/rnr074}.

\bibitem{Wooley2018}
Trevor~D. Wooley.
\newblock Nested efficient congruencing and relatives of {V}inogradov’s mean
  value theorem.
\newblock {\em Proceedings of the London Mathematical Society},
  118(4):942–1016, October 2018.
\newblock \href{http://dx.doi.org/10.1112/plms.12204}{doi:10.1112/plms.12204}.

\bibitem{Yamagishi2016}
Shuntaro Yamagishi.
\newblock The asymptotic formula for {W}aring’s problem in function fields.
\newblock {\em International Mathematics Research Notices},
  2016(23):7137–7178, March 2016.
\newblock
  \href{http://dx.doi.org/10.1093/imrn/rnv392}{doi:10.1093/imrn/rnv392}.

\end{thebibliography}

\appendix

\section{Linear programming for the arbitrary hypersurface case}

\begin{cor}\label{lin-prog} The bound \eqref{wanted-dim-bound} is satisfied for  $\delta>0$ as long as $n> 3d-3+(d-1)\delta $,  $e $ is sufficiently large depending on $n,g,d,\delta$ and for all tuples $i_1,i_2,i_3$ of nonnegative integers such that $i_1 +i_2 \leq i_3 +2g-2$ and $d (i_2+1-2g) \leq (d-2) (i_3-1)$ we have
\begin{equation}\label{i1i2i3-bound}  \dim \operatorname{Mor}'_{i_1,i_2, i_3 } \leq O(1) + 
 \begin{cases}  (n+2-\delta) i_1   +i_2 -3 i_3  & \textrm{if } i_3 \leq \frac{ d i_1}{2}+1  \\
(n+2 - \frac{3d}{2} -\delta)i_1   + i_2 & \textrm{if } i_3 > \frac{ d i_1}{2} +1 \textrm{ and } \frac{(d-2)i_1}{2} +2g-1 \geq i_2 \\ 
i_1 +  \frac{ 2 (n+1-\delta) -2d -2}{d-2}  i_2  & \textrm{if }  \frac{(d-2)i_1}{2} +2g-1 < i_2   \\\end{cases}. \end{equation}  \end{cor}

\begin{proof} In view of Proposition \ref{general-dim-bound}, it suffices to show that the right hand side of Proposition \ref{general-dim-bound} is bounded by the right-hand side of \eqref{wanted-dim-bound}. In other words, we must establish that \begin{equation}\label{bound-to-establish} \dim \operatorname{Mor}'_{e-j_1 , m+2g-2-e-j_2, m-j_3 } (C ,Y) <  (n+1) (e+1-g) -2m -2 - j_1 - j_2 - e \delta +O(1) \end{equation} 
for all $m \in (e+1-2g, \frac{de}{2}+1]$ and $j_1,j_2,j_3$ satisfying $0 \leq j_1 \leq e$, $0 \leq j_2 \leq m+2g-2-e$, $j_3 \leq m$, $j_3 \leq j_2$,  $j_3 \leq (d-1) j_1$ and also that
\begin{equation}\label{bound-to-establish-1} m+ \lceil \frac{m}{d-1} \rceil + (n+1)  ( e+1 - g- \lceil \frac{m}{d-1} \rceil ) <  (n+1) ( e+1-g) - 2m -e \delta  \end{equation}
\begin{equation}\label{bound-to-establish-2} m+ \frac{ de-m-2g+2}{d-1} + (n+1) ( e+1-g- \frac{ de-m-2g+2}{d-1})<  (n+1) ( e+1-g) - 2m -e \delta  \end{equation}
\begin{equation}\label{bound-to-establish-3} 2g-1+2m  <  (n+1) ( e+1-g) - 2m -e \delta  \end{equation}
in each case for $m\in (e+1-2g, \frac{de}{2}+1]$, where we have dropped the $O(1)$ as unnecessary in the last three equations.

We handle these in reverse order. For \eqref{bound-to-establish-3}, the left hand side grows and the right hand side shrinks as $m$ grows, so it suffices to handle the case $m = \frac{de}{2}+1$, where the desired bound is
\[ 2g + 1 + de < (n+1) (e+1-g) - de-2 -e \delta \]
or equivalently
\[2g+3 + (n+1) g + n< (n+1-d - \delta) e \] which is satisfied for $e$ sufficiently large as long as $n+1> d+\delta$, which is weaker than our assumption.

For \eqref{bound-to-establish-2}, we first simplify by cancelling terms, obtaining
\[ 3m + e \delta <  n \frac{ de-m-2g+2}{d-1}. \] 
We again observe that the left-hand side grows and the right-hand side shrinks when $m$ grows, and thus may substitute $m= \frac{de}{2}+1$, giving
\[ 3 \frac{de}{2} +3 + e \delta < n \frac{  \frac{de}{2} -2g+1}{d-1} \]
or equivalently
\[ 3+ n \frac{2g-1}{d-1} < \left( n \frac{d}{2(d-1)} - \frac{3d}{2} -\delta \right) e \]
which is satisfied for $e$ sufficiently large as long as $n \frac{d}{2(d-1)} > \frac{3d}{2}+ \delta$, which is equivalent to $n> 3d-3 + \delta \frac{2 (d-1)}{d}$, which is weaker than our assumption.

For \eqref{bound-to-establish-1}, we first simplify by canceling terms, obtaining
\[3m+ e\delta   < n     \lceil \frac{m}{d-1} \rceil  \] 
for which it suffices to have
\[3m + e \delta < n \frac{m}{d-1}.\]
Since our assumption implies $n> 3d-3$, the right hand side grows faster than the left hand side with $m$, so it suffices to handle the case $m=e+2-2g$, which gives
\[ 3 e +6-6g + e\delta < n \frac{e+2-2g}{d-1} \] or equivalently
\[  6-6g+ n \frac{2g-2}{d-1} < ( \frac{n}{d-1}-3 -\delta) e \]
which is satisfied for $e$ sufficiently large as long as $n> 3d-3 + (d-1) \delta$, which we assumed.

We finally consider the most difficult equation \eqref{bound-to-establish}. We fix $e,m, j_1, j_2,j_3$ satisfying the hypotheses 
\[0 \leq j_1 \leq e, 0 \leq j_2  \leq m+2g-2 -e, j_3 \leq m, j_3 \leq j_2, j_3 \leq (d-1) j_1 , e\geq 0, e+2-2g \leq m, m\leq \frac{de}{2}+1\]  and set $i_1 = e-j_1,  i_2 = m+2g-2-e-j_2, i_3 =m-j_3$. We have 
\[i_3 + 2g-2 - i_1 - i_2 = j_1+ j_2 - j_3 \geq j_1 \geq 0 \] so we always have  $i_1 + i_2 \leq i_3 +2g-2$.

We also have
\[ i_2 + 1-2g = m-e-j_2 -1 \leq m-e-1 \leq \frac{de}{2}+1-e-1= \frac{d-2}{2} e\] \[ = \frac{d-2}{2} (i_3-i_2+2g-2 +j_3 -j_2) \leq \frac{d-2}{2} (i_3-i_2+2g-2) \]
which gives $d (i_2+1-2g) \leq (d-2) (i_3-1)$.

Using the equations $e = i_1+ j_1$,  $m=i_3+j_3$, $ j_2 = m+2g-2-e-i_2= i_3 +j_3+2g-2-i_1-j_1 - i_2 = i_3 -i_1-i_2 +j_3-j_1 + O(1) $ we obtain
\[ (n+1) (e+1-g) -2m -2 - j_1 - j_2 - e \delta = O(1) + (n+1-\delta )  e-4m - j_1 -j_2 \] \[= O(1) + (n+1-\delta) (i_1+j_1) - 2i_3 - 2j_3 - j_1 -i_3+i_1+i_2 - j_3 + j_1 \]
\[ = O(1) + (n+2-\delta) i_1   +i_2 -3 i_3 + (n+1 -\delta)j_1 -3 j_3 \] \[ \geq  O(1) + (n+2-\delta) i_1   +i_2 -3 i_3 + (n+1 -\delta)j_1 -3 (d-1) j_1\]
\[ \geq   O(1) + (n+2-\delta) i_1   +i_2 -3 i_3 \]
where in the last two lines we use that $j_3 \leq (d-1) j_1$ and $j_1 \geq 0$ while $n+1-\delta >3 (d-1)$, which is weaker than our assumption. This verifies \eqref{bound-to-establish} in the first case of \eqref{i1i2i3-bound}.

Next, using $j_2 = m+2g-2-e-i_2 = m-e-i_2 +O(1)$ and $m \leq \frac{de}{2}+ 1$ and $e=i_1+j_1$ we obtain
\[ (n+1) (e+1-g) -2m -2 - j_1 - j_2 - e \delta = O(1) + (n+1-\delta )  e-2m - j_1 -j_2 = O(1) + (n+2-\delta)e - 3m - j_1 + i_2\]
\[ \geq O(1)+ (n+2 - \frac{3d}{2} -\delta)e  - j_1 + i_2 = O(1)+(n+2 - \frac{3d}{2} -\delta)i_1 + (n+1 - \frac{3d}{2} -\delta)j_1  + i_2 \]\[ \geq  O(1)+(n+2 - \frac{3d}{2} -\delta)i_1   + i_2\]
since $j_1 \geq 0$  and $n+1 \geq \frac{3d}{2} + \delta$. This verifies \eqref{bound-to-establish} in the second case of \eqref{i1i2i3-bound}. 

In the third case, we set $\tilde{m} =  \frac{de}{2}+1- m$ so that we have $\tilde{m}\geq 0$ and observe that 
\[i_2 + j_2+1-2g = m-1 -e = \frac{de}{2} - \tilde{m} - e = \frac{ (d-2)e}{2} - \tilde{m} \]
so that
\[ e =\frac{2}{d-2} ( i_2+j_2+\tilde{m}+1-2g) = \frac{2}{d-2} (i_2+j_2+ \tilde{m})+O(1)\]
and thus \[ j_1 = e -i_1 =  \frac{2}{d-2} (i_2+j_2+ \tilde{m}) -i_1 +O(1)\]
which gives \[ (n+1) (e+1-g) -2m -2 - j_1 - j_2 - e \delta = O(1) + (n+1-\delta )  e-2m - j_1 -j_2  \]
\[ = O(1) + \frac{2}{d-2}  (n+1-\delta)  (i_2+j_2+ \tilde{m}) - \frac{2d}{d-2} (i_2+j_2+ \tilde{m}) + 2 \tilde{m}  - \frac{2}{d-2} (i_2+j_2+ \tilde{m})+ i_1 - j_2 \]
\[ = O(1) +i_1 +  \frac{ 2 (n+1-\delta) -2d -2}{d-2}  i_2 +   \left(  \frac{ 2 (n+1-\delta) -2d -2}{d-2}  -1 \right) j_2 +  \left( \frac{ 2 (n+1-\delta) -2d -2}{d-2} +2 \right) \tilde{m} .\]

We have $\frac{2 (n+1-\delta) -2d -2}{d-2}  -1>0$ since $2 (n+1-\delta)> 2d + 2 + (d-2)= 3d$ as this is weaker than our assumption. Thus we also have $\frac{ 2 (n+1-\delta) -2d -2}{d-2} +2 >0$. Since $j_2\geq 0$ and $\tilde{m} \geq 0$, these terms may be dropped, and we obtain
\[ (n+1) (e+1-g) -2m -2 - j_1 - j_2 - e \delta \geq O(1) +i_1 +  \frac{ 2 (n+1-\delta) -2d -2}{d-2}  i_2 .\]
This verifies \eqref{bound-to-establish} in the third case of \eqref{i1i2i3-bound}.\end{proof}

One can further check that \eqref{i1i2i3-bound} is sharp, in the sense that for each triple of nonnegative integers $i_1,i_2,i_3$ with $i_1+i_2 \leq i_3 +2g-2$ there exist values of $m,e,j_1,j_2,j_3$ satisfying all the inequalities where the needed dimension bound in \eqref{bound-to-establish} in fact equals the assumed dimension bound in \eqref{i1i2i3-bound}. We never need to check the inequalities $j_1 \leq e, j_2 \leq m+2g-2-e, j_3\leq m$ as these are equivalent to $i_1,i_2,i_3\geq 0$ which are assumed anyways. Furthermore $e+2-2g \leq m$ follows from $0 \leq j_2 \leq m+2g-2-e$ and so does not need to be checked separately, and the same is true for $e\geq 0$ following from $0\leq j_1 \leq e$.  The remaining inequalities that need to be checked are $0\leq j_1, 0\leq j_2, j_3 \leq j_2, j_3\leq (d-1) j_1, $ and $ m \leq \frac{de}{2}+1$.

In the first case, we take $e=i_1$, $m=i_3$, $j_1=0$, $j_2 = i_3+2g-2-i_1-i_2$, $j_3=0$. The assumption on $i_1,i_2,i_3$ implies $j_2\geq 0$, from which it is easy to see that all the inequalities are satisfied except possibly  $m\leq \frac{de}{2}+1$. The inequality $m\leq \frac{de}{2}+1$ expands to $ i_3 \leq \frac{ d i_1}{2}+1 $ which is assumed in the first case.

In the second case, we take $e= i_1$, $m = \frac{de}{2}+1 =\frac{d i_1}{2}+1$, $j_1=0$, $j_2 = \frac{ (d-2) i_1}{2} + 2g-1- i_2$, $j_3= \frac{d i_1}{2}+1-i_3$. The assumption $ i_3 > \frac{ d i_1}{2}+1 $ implies $j_3<0$ and the assumption $\frac{(d-2)i_1}{2} +2g-1 \geq i_2 $ implies $j_2 \geq 0$.  This verifies all the inequalities. 

In the third case, we take $e=\frac{2}{d-2} (i_2+1-2g)$, $m= \frac{de}{2}+1 = \frac{d}{d-2} (i_2+1-2g) +1$, $j_1=\frac{2}{d-2} (i_2+1-2g)-i_1$,  $j_2=0$,  $j_3= \frac{d}{d-2} (i_2+1-2g) +1-i_3$.   The assumption $\frac{(d-2)i_1}{2} +2g-1 < i_2$ gives $j_1>0$.  The assumption $d (i_2+1-2g) \leq (d-2) (i_3-1)$  gives $j_3\leq 0$. This verifies all the inequalities.

%
%
%
%
%
%
%
%

\end{document}